\documentclass[11pt,reqno]{amsart}
\usepackage{times}
\usepackage{amsmath,amsfonts,amstext,amssymb,amsbsy,amsopn,amsthm,eucal}
\usepackage{txfonts}
\usepackage{dsfont}
\usepackage{graphicx}   % for figures
\usepackage{hyperref}
\usepackage{color}
\usepackage{verbatim}
\usepackage[all]{xy}

\numberwithin{equation}{section}

\hypersetup{
  pdftitle={Energy Identity for Stationary Yang Mills},
  pdfauthor={Aaron Naber and Daniele Valtorta},
  pdfsubject={},
  pdfkeywords={},
  pdfpagelayout=SinglePage,
  pdfpagemode=UseOutlines,
  colorlinks,
  bookmarksopen,
  linkcolor=[rgb]{0,0,0.7},
  urlcolor=[rgb]{0,0,0.4},
  citecolor=[rgb]{0.4,0.1,0}
}

\newcommand{\SO}{\text{SO}}

\newcommand{\Lip}{\text{Lip}}

\newcommand{\Vol}{\text{Vol}}

\newcommand{\inj}{\text{inj}}

\newcommand{\bb}{\bf b}

\newcommand{\dN}{\mathds{N}}

\newcommand{\dR}{\mathds{R}}

\newcommand{\cA}{\mathcal{A}}
\newcommand{\cB}{\mathcal{B}}

\newcommand{\cC}{\mathcal{C}}

\newcommand{\cF}{\mathcal{F}}
\newcommand{\cG}{\mathcal{G}}

\newcommand{\cL}{\mathcal{L}}

\newcommand{\cM}{\mathcal{M}}

\newcommand{\cS}{\mathcal{S}}

\newcommand{\cW}{\mathcal{W}}

\newcommand{\ton}[1]{\left(#1\right)}

\newcommand{\cur}[1]{\left\{#1\right\}}
\newcommand{\abs}[1]{\left|#1\right|}

\newcommand{\B}[2]{B_{#1}\ton{#2}}

\newcommand{\supp}[1]{\operatorname{supp}\ton{#1}}

\renewcommand{\paragraph}[1]{\ \newline \ \textbf{#1\ }}

\newtheorem{theorem}{Theorem}[section]

\newtheorem{proposition}[theorem]{Proposition}
\newtheorem{lemma}[theorem]{Lemma}
\newtheorem{corollary}[theorem]{Corollary}
\theoremstyle{definition}
\newtheorem{definition}[theorem]{Definition}
\theoremstyle{remark}
\newtheorem{remark}{Remark}[section]
\theoremstyle{remark}

\theoremstyle{remark}

\theoremstyle{remark}

\theoremstyle{remark}

\begin{document}

\title{Energy Identity for Stationary Yang Mills}

%\title{Quantitative Stratification, $W^{1,p}$-Reifenberg, and the\\ Regularity of Stationary and Minimizing Harmonic Maps}

\author{Aaron Naber and Daniele Valtorta}\thanks{The first author has been supported by NSF grant DMS-1406259, the second author has been supported by SNSF grant 200021\_159403/1}

%%\date{\today}
\date{\today}

\begin{abstract}
Given a principal bundle $P\to M$ over a Riemannian manifold with compact structure group $G$, let us consider a stationary Yang-Mills connection $A$ with energy $\int_M |F_A|^2\le \Lambda$.  If we consider a sequence of such connections $A_i$, then it is understood by \cite{Tian_CalYM} that up to subsequence we can converge $A_i\to A$ to a singular limit connection such that the energy measures converge $|F_{A_i}|^2 dv_g\to |F_A|^2dv_g +\nu$, where $\nu=e(x)d\lambda^{n-4}$ is the $n-4$ rectifiable defect measure.  Our main result is to show, without additional assumptions, that for $n-4$ a.e. point the energy density $e(x)$ may be computed explicitly as the sum of the bubble energies arising from blow ups at $x$.  Each of these bubbles may be realized as a Yang Mills connection over $S^4$ itself.

This energy quantization was proved in \cite{Riv_YM} assuming a uniform $L^1$ hessian bound on the curvatures in the sequence.  In fact, our second main theorem is to show this hessian bound holds automatically.  Precisely, given a connection $A$ as above we have the apriori estimate $\int_M |\nabla^2 F_A| < C(\Lambda,\dim G,M)$ for the curvature.  It is important to note this result is proved in tandem with the energy quantization, and not before it.  Indeed, we will in fact prove an effective version of the energy identity, and it is this effective version which will lead to both the $L^1$ hessian bound and the classical energy quantization results.  In the course of the proof we will provide a quantitative version of the bubble tree decomposition which hold in all dimensions with effective estimates for a fixed stationary connections.  To produce to strongest estimates in the paper we introduce an $\epsilon$-gauge condition, which generalizes the usual Coulomb gauge and which will exist, with effective control, even over singular regions.  On these $\epsilon$-gauges we will provide a new superconvexity estimate which will be a key tool in analyzing higher dimensional annular regions.
\end{abstract}

\maketitle

\tableofcontents

\section{Introduction}

This paper is focused on studying principle bundles $P\stackrel{G}{\to} M$ over Riemannian manifolds $(M^n,g)$ with compact structure groups $G\subseteq \SO(k)$.  A Yang-Mills connection $A$ on $P$ is a critical point of the Yang-Mills $L^2$ curvature functional 
\begin{align}\label{e:intro:YM_functional}
	\cF[A]\equiv \int |F_A|^2\, .
\end{align}
Most of the results of this paper are local in nature, and therefore it will be sufficient to consider connections over some ball $B_2 \subseteq M$.  For technical simplicity we will restrict ourselves to smooth critical points of \eqref{e:intro:YM_functional}, however all of our techniques generalize to stationary points in more general singular contexts (e.g. admissible connections in the sense of \cite{TaoTian_YM} or stationary connections in the sense of \cite{RiPe},\cite{RiPe2}), and we will make comments on the necessary ingredients to make such generalizations, which are fairly straight forward.  \\

It is often times the case one is interested in not just a fixed connection, but in sequences and limits of such connections.  For instance, when studying moduli spaces \cite{Donaldson_Compactification},\cite{DoTh_gauge} of connections, or when considering contradiction arguments.  This is the appropriate context to study the energy identity.  Indeed, if $A_i$ are a sequence of Yang-Mills connections satisfying the uniform energy bound $\int |F_A|^2\leq \Lambda$, then after passing to a subsequence it is known by \cite{Uhl_Rem},\cite{TaoTian_YM},\cite{RiPe3} that we can converge the $A_i$ modulo gauge
\begin{align}
A_i\to A\, ,	
\end{align}
to a connection $A$ which is smooth away from a set of $n-4$ measure zero.  More than that, by \cite{Tian_CalYM} we can limit the energy measures
\begin{align}
|F_{A_i}|^2 dv_g \to |F_A|^2dv_g	 + \nu = |F_A|^2dv_g	 + e(x)\,d\lambda^{n-4}\big|_{S}\, ,
\end{align}
where $\nu = e(x)\,d\lambda^{n-4}$ is the $n-4$ rectifiable defect measure supported on $S=\text{supp}[\nu]$.\\

The goal of this paper is then two fold.  We wish to study better the regularity properties of stationary Yang-Mills connections, and we wish to understand better the defect measure $e(x)d\lambda^{n-4}$ which arises as the singular part of the limit of energy measures from a sequence of Yang-Mills connections.  \\

To state accurately our results on the structure of $e(x)$ let us first define with some accuracy the notion of a bubble.  Notationally, let us remark that if $A$ is a connection on the pointed manifold $(M,g,x)$, then we write $r^{-1}A$ to denote the induced connection on $(M,r^{-2}g,x)$.  This has the effect of rescaling the ball $B_{r}(x)\to B_1(x)$ to unit size, and therefore {\it blows up} $A$ at $x$ at scale $r$.  Given this we define the notion of bubbling:

\begin{definition}\label{d:bubble}
We define the following:
\begin{enumerate}
\item A bubble $B$ is a smooth Yang Mills connection on $\dR^n\times G$ which is invariant under translation with respect to some $n-4$ subspace $\cL_B\subseteq \dR^n$.  We define the energy of $B$ to be $E[B]\equiv \int_{\cL^\perp_B} |F_B|^2$.
\item We say that $B$ is a bubble at $x\in \text{supp}\{\nu\} $ if there exists a sequence $x_i\to x$ and $r_i\to 0$ such that the blow ups converge $r_i^{-1}A_i\to B$.  We denote by $\cB[x]$ the collection of all bubbles at $x$.
\end{enumerate}
\end{definition}
%\begin{remark}
%See Section \ref{} for a discussion of convergence.	
%\end{remark}

Notice that if $B$ is a bubble, then by restricting $B$ to $\cL_B^\perp$ and transforming conformally we may view $B$ as a smooth Yang Mills connection on $S^4$.  The bubbles at a point $x\in \text{supp}[\nu]$ turn out to be related to energy density $e(x)$.  It turns out that if $B_1,\ldots, B_k\in \cB[x]$ are {\it distinct} bubbles at $x$ then one can rather easily prove the inequality
\begin{align}
e(x)\geq \sum_{B_j} E[B_j] = \sum_{B_j} \int_{\cL^\perp_{B_j}}|F_{B_j}|^2 \, .
\end{align}
In words, the energy density $e(x)$ is at least as large as the energy contribution of every bubble at $x$.  It has been an open problem about whether this inequality is an equality.  This was first considered and proved for four dimensional instantons.  In higher dimensions, the first results were due to \cite{Tian_CalYM}, where for generalized instantons it was shown that for $n-4$ a.e. point that $x\in \text{supp}[\nu]$ we can indeed compute $e(x)$ explicitly by the energy identity
\begin{align}\label{e:intro:energy_identity}
	e(x) = \sum_{B_j\in\cB[x]} \int_{\cL^\perp_{B_j}}|F_{B_j}|^2\, ,
\end{align}
for a collection of distinct bubbles.  There is a variety of work in the literature toward dropping the instanton assumption.  In dimension four the energy identity \eqref{e:intro:energy_identity} was proved in full generality in \cite{Riv_YM}.  In higher dimensions, the best result is by Riviere \cite{Riv_YM}, where it was shown that \eqref{e:intro:energy_identity} holds if one additionally assumes a uniform $L^1$ bound on the hessian of the curvature.  The idea of \cite{Riv_YM} was in the spirit of \cite{LinRiv_HarEQ}, where a similar result was proved for harmonic maps, and exploited certain Lorentz space estimates.  The first main result of this paper is to prove \eqref{e:intro:energy_identity} in full generality, and in particular we drop the assumed $L^1$ hessian bound.\\

In fact, the second main result of this paper is to show that if $A$ is a stationary Yang-Mills connection, then one does in fact automatically have the apriori $L^1$ hessian estimate
\begin{align}
	\int_M |\nabla^2 F_A| < C(\Lambda,k,M)\, ,
\end{align}
where the dependence of $C$ above on $M$ is only on the $C^2$ geometry of $M$.  It is worth noting that this result is proved in tandem with the energy identity of \eqref{e:intro:energy_identity}, and not before it.  In fact, as the paper is arranged we shall prove the energy identity first.  However, fundamentally both results will follow from an effective version of the energy identity, which will describe the breakup of a fixed stationary connection.  This will be outlined in Section \ref{ss:outline_proof} and is described in rigor in Section \ref{s:quant_bubbletree}.\\

\subsection{Main Result for Stationary Yang Mills Connections}\label{ss:main_results_stationary}

Our regularity results are all local in nature, and therefore we will only ever consider connections on open balls in manifolds.  Precisely, we consider in this subsection a Yang-Mills connection $A$ living on a principle $G$-bundle $P\to B_2(p)\subseteq M$, with $G\subseteq \SO(k)$ compact, which satisfies
\begin{align}\label{e:YM_assumptions}
	&|\sec| < K^2\, ,\notag\\
	&\inj(p)>K^{-1}\, .
\end{align}

Our main result on the regularity of stationary Yang Mills connections is the following:\\

\begin{theorem}[$L^1$ Hessian Estimate]\label{t:main_L1_hessian}
	Let $A$ be a stationary Yang-Mills connection satisfying \eqref{e:YM_assumptions} and $\fint_{B_2} |F_A|^2 \leq \Lambda$.  Then we have
	\begin{align}
		\fint_{B_1} |\nabla^2 F_A| < C(n,k,K,\Lambda)\, .	
	\end{align}
\end{theorem}
\begin{remark}
	In fact one can easily see that the estimate does not depend on the lower injectivity radius bound by simply lifting to a local cover.
\end{remark}

\vspace{.25cm}

\subsection{Main Results for Weak Limits of Yang Mills}

Let us now discuss limits $A_i\to A$ of stationary Yang-Mills connections with uniformly bounded energy.  In this case we consider the defect measure
\begin{align}
|F_{A_i}|^2 dv_g \to |F_A|^2 dv_g + \nu\, ,	
\end{align}
where by \cite{Tian_CalYM} we have that the defect measure $\nu=e(x) \lambda^{n-4}$ is $n-4$ rectifiable.  Recall from Definition \ref{d:bubble} the precise meaning of a bubble at a point $x$.  Our main result is the following energy quantization, which tells us that we may compute the energy density $e(x)$ through the bubbles at $x$:\\

\begin{theorem}[Energy Identity]\label{t:main_energy_quantization}
	Let $A_i\to A$ be a limit of stationary Yang-Mills connections satisfying \eqref{e:YM_assumptions} and $\fint_{B_2} |F_{A_i}|^2 \leq \Lambda$, and let $\nu$ be the associated defect measure.  Then for $n-4$ a.e. $x\in \text{supp}\{\nu\}$, there exists a finite collection of distinct bubbles $B_1,\ldots,B_k\in\cB[x]$ such that
	\begin{align}
		e(x) = \sum_{B_j} E[B_j] = \sum_{B_j} \int_{\cL^\perp_{B_j}}|F_{B_j}|^2\, .	
		\end{align}
\end{theorem}
\vspace{.25cm}

\subsection{Outline of Proofs and Techniques}\label{ss:outline_proof}
 
Let us now outline the paper and the techniques involved in the proofs of Theorem \ref{t:main_L1_hessian} and Theorem \ref{t:main_energy_quantization}.  Everything in this subsection is rough in nature, and is meant to convey intuition without being dragged down by the large number of necessary technical details needed for the rigorous statements. \\

To begin with some basics, by covering our ball $B_1(p)$ with balls of small radius we can always assume our bound $K<<1$ from \eqref{e:YM_assumptions} is very small.  In particular, by writing in harmonic coordinates we can assume we are on a chart so that our ball is a Euclidean ball with metric $g_{ij}$ which satisfies for all $\alpha<1$
\begin{align}\label{e:YM_assumptions2}
||g_{ij}-\delta_{ij}||_{C^{1,\alpha}}\leq C(n,\alpha)K\, .	
\end{align}
The $K$ in the above is not quite the same as the one in \eqref{e:YM_assumptions}, however it is does tend to zero as the sectional curvature does.  In fact, there is really little lost in just assuming we are working on Euclidean space itself, as the proof of the general case requires only some minor extra technical work of a non fundamental nature.  We will do just that for the remainder of this outline.\\

The proof of the main Theorems will center around the two main decomposition theorems given as Theorem \ref{t:quantitative_bubbletree} and Theorem \ref{t:bubble_decomposition}.  The content of the quantitative annulus/bubble decomposition of Theorem \ref{t:bubble_decomposition} is to split a ball $B_1(p)$ into two primary types of pieces
\begin{align}\label{e:outline:covering}
B_1(p)\subseteq \bigcup_a \cA_a \cup \bigcup_b \cB_b\, ,	
\end{align}
where $\cB_b\subseteq B_{r_b}(x_b)$ and $\cA_a\subseteq B_{r_a}(x_a)$ are quantitative versions of bubble and annular regions, such that we have the quantitative content covering control
\begin{align}\label{e:outline:content}
	\sum r_a^{n-4} + \sum r_b^{n-4}\leq C\, .
\end{align}
This annulus/bubble decomposition will be the primary decomposition used toward the proof of the $L^1$ hessian estimate.  For the energy identity we will rely on the quantitative bubble tree decomposition of Theorem \ref{t:quantitative_bubbletree}, which we will describe after we have discussed the quantitative bubble and annular regions in more detail. \\

The quantitative bubble regions $\cB_b\subseteq B_{r_b}(x_b)$ are relatively easy to analyze, however the quantitative annular regions $\cA_a$ will require several new ideas and will take up the majority of our discussion.  Let us begin with a brief discussion of our goals with both, and then we will turn our attention to the methods.

Recall that a bubble, see Definition \ref{d:bubble}, is a smooth solution to the Yang-Mills equation on $\dR^n$ which is invariant under translation by some $n-4$ dimensional subspace $\cL\subseteq \dR^n$.  Likewise, a quantitative bubble $\cB_b$ should be a solution which is close to looking like a bubble in an appropriate sense.  The actual definition, given in Section \ref{s:bubble}, requires a little work because one has to account for the possibility for bubbles inside of bubbles, however for our outline the main point to emphasize is that it follows directly from the definition that the quantitative bubbles $\cB_b$ are regions which are nearly invariant by $\cL$ and have uniformly bounded curvature:  
\begin{align}
	r_b^{4-n}\int_{B_{r_b}}|F[\cL]|^2 < \delta\, ,\notag\\
	r_b^2 |F| \leq C \text{ on }\cB_b\, .
\end{align}
Using elliptic estimates one can then obtain pointwise scale invariant hessian estimates $r_b^4|\nabla^2 F_A|<C$ on the curvature, which in particular lead to the integral estimates
\begin{align}\label{e:outline:bubble_estimate}
\int_{\cB_b} |\nabla^2 F_A| \leq C r_b^{n-4}\, .	
\end{align}
Since the above are straightforward we will take it all in a blackbox in this outline and refer to Section \ref{s:bubble} for a more detailed description of the quantitative bubble regions.\\

To discuss quantitative annular regions let us begin with a review of $\delta$-flat and $\delta$-weakly flat balls.  Namely, we say a ball $B_r(x)$ is $\delta$-flat if one has the scale invariant curvature estimate $r^2 |F_A|<\delta$ in $B_r(x)$.  This is essentially the strongest condition one might ask for on a ball, and in fact too strong for practical applications.  A weaker condition is that of a $\delta$-weakly flat ball.  As in Definition \ref{d:weak_flat} we say that $B_r(x)$ is $\delta$-weakly flat if for some $n-4$ subspace $\cL\subseteq B_r(x)$ we have the scale invariant estimate $d(x,\cL)^2 |F_A|<\delta$ on $A_{\delta r,r}(\cL)\cap B_r(x)$.  That is, $B_r(x)$ is $\delta$-weakly flat if the curvature is small away from $\cL$, however it is still quite possible to have large curvature concentration near $\cL$ itself.  This situation happens frequently, and is in fact typical when studying defect measures.  One could attempt to cover $B_1(p)$ as in \eqref{e:outline:covering} by quantitative bubble regions and weakly flat balls, unfortunately such a covering cannot necessarily be built to satisfy the content estimate \eqref{e:outline:content}, which will be crucial.  

%Consider for example a radially symmetric instanton in $\R^4$ which becomes flat at infinity. Since in dimension $4$ Yang Mills connections are conformally invariant, by translation and scaling we can obtain a connection such that $\abs{x}^{-2}\abs{F_A(x)}\leq \delta$ for all $\abs x\in [r,1]$, where $r$ is arbitrary (note that by taking a sequence of these connections with $r_i\to 0$, in the limit we would obtain a flat connection on $\B 1 0 \setminus \{0\}$ plus a defect measure at the origin). 

%By taking $0<r<<\delta^{N+1}$ sufficiently small, it is clear that for all $s\in [\delta^N,\delta]$, the balls $\B s 0$ are $\delta$-weakly flat. Moreove, on each of the annuli $A_s=\B {s}{0}\setminus \B {\delta s }{0}$, we have
%\begin{gather}\label{eq_stupid}
% \int_{A_s} \frac{\abs{F}}{\abs x^2} + \int_{A_s} \abs{\nabla^2 F} \leq C\, .
%\end{gather}
%However, if we cover $\B 1 0 $ by
%\begin{gather}
% \B 1 0 \subset \B {\delta^{N+1}}{0} \cup \bigcup_{i=0}^{N} \B {\delta^i}{0}\setminus \B {\delta^{i+1}}{0}\, ,
%\end{gather}
%we will clearly loose any uniform estimates on the number of pieces in this decomposition, and thus we cannot use \eqref{eq_stupid} to obtain a uniform $L^1$ bound on the hessian gloally. In order to do so, we need to cover the whole region $\B {1}{0}\setminus \B {\delta^N}{0}$ by a single annulus, and prove estimates on this annulus which are independent on $N$. To accomplish this we extend the notion of a $\delta$-weakly flat ball to that of an annular region.\\

A $\delta$-annular region $\cA\subseteq B_{2r}$ is a region which looks $\delta$-weakly flat on many scales.  Slightly more precisely, if $\cC$ is a closed set of center points and $r_x:\cC\to \dR$ is a positive function then
\begin{align}
\cA\equiv B_{2r}\setminus \overline B_{r_x}(\cC) \equiv B_{2r}\setminus \bigcup_{x\in\cC} \overline B_{r_x}(x)\, .	
\end{align}
There are variety of useful technical conditions given in Definition \ref{d:annulus} in the definition of a $\delta$-annular region, e.g. a Vitali condition on the balls, however the relevant assumptions to keep in mind is that for each center point $x\in \cC$ and all $r_x<s<2r$ we have that $B_s(x)$ is $\delta$-weakly flat and that $\cC\cap B_s(x)$ looks approximately like the subspace $\cL_x = x+\cL$.  Thus as claimed we have that annular regions are those which look weakly flat on a potentially arbitrary number of scales.  We also have that $\cC$ looks approximately like an $n-4$ dimensional  space, and it is convenient to define the packing measure $\mu = \sum r_x^{n-4}$ associated to it.  In Theorem \ref{t:annular_region} we prove a structure theorem for annular regions, which from the analysis point of view is the most important in the paper.  The main results of this structure theorem are the following:
\begin{align}\label{e:outline:annular_estimate}
	&c s^{n-4} < \mu(B_s(x)) < C s^{n-4}\, \text{ for }r_x<s<2r\, ,\notag\\
	&r^{4-n}\int_{\cA} |F|^2 < \epsilon\, ,\notag\\
	&r^{4-n}\int_{\cA} |\nabla^2 F| < \epsilon\, .
\end{align}  
The first result is an Ahlfor's regularity type result on $\mu$, which tells us that $\cC$ approximates an $n-4$ dimensional space in a strong sense.  For the energy identity of Theorem \ref{t:main_energy_quantization} it is the second estimate above which plays the key role, while for the $L^1$ hessian estimate of Theorem \ref{t:main_L1_hessian} is it the third estimate which plays the important role.  Let us first roughly see how to conclude the main theorems once \eqref{e:outline:annular_estimate} is known, and then the rest of the outline will focus on the proof of \eqref{e:outline:annular_estimate} itself.  \\

The proof of the $L^1$ hessian estimate of Theorem \ref{t:main_L1_hessian} is now nothing more than a combination of the covering \eqref{e:outline:covering} with the content estimate \eqref{e:outline:content} and the scale invariant integral estimates \eqref{e:outline:bubble_estimate}, \eqref{e:outline:annular_estimate}.  Indeed:
\begin{align}
\int_{B_1} |\nabla^2 F|\leq \sum_a \int_{\cA_a}	|\nabla^2 F| + \sum_b \int_{\cB_b} |\nabla^2 F| \leq C\Big(\sum_a r_a^{n-4} + \sum_b r_b^{n-4}\Big) \leq C\, .
\end{align}

To discuss the energy identity of Theorem \ref{t:main_energy_quantization} we describe a quantitative version, which is given in Theorem \ref{t:quantitative_eq}.  To accomplish this we need the quantitative bubble tree decomposition of Theorem \ref{t:quantitative_bubbletree}, which is a refinement of the decomposition of \eqref{e:outline:covering} (though in fact we prove the following decomposition first).  Under the assumption that $B_1(p)$ is weakly flat with respect to $\cL=\cL^{n-4}\subseteq B_1$ we will decompose the ball
\begin{align}\label{e:outline:covering2}
B_1(p)\subseteq \bigcup_a \cA_a \cup \bigcup_b \cB_b\cup \bigcup_c B_{r_c}(x_c)\, ,	
\end{align}
where in addition to the content estimate \eqref{e:outline:content} for the $\delta$-annular and $\delta$-bubble regions we also have the small content estimate $\sum_c r_c^{n-4}<\epsilon$.  The key difference between the two decompositions is that by assuming $B_1$ is weakly flat and throwing out this set of small content we can assume every annular region $\cA_a$ and every bubble region $\cB_b$ are with respect to the {\it same} $n-4$ plane $\cL$.  This does not hold for the decomposition in \eqref{e:outline:covering}.  More than that, the quantitative energy identity of Theorem \ref{t:quantitative_eq} gives us that for each $q\in \cL$ such that the $4$-plane $\cL^\perp_q\equiv q+\cL^\perp$ satisfies $\cL^\perp_q\cap \bigcup_c B_{r_c}=\emptyset$, then if $\cB_q \equiv \cL^\perp_q\cap \bigcup_b \cB_b$ are the bubble regions which intersect the slice $\cL^\perp_q$ we have the following:
\begin{align}\label{e:outline:quant_ei}
	&\#\{\cA_a\cap \cL^\perp_q \neq \emptyset\}\, ,\; \#\{\cB_b\cap \cL^\perp_q \neq \emptyset\} \leq N=N(n,k,\Lambda)\, ,\notag\\
	&\big|\int_{B_1(q)}|F_A|^2 - \int_{\cB_q}|F_A|^2 \big|<\epsilon\, .
\end{align}
The second estimate above is primarily due to the first estimate in \eqref{e:outline:annular_estimate}, so that most slices in an annular region have small energy, and that the constant $N$ above is independent of $\delta$, so that a typical slice only intersects a bounded number of small energy regions.  The $\delta$-independent bound on $N$ takes a bit of work, and morally follows because the covering is built to satisfy the correct nontriviality assumptions.  More specifically, each time an annular region is intersected by $\cL^\perp_q$ there is a corresponding bubble region which also gets intersected, and further each such bubble region contains some definite amount of energy, so that this may happen only a uniformly bounded number of times.  To conclude the classical energy identity we study the $n-4$ rectifiable defect measure $\nu$ of a limiting sequence and observe that a.e. there exists a tangent measure and it is a multiple of the Hausdorff measure on some $\cL$.  In particular, by blowing up at such points we get that neighborhoods are arbitrarily weakly flat, and so we may apply \eqref{e:outline:quant_ei} with $\epsilon\to 0$.  This will conclude the classical energy identity.  See Section \ref{s:energy_identity} for more on this.\\

What is then left in our outline is to describe the proof of the structure theorem estimates of \eqref{e:outline:annular_estimate} for annular regions.  For $n=4$ there are several known ways of doing such estimates, unfortunately all such previous methods break down in higher dimensions.  To describe our methods, which take place over Sections \ref{s:annular}, \ref{s:eps_gauge}, and \ref{s:eps_gauge_annular_estimates} and are related to the arguments of \cite{JiNa}, we begin with a discussion of gauges.  The choice of a good gauge is often crucial in regularity issues, and in the context of Yang Mills the standard gauge condition is a Coulomb gauge.  Unfortunately, the existence of a Coulomb gauge is typically only over balls which are geometrically very simple.  For instance, if $B_r(x)$ is $\delta$-flat as above, then it is easy to prove the existence of a Coulomb gauge on $B_r(x)$.  However, if $B_r(x)$ is only $\delta$-weakly flat, then a Coulomb gauge simply need not exist on the ball.\\

Instead, in Section \ref{s:eps_gauge} we describe the notion of a harmonic $\epsilon$-gauge, which in Theorem \ref{t:eps_gauge_existence} we show does exist on weakly flat balls.  To describe it recall that our structure group $G\subseteq \SO(k)$ so that we have an induced vector bundle $E\to B_1$.  In Definition \ref{d:eps_gauge} we say that sections $V^1,\ldots,V^k\in\Gamma(B_r,E)$ form a harmonic $\epsilon$-gauge on $B_r(x)$ if the following hold:
\begin{align}\label{e:outline:eps_gauge}
	&\Delta V^a = 0\, ,\notag\\
	&|V^a|\leq 1+\epsilon\, ,\;\;\;\fint_{B_r}|\langle V^a,V^b\rangle - \delta^{ab}|<\epsilon\, ,\;\;\; r^2\fint_{B_r}|\nabla V|^2<\epsilon^2\, .
\end{align}
We call the sections an $\epsilon$-gauge because the condition $\fint_{B_r}|\langle V^a,V^b\rangle - \delta^{ab}|<\epsilon$ only guarantees that away from a set of small measure that the $V^a$ form an $\epsilon$-orthonormal basis.  On a $\delta$-flat ball it is not so hard to see that the $V^a$ form an actual $\epsilon$-orthonormal basis at every point, however if the ball is only $\delta$-weakly flat ball then this need not hold.\\

Section \ref{s:eps_gauge} and in particular Section \ref{s:eps_gauge_annular_estimates} are dedicated to proving much more powerful estimates on $\epsilon$-gauges than those in \eqref{e:outline:eps_gauge}.  Before describing these, let us outline how the $V^a$ may be used to control the curvature on annular regions.  Indeed, using the definition of curvature one may compute at any point the bounds
\begin{align}
|F(V)|\leq 2|\nabla^2 V|\, ,\;\;\; |\nabla^2 F(V)|\leq 2|\nabla^4 V|+2|\nabla F|\,|\nabla V| + |F|\,|\nabla^2 V|\, .	
\end{align}
Imagine now that we are at a point $x$ such that $|\langle V^a,V^b\rangle - \delta^{ab}|<\epsilon$ and $r^2|F|$, $r^3|\nabla F|<1$, then we have
\begin{align}
	|F|^2(x)\leq 8\sum |\nabla^2 V^a|^2\, ,\;\;\; |\nabla^2 F|(x)\leq 8\sum |\nabla^4 V^a|+r^{-3}\,|\nabla V| + r^{-2}|\nabla^2 V|\, .
\end{align}
In particular, if $\cA$ is an annular region and $d(x,\cC)$ is the distance from a point in $\cA$ to the ball centers which were subtracted away, imagine we could prove 
\begin{align}\label{e:outline:eps_gauge_estimates}
&\int_{\cA} d(x,\cC)^{-3}|\nabla V| < \epsilon\, ,\\
&\int_{\cA} d(x,\cC)^{-2}|\nabla^2 V|\, ,\; \int_{\cA} |\nabla^2 V|^2 < \epsilon\, ,
\end{align}
and that $|\langle V^a,V^b\rangle - \delta^{ab}|<\epsilon$ in $\cA$, then we will have finished the annular structure estimates of Theorem \ref{t:annular_region} given in \eqref{e:outline:annular_estimate} and hence the proof of our main Theorems.  \\

We will indeed prove these integral gradient estimates on $V$, though the pointwise $\epsilon$-orthogonality on $\cA$ is a little too much to hope for.  However, we will show in Theorem \ref{t:eps_gauge_transformation} that the $V^a$ do at least form a legitimate vector bundle gauge on all of $\cA$, and in Theorem \ref{t:eps_gauge_on} we will see they even form an $\epsilon$-orthonormal basis away from a set whose $n-4$ content is less than $\epsilon$.  If one is careful, this turns out to be good enough because one can iterate on this bad set in order to eventually get the curvature bounds on all of $\cA$.\\

The integral estimates on $V$ and the $\epsilon$-orthogonality on $\cA$ all follow from the first estimate of \eqref{e:outline:eps_gauge_estimates}, and therefore our main goal is to show this.  Indeed, the various other integral estimates on the hessian of $V$ follow from \eqref{e:outline:eps_gauge_estimates} in combination with $\epsilon$-regularity theorems in $\cA$, while the $\epsilon$-orthonormality follows from \eqref{e:outline:eps_gauge_estimates} combined with a telescoping argument.  We refer the reader to Section \ref{s:eps_gauge_annular_estimates} for more on this, and focus now on the proof of \eqref{e:outline:eps_gauge_estimates}.\\

In the proof of \eqref{e:outline:eps_gauge_estimates} we begin by defining a smoothing of $d(x,\cC)$.  Indeed, recall the packing measure $\mu$ associated to the annular region, and let us define the Green's function and associated Green's distance function:
\begin{align}
&-\Delta G_\mu = \mu\, ,	\notag\\
&b^{-2} = G_\mu\, .
\end{align}
Note that if $\mu$ were exactly the $n-4$ Hausdorff measure on a $n-4$ dimensional subspace then $b$ would be proportional to the distance to that subspace.  Using the Ahlfor's regularity on $\mu$ we prove in Lemma \ref{l:annular_distance} that in $\cA$ we at least have the uniform estimates
\begin{align}
c d(x,\cC) < b(x) < Cd(x,\cC) \, ,\text{ and } c<|\nabla b|<C\, ,
\end{align}
and thus $b$ is a legitimate smoothing of the distance $d(x,\cC)$.\\

Now let $\phi$ be a reasonable cutoff function with $\phi\equiv 1$ on $\cA$, and let us define the scale invariant quantity
\begin{align}
S(r)\equiv r\cdot r^{-3}\int_{b=r} |\nabla V|\phi |\nabla b|\, .	
\end{align}
Then the estimate \eqref{e:outline:eps_gauge_estimates} is equivalent to the Dini estimate
\begin{align}\label{e:outline:dini}
	\int_0^\infty \frac{S(r)}{r} < \epsilon\, .
\end{align}
In order to prove this we will show in Proposition \ref{p:super_convexity} the following superconvexity:
\begin{align}\label{e:outline:superconvex}
r\frac{d}{dr}\Big(d\frac{d}{dr} S\Big)\geq (1-\epsilon)^2 S(r) - e(r)\, ,	
\end{align}
where $e(r)\leq \epsilon \mu\big(\{x\in \cC: c r_x < r < C r_x\}\big)$.  Note that if there was no error, then this superconvexity tells us that $S(r)$ decays and grows polynomially in $r$, which is more than enough for \eqref{e:outline:dini}.  In fact, this is exactly what happens in the four dimensional case.  In the general case, we can at do an ode estimate using \eqref{e:outline:superconvex} in Proposition \ref{p:dini_ode} in order to conclude
\begin{align}
	\int_0^\infty \frac{S(r)}{r} &\leq \int_0^\infty \frac{e(r)}{r}\leq \epsilon \int_0^\infty \frac{\mu\big(\{x\in \cC: c r_x < r < C r_x\}\big)}{r} \leq \epsilon \mu(\cC)<\epsilon\, ,
\end{align}
which finishes the proof of the Dini estimate, and hence the Theorems themselves.

\vspace{1cm}
  
\section{Preliminaries}\label{s:Prelim}

\subsection{Stationary Yang Mills and Monotonicity}\label{ss:monotonicity}

Given a principal bundle $P\to M$ over a Riemannian manifold we can consider the Yang Mills functional, which associates to a connection $A$ the $L^2$ curvature
\begin{align}
  \cF[A] = \int_M |F_A|^2\, .	
\end{align}

A stationary Yang-Mills connection is one which is a critical point of the above functional.  Such a connection solves the Yang-Mills equations
\begin{align}
&d_A F_A = 0\, ,\notag\\
&d_A^* F_A = 0\, .	
\end{align}

The terminology {\it stationary} comes from the fact that such a connection also solves the stationary equation
\begin{align}
\text{div}\Big( |F_A|^2 g_{ij} - 4\langle F_i, F_j\rangle\Big)	= 0\, .
\end{align}
By pairing the above with the radial vector field $\nabla d_x^2$ from a point $x\in M$ we see as in \cite{Price_Mon} that the scale invariant curvature functional
\begin{align}\label{e:monotonicity}
\theta_r(x) \equiv r^{4-n}\int_{B_r(x)} |F_A|^2\, ,	
\end{align}
is a monotone quantity with 
\begin{align}
	\frac{d}{dr} \theta_r(x) = 4r^{4-n}\int_{S_r(x)} \big|F_A[\partial_r]\big|^2\, .
\end{align}

It will often be useful in the constructions to focus on the following, which measures the energy a whole ball:

\begin{align}\label{e:improved_monotonicity}
\overline \theta_r(x) \equiv \sup_{y\in B_r(x)}\theta_r(y)\, .
\end{align}

\subsection{Symmetry of Connections}

We briefly review the notion of symmetry in this subsection.  In this paper we will only be interested in top dimensional symmetry (i.e., $k=n-4$ in what follows). However we give the general definition as it is instructive:

\begin{definition}[$(k,\epsilon)$-Symmetry of Connections]\label{d:symmetry}
	If $A$ is a connection on $B_1(p)$ with $K<\epsilon$ from \eqref{e:YM_assumptions2}, then we say $A$ is $(k,\epsilon)$-symmetric if the following hold:
	\begin{enumerate}
	\item[(0)] For the radial vector field $\partial_r = \nabla d_p$ we have $\fint_{B_1} r^4|F[\partial_r]|^2<\epsilon$.
	\item[(k)] There exists a $k$-dimensional subspace $\cL^k$ such that $\fint_{B_1} |F[\cL]|^2<\epsilon$.
	\end{enumerate}
\end{definition}
\begin{remark}
We say $A$ is $0$-symmetric if only condition $(0)$ holds above. 	
\end{remark}

Using the monotone quantity $\theta_r(x)$ the following is a nice exercise, see for instance \cite{Wa} or \cite{NaVa+}:\\

\begin{theorem}\label{t:independent_points_symmetry}
Let $A$ be a stationary Yang-Mills with $\fint_{B_2}|F_A|^2\leq \Lambda$.  Then for each $\epsilon,\tau>0$ there exists $\delta(n,\epsilon,\tau,\Lambda)>0$ such that if $K<\delta$ from \eqref{e:YM_assumptions2} and there exists $x_0,\ldots,x_{k}\in B_1$ with
\begin{enumerate}
\item if $\cL^\ell = \text{span}\{x_1-x_0,\ldots,x_\ell-x_0\}$ then $d(x_{\ell+1},\cL^\ell)>\tau$,
\item $|\theta_3-\theta_\delta|(x_i)<\delta$,	
\end{enumerate}
	then $B_1$ is $(k,\epsilon)$-symmetric.
\end{theorem}
\vspace{.25cm}

\subsection{Defect Measures}

The first technical result we review is from \cite{Tian_CalYM} and relates the symmetry of the defect measure to the symmetry of the converging connections:

\begin{theorem}[Symmetries of Defect Measures \cite{Tian_CalYM}]\label{t:defect_measure_symmetries}
	Let $A_i\to A$ with $|F_{A_i}|^2dv_{g_i}\to |F_A|^2dv_g+\nu$ be converging stationary Yang-Mills connections with $\fint_{B_4}|F_A|^2\leq \Lambda$.  The following hold:
	\begin{enumerate}
	\item If $\int_{B_2}|F_{A_i}[\cL^k]|^2\to 0$ and $K_i\to 0$ from \eqref{e:YM_assumptions2} then $A$ and $\nu$ are both translation invariant by $\cL^k$.  If $k=n-4$ then $A$ is smooth and $\nu\equiv c\,\lambda^{n-4}_{\cL}$ is a constant multiple of the Hausdorff measure on $\cL$. 
	\item If $A_i$ are $(n-4,i^{-1})$-symmetric wrt $\cL^{n-4}$ then $A\equiv 0$ and $\nu\equiv c\,\lambda^{n-4}_{\cL}$ is a constant multiple of the Hausdorff measure on $\cL$. 
	\end{enumerate}
\end{theorem}
 \vspace{.25cm}
  
The above can be viewed as the basis for the basic regularity result:

\begin{theorem}[\cite{Tian_CalYM}]\label{t:defect_measure_rectifiable}
	Let $A_i\to A$ with $|F_{A_i}|^2dv_{g_i}\to |F_A|^2dv_g+\nu$ be converging stationary Yang-Mills connections with $\fint_{B_4}|F_A|^2\leq \Lambda$.    Then $\nu=e(x)\lambda^{n-4}_S$ is $n-4$ rectifiable with density $e(x)> \epsilon_{n,k}$.
\end{theorem}
\vspace{.25cm}

\subsection{\texorpdfstring{$\epsilon$-regularity for Stationary Yang-Mills}{epsilon-regularity for Stationary Maps}}\label{ss:prelim:eps_reg}

One of the key tools in the study of any nonlinear equation are $\epsilon$-regularity theorems.  In this subsection we will discuss two such theorems.  The first is the classical $\epsilon$-regularity theorem, which tells us that balls with small energy must be smooth.  Precisely, we have the following:

% \begin{theorem}[\cite{Uhl_Rem}]\label{t:eps_reg}
% 	For each $\epsilon>0$ there exists $\delta(n,\epsilon)>0$ such that if $A$ is a stationary Yang-Mills connection with $K<\delta$ from \eqref{e:YM_assumptions2} and $\fint_{B_2} |F_A|^2<\delta$, then $sup_{B_1}|F_A|\leq \epsilon$.
% \end{theorem}

\begin{theorem}[\cite{Uhl_Rem}]\label{t:eps_reg}
	There exists an $\epsilon_{n,k}>0$ such that if $A$ is a stationary Yang-Mills connection with $K<\epsilon_{n,k}$ from \eqref{e:YM_assumptions2} and $\theta(0,2)= 2^{4-n}\fint_{B_2} |F_A|^2<\epsilon_{n,k}$, then $\operatorname{sup}_{B_1}|F_A|^2\leq C(n,k)\fint_{B_2} |F_A|^2\leq 1$.
\end{theorem}

\vspace{.25cm}

For convenience, we also introduce the concept of regularity scale at a point.
\begin{definition}\label{deph_r_A}
 Given a smooth Yang-Mills connection $A$, we define the regularity scale $r_A(x)$ by
 \begin{gather}
  r_A(x)= \sup \cur{s\geq 0 \ \ s.t. \ \ \forall y\in \B {s}{x}\, , \ \ \abs{F_A(y)}\leq s^{-2}}\, .
 \end{gather}
Note that this quantity is scale-invariant, and that an immediate consequence of the $\epsilon$-regularity theorem is that $\theta(0,2r)\leq \epsilon_{n,k}$ implies $r_A(x)\geq r$.
\end{definition}

\vspace{.25cm}

Let us now discuss one further $\epsilon$-regularity result which will play a role in our paper.  The following tells us that instead of assuming the energy is small on a ball, we need only assume the energy is small in a sufficient number of directions.  That is, if a ball is sufficiently symmetric in the sense of Definition \ref{d:symmetry} then automatically the energy is small and a smaller ball is smooth.  Precisely:

\begin{theorem}\label{t:eps_reg_symmetry}
	Let $A$ be a stationary Yang-Mills connection with $\fint_{B_4}|F_A|^2\leq \Lambda$.  Then for each $\epsilon>0$ there exists $\delta(n,\Lambda,\epsilon)>0$ such that if $K<\delta$ from \eqref{e:YM_assumptions2} and $B_2$ is $(n-3,\delta)$ symmetric, then $sup_{B_1}|F_A|\leq \epsilon$.
\end{theorem}
\begin{proof}
We will only sketch the proof, as we refer to \cite{ChNa2} for a verbatim argument in the nonlinear harmonic maps context.	  So indeed, assume for some $\epsilon>0$ such a $\delta(n,\Lambda,\epsilon)>0$ does not exist, so that we can find a sequence of connections $A_i$ for which $B_2$ is $(n-3,\delta_i)$-symmetric, but the curvature is not uniformly bounded.  We may pass to a subsequence to limit $A_i\to A$ with corresponding defect measure $|F_{A_i}|^2dv_{g_i}\to |F_A|^2dv_{\dR^n} + \nu$.  However, by theorem \ref{t:defect_measure_symmetries} we then have that both $A$ and $\nu$ are now $n-3$ symmetric, which is to say invariant under translations of some $n-3$ subspace in $\dR^n$.  However, by theorem \ref{t:defect_measure_rectifiable} we also have that $\nu$ is $n-4$ rectifiable.  Combining these two points tells us that $\nu=0$ must be trivial.  Therefore $A$ is smooth and reduces to a Yang-Mills connection on some $\dR^3$, and is therefore itself a flat connection.  Thus we have that $|F_{A_i}|^2dv_{g_i}\to 0$.  In particular, for far enough down the sequence we may apply the classical $\epsilon$-regularity of Theorem \ref{t:eps_reg} to get a contradiction, and thus prove the Theorem.
\end{proof}
\vspace{.25cm}

\subsection{Yang Mills in Four Dimensions}\label{ss:prelim:n=4}

In this subsection we recall a couple basic points of four dimensional solutions to the Yang Mills equations.  The results of this subsection are well understood, even if they are packaged in a form which is not completely standard.  In addition to the classical methods for proving such results, the methods of Sections \ref{s:eps_gauge} and \ref{s:quant_bubbletree} may be used to give distinct proofs which have the advantage of generalizing to higher dimensions (as we shall see).

We begin by discussing four dimensional annular regions.  The notion of an annular region will be a central point to this paper, especially in higher dimensions where the analysis has been lacking to study such regions.  In dimension four they are quite well understood, and so we briefly review them here.  The classical techniques used to prove the results of this section do not pass to higher dimensions, where the singularities are not isolated.  However, it is still helpful for intuition to review the isolated singularity case, and additionally we will explicitly use these results in the study of finite regions in higher dimensions.  The main result is the following:

\begin{theorem}[Annular Regions in $n=4$]\label{t:annular_n=4}
	Let $A$ be a stationary Yang-Mills connection on a four dimensional space with $r_+^{4-n}\int_{B_{2r_+}} |F_B|^2 \leq \Lambda$, and let $\cA\equiv A_{r_-,r_+}(p)$.  There exists $\epsilon(k)>0$ such that if
\begin{enumerate}
\item $K<\epsilon$ from \eqref{e:YM_assumptions2}.
\item $d_x^2 \,|F_A|(x)< \epsilon$ for $x\in A_{r_-,r_+}$ with $d_x\equiv d(p,x)$.	
\end{enumerate}
then there exists $\alpha(k)>0$ and $C(k)>0$ such that if $\delta\equiv \sup_{A_{r_-,2r_-}} |F_A|+ \sup_{A_{r_+,2r_+}} |F_A|$ then we have the improved estimate for $x\in A_{r_-,r_+}$:
\begin{align}
	d_x^2\,|F_A|(x) < C\Big(\Big(\frac{d_x}{r_+}\Big)^\alpha + \Big(\frac{r_-}{d_x}\Big)^\alpha\Big)\cdot \delta\, .
\end{align}
In particular, we have that $\int_{A_{r_-,r_+}}|F_A|^2 < C(k)\delta^2$ and $\int_{A_{r_-,r_+}}|\nabla^2F_A| < C(k)\delta$
\end{theorem}
\vspace{.25cm}

The classical method for proving results like the above is through a {\it three annulus} type lemma.  The techniques in Section \ref{s:eps_gauge_annular_estimates} also give a (more involved) proof of the above result, though has the advantage of working in higher dimensions.

We now turn our attention to more global information about four dimensional solutions.  The content of the next result is to see that the energy of a bubble can concentrate on at most a finite number of regions.  

\begin{theorem}\label{t:prelim:bubbles}
Let $B$ be a stationary Yang-Mills connection on $\dR^4\times G$ with finite energy $\int_{\dR^4} |F_B|^2 \leq \Lambda$.  We have $C(k,\Lambda), N(k,\Lambda)$ such that $\exists$ disjoint balls $\{B_{r_i}(c_i)\}_1^N\subseteq \dR^4$ with $\int_{B_{r_i}}|F_B|^2>\epsilon(k)$ such that for every $\eta>0$ if $R\geq R(k,\Lambda,\eta)$ then
\begin{align}
\Big|\int_{\dR^4} |F_B|^2-\int_{\bigcup B_{R r_i}(c_i)} |F_B|^2\Big|<\eta\, .
\end{align}
\end{theorem}
\begin{remark}
There is a local version as well if $A$ is a connection on $B_{2r}(0^4)$ with $\fint_{B_{2r}}|F_A|^2\leq \Lambda$.  Then one has the estimate $\Big|\int_{B_r} |F_B|^2-\int_{B_r\cap \bigcup B_{R r_i}(c_i)} |F_B|^2\Big|<\eta$.
\end{remark}

\vspace{.25cm}

In particular, if we consider a sequence of bubbles then the resulting limit may split into at most $N$ independent limiting bubbles.  Classically one could prove this by contradiction using a bubble tree argument and Theorem \ref{t:annular_n=4}.  The techniques of Section \ref{s:bubble_decomp} can be used to give an effective proof, even in higher dimensions.\\

\vspace{.5cm}

\section{Weakly Flat Balls}\label{s:weakly_flat}

In Section \ref{ss:prelim:eps_reg} it was discussed that the structure of a Yang-Mills connection $A$ which is $(n-3,\delta)$-symmetric is quite trivial, namely $A$ is smooth and $\epsilon$-flat.  In this subsection we want to study the structure of connections which are $(n-4,\delta)$-symmetric.  In this case, $A$ need not be trivially smooth, however the structure of $A$ is still quite simple and will form the basic building block for the top stratum of the defect measure.  To understand this better let us discuss the notion of a weakly flat ball:\\

\begin{definition}[$\delta$-weakly flat]\label{d:weak_flat}
If $K<\delta$ from \eqref{e:YM_assumptions2}, then we say a connection $A$ is $\delta$-weakly flat on $B_1(p)$ with respect to $\cL^{n-4}$ if 
\begin{enumerate}
\item For each $y\in \cL_p\cap B_1$ and $\delta\leq r\leq 1$ we have $r^{4}\fint_{B_r(y)}|F[\cL]|^2<\delta$.
\item For $y\in B_1(p)\setminus B_{\delta}(\cL_p)$ we have the estimate $|F_A|(y)\leq \delta\, d(y,\cL_p)^{-2}$.
\end{enumerate}
\end{definition}
\begin{remark}
As with all definitions in this paper we will apply this in a scale invariant sense to any ball.  That is, we call $B_r(p)$ a $\delta$-weakly flat ball if after rescaling $r^{-1}B_r(p)\to B_1(\tilde p)$ we have that the above holds.	
\end{remark}
\vspace{.5cm}

Therefore a weakly flat ball does not have uniformly small curvature, but the curvature is quite small away from a neighborhood of a $n-4$ plane $\cL$.  The notion of a weakly flat ball will come into play at several stages, in particular in our defining of annular regions (one of our two main building blocks) in Section \ref{s:annular}.  Our first result, which follows quite easily from Theorem \ref{t:eps_reg_symmetry}, tells us that balls which are very $n-4$ symmetric must be weakly flat:\\

\begin{theorem}[Existence of Weakly Flat Balls]\label{t:symmetry_implies_weakly_flat}
Let $A$ be a stationary Yang-Mills connection satisfying \eqref{e:YM_assumptions2} with $\fint_{B_2} |F_A|^2 \leq \Lambda$ and let $\epsilon>0$.   Then there exists $\delta(n,k,K,\Lambda,\epsilon)>0$ such that if $B_2$ is $(n-4,\delta)$-symmetric, then either:
\begin{enumerate}
\item $\sup_{B_{3/2}} |F|\leq \epsilon$, or
\item $B_1$ is $\epsilon$-weakly flat and $\fint_{B_1}|F|^2>\epsilon_{n,k}$ .
\end{enumerate}
\end{theorem}
\begin{remark}
The constant $\epsilon_{n,k}>0$ is from the $\epsilon$-regularity of Theorem \ref{t:eps_reg}.	
\end{remark}
\vspace{.25cm}

Let us now consider the following self-improvement theorem for weakly flat balls, which will be used for local estimates in the study of both quantitative bubble regions and annular regions.  In short, it gives us a local pinching estimate which tells us that if $B_1(p)$ looks $\delta$-weakly flat on many scales, then $B_1$ improves and it actually $10^{-2}\delta$-weakly flat.

\begin{theorem}[Curvature Pinching of Weakly Flat Balls]\label{t:weakly_flat_structure}
Let $A$ be a stationary Yang-Mills connection with $0<\kappa<1$ and $\delta>0$.  There exists $c(k,\kappa)>0$ and $\delta'(n,k,\Lambda,\kappa,\delta)>0$ such that if $K<\delta'$ from \eqref{e:YM_assumptions2} with $\fint_{B_{c^{-1}}} |F_A|^2 \leq \Lambda$, $\fint_{B_{c^{-1}}}|F_A[\cL]|^2<\delta'$ and such that $B_r(p)$ is $\delta$-weakly flat ball wrt $\cL$ for $c\leq r\leq c^{-1}$, then we have that $B_1(p)$ is $\kappa\delta$-weakly flat.
%then for $x\in B_1\cap A_{\delta\,r_0,1}(\cL)$ with $d_x\equiv d(x,\cL)$ we have the estimate
%\begin{align}
%d_{x}^2|F_A| < C(k)\delta\Big(d_x^\alpha + \Big(\frac{\delta r_0}{d_x}\Big)^\alpha\Big)\, .
%\end{align}
\end{theorem}
\vspace{.25cm}

If $n=4$ the above can be viewed as a rewriting of Theorem \ref{t:annular_n=4}.  However, in higher dimensions the above result is morally much weaker because the bound on $\fint_{B_2}|F_A[\cL]|^2$ depends on $\delta$.  In particular, iterating this result to obtain further improvements on more scales requires apriori better bounds in the $\cL$-directions.  Improving this estimate to a true generalization of Theorem \ref{t:annular_n=4} which is scale independent is accomplished in the structure theorem on annular regions in Section \ref{s:annular}.  Regardless, the above local version will be useful at several stages as a more coarse estimate.  \\

We end with a key covering tool, which for us will be the dichotomy that either a ball is weakly flat, or away from a set of small $n-4$ content we must have the energy strictly drops.  Precisely, our result is the following:\\

\begin{theorem}[Weakly Flat Covering]\label{t:weakly_flat_decomp}
Let $A$ be a stationary Yang-Mills connection with $K<\delta$ from \eqref{e:YM_assumptions2},  $\fint_{B_{16}} |F_A|^2 \leq \Lambda$.   Then there exists $\eta(n,k,K,\Lambda,\delta)>0$ such that either
\begin{enumerate}
\item $B_4(p)$ is $\delta$-weakly flat, or
\item We can cover $B_1(p)\subseteq \bigcup_c B_{r_c}(x_c)\cup\bigcup_d B_{r_d}(x_d)$ such that
\begin{enumerate}
\item $\sum_c r_c^{n-4}<\delta$,
\item $\overline\theta_{r_d}(x_d)\leq \overline\theta_1(p)-\eta$ with $\sum_d r_d^{n-4}<C(n,k,K,\Lambda,\delta)$.
\end{enumerate}
\end{enumerate}
\end{theorem}
\vspace{.25cm}

\subsection{Proof of Theorem \ref{t:symmetry_implies_weakly_flat}}

The proof is really just an application of theorem \ref{t:defect_measure_symmetries} with a contradiction argument.  Thus let us assume for some $\epsilon>0$ no such $\delta>0$ exists.  Then there exits a sequence $A_i$ of connections which are $(n-4,\delta_i)$-symmetric with $\int_{B_2}|F_{A_i}|^2\leq \Lambda$.  After passing to a subsequence we have
\begin{align}
&A_i\to A\, ,\notag\\
&|F_{A_i}|^2 dv_{g_i}\to |F_A|^2 dv_g + \nu\, .	
\end{align}

Using theorem \ref{t:defect_measure_symmetries} we have that $A\equiv 0$ with $\nu = c \lambda^{n-4}_\cL$ a constant multiple of the $n-4$ Hausdorff measure on $\cL$. In particular, for far enough into the sequence $B_1$ is $\epsilon$-weakly flat.\\

To finish the proof we have two options, either $\nu(B_1)\geq \epsilon_{n,k}$ or not.  In the first case we have $(2)$, thus let us assume $\nu(B_1)< \epsilon_{n,k}$.  In this case we have by Theorem \ref{t:eps_reg} that $A_i$ are uniformly smooth on $B_{1/2}$ sufficiently far in the sequence, and in particular we actually have $\nu\equiv 0$.  Thus we have that $A_i\to 0$ smoothly on $B_{3/2}$, and for far enough into the sequence we have $|F|<\epsilon$ on the ball, which shows case $(1)$ is satisfied. $\square$\\

\subsection{Proof of Theorem \ref{t:weakly_flat_structure}}

For $\delta,\kappa>0$ fixed let us assume no such $\delta'(n,k,\Lambda,\delta,\kappa)>0$ exists.  Thus we have a sequence of connections $A_i$ with $\int_{B_{c^{-1}}}|F_{A_i}[\cL]|^2\to 0$ such that for $c\leq r\leq c^{-1}$ we have that $B_r(p)$ is $\delta$-weakly flat.  We will choose $c=c(n,\kappa)>0$ before the end of the proof.  Passing to a subsequence we can use theorem \ref{t:defect_measure_symmetries} to conclude that $A_i\to A$ and $|F_{A_i}|^2dv_{g_i}\to |F_A|^2dv_g+\nu$, where $A$ defines a smooth $n-4$ symmetric Yang Mills connection on $B_{c^{-1}}(p)$ and $\nu=c\lambda^{n-4}_{\cL}$.  In particular, we can now apply theorem \ref{t:annular_n=4} to conclude for $c\delta \leq d(x,\cL)\leq 1$ the estimate
\begin{align}
	d(x,\cL)^2|F_A|(x)\leq C(k)c(k,\kappa)^\alpha\Big(\Big(\frac{d_x}{1}\Big)^\alpha + \Big(\frac{\delta}{d_x}\Big)^\alpha\Big)\delta\, .
\end{align}
For $c(k,\kappa)$ sufficiently small we see that $B_1(p)$ is $\frac{1}{2}\kappa\delta$-weakly flat.  In particular, for far enough in the sequence we have that $B_1(p_i)$ is $\kappa\delta$-weakly flat, which proves the Theorem. $\square$\\

\subsection{Proof of Theorem \ref{t:weakly_flat_decomp}}

Let us define the set
\begin{align}
E_\eta \equiv \{x\in B_1 : |\theta_{10}-\theta_{\eta}|<\eta\}\, .	
\end{align}

Picking $\tau=\delta^2$ and $c(n)\equiv 10^{-6n}$ and using Theorem \ref{t:independent_points_symmetry}, if $\eta<\eta(n,k,\Lambda,\delta')$ and $Vol(B_{2\tau} E_\eta)>c(n)\delta \tau^{4}$ then we know $B_8$ is $(n-4,\delta')$-symmetric, and hence by Theorem \ref{t:symmetry_implies_weakly_flat} that $B_4$ is $\delta$-weakly flat so that $(1)$ is satisfied.  Therefore, we may assume this is not the case and choose a covering
\begin{align}
B_{2\tau}(E_\eta)\subseteq \bigcup_c B_{r_c}(x_c)\, ,
\end{align}
where $r_c\equiv \tau$ and $\sum r_c^{n-4} < \delta$.  Let us define $r_d\equiv \eta$, and note now that for every point $x\in B_1\setminus \bigcup_c B_{r_c}(x_c)$ we have that $\overline\theta_{r_d}(x)\leq \overline\theta_1(p)-\eta$.  By picking a Vitali covering of $B_1\setminus \bigcup_c B_{r_c}(x_c)\subseteq \bigcup B_{r_d}(x_d)$ of such balls we have completed the construction.  $\square$\\

\vspace{.5cm}

\section{\texorpdfstring{$\delta$-Bubble Regions}{delta-Bubble Regions}}\label{s:bubble}

In this section we study a quantitative form of a bubble region.  These will play a role in both our Quantitative Bubble Tree decomposition in Section \ref{s:quant_bubbletree} and the Annulus/Bubble decomposition of Section \ref{s:bubble_decomp}.  Recall from Definition \ref{d:bubble} the notion of a bubble $B$.  We begin with a definition:

\begin{definition}[$\delta$-Bubble Region]\label{d:bubble_quantitative}
Given that $K<\delta$ from \eqref{e:YM_assumptions2} with $\bb$ $=\{b_j\}$ a discrete set and $r:\bb\to \dR$ such that $r_j\leq \delta$, then $\cB=B_{\delta^{-1}}(p)\setminus \overline B_{r_j}(\bb)$ is a $\delta$-bubble region with respect to $\cL=\cL_B^{n-4}$ if 
\begin{enumerate}
\item[(b1)] If $\bb^\perp\equiv \bb\cap \cL^\perp_p$ then for any $b_i\in \bb$ we have $b_i\in \cL_{b_j^\perp}$ with $r_i=r^\perp_j$ for some $b_j^\perp\in\bb^\perp$.
\item[(b2)] $B_{\delta^{-1}}(p)$ and $\{B_{r_j}(b_j)\}$ are $\delta$-weakly flat.
\item[(b3)] $r_A>\bar r(\Lambda,\delta)$ on $\cB$. 
\item[(b4)] $\{B_{r_j/10}(b_j)\}$ are disjoint with $r_j^{4-n}\int_{B_{r_j}(b_j)}|F|^2>\epsilon(n,k)$ for each $b\in \bb$.
%\item $|\theta_{\delta^{-1}}-\theta_{\delta}|(b_j)>\epsilon_n$.
\end{enumerate}
\end{definition}
\begin{remark}
The constant $\bar r(\Lambda,\delta)>0$ is fixed according to Theorem \ref{t:bubble_existence}.  Indeed, with a little work one could estimate $\bar r$ explicitly as a polynomial of $\delta$, but this would require many pages of tolling with little added value.
\end{remark}
\begin{remark}
	In $(b4)$ we can take $\epsilon(n,k)=\frac{1}{2}\epsilon_{n,k}$ where $\epsilon_{n,k}$ is from the $\epsilon$-regularity of Theorem \ref{t:eps_reg}.
\end{remark}

 One should view a bubble region in the following manner.  In dimension $4$, consider a $\delta$-weakly flat ball $\B {\delta}{p}$. Ideally, a perfect bubble would be a smooth nontrivial YM connection such that for all $x\in \B {\delta^{-1}}{p}$ we have $r_A\geq r(\Lambda,\delta)$.  This implies that in some sense the curvature is not concentrating on too small scales on this bubble, and in particular the $\epsilon$-regularity theorem ensures that a sequence of such bubbles would automatically converge smoothly with estimates to another $YM$ connection.  However, it may happen that this is not the case. Thus, if we have some concentration of energy on scales smaller than $r(\Lambda,\delta)$, we do not want these pieces to be part of our bubble, and so we cut them out by our balls $B_{r_j}(b_j)$. By definition, each one of this pieces will carry some definite amount of scale-invariant curvature, since otherwise we need not have cut them out in the first place.
 
 In higher dimensions, the situation is very similar. The only difference is that, instead of cutting out $n$-dimensional balls, we cut out tubes around $n-4$ dimensional planes. 
 
 Indeed, one should think of a bubble as being
 \begin{gather}
  \cB = B_{\delta^{-1}}(p) \setminus \bigcup_{b_i\in \bb^\perp} \overline B_{r_i}(b_i+\cL)\, .
 \end{gather}
 For future convenience, instead of writing a bubble as subtracting off a tube around the planes $\cL_{\bb^\perp}$ from $B_{\delta^{-1}}$, we will instead cover those planes by a Vitali collection of balls centered on them and subtract this collection off.  There is no fundamental difference except it is more convenient for technical reasons later.  Condition (b1) is a rephrasing of this idea.\\

We will prove two primary results in this section.  The first is an existence theorem, which will both tell us when bubble regions exist and fix for us our constant $\bar r$ in our definition.  The criteria for existence will help us in our construction of bubble regions in our quantitative bubble tree decomposition.  The second purpose of this section is a structure theorem, the results of which are mostly a straight forward consequence of the definition combined with the knowledge of the behavior of four dimensional solutions given in Section \ref{ss:prelim:n=4} .  This is as opposed to the corresponding structure theorem on $\delta$-annular regions which will be introduced in the next section, which will be quite challenging.\\  

Let us begin with our main existence theorem.  The result will build for us nontrivial bubble regions in the sense that the energy of any removed ball $B_{r_i}(b_i)$ will drop by some strict amount.  This will be an important aspect of future constructions:

\begin{theorem}[Existence of Bubble Regions]\label{t:bubble_existence}
Let $A$ be a stationary Yang-Mills connection, and assume $\int_{B_{2\delta^{-1}}} | F[\cL] |^2<\delta'$, $K<\delta'$ from \eqref{e:YM_assumptions2} and that $B_{\delta^{-1}}(p)$ is $\delta$-weakly flat wrt $\cL$.  If $\delta'<\delta'(n,k,\Lambda,\delta)$, then for $\bar r(\Lambda,\delta)>0$ in Definition \ref{d:bubble_quantitative} there exists a $\delta$-bubble region $\cB=B_{\delta^{-1}}(p)\setminus \overline B_{r_x}(\bb)$ with best subspace $\cL$.  If further we assume $B_r(b)$ is not a $\delta$-weakly flat ball for some $\delta^3<r<\delta^{-1}$ and some $b\in\bb$, then we can build the bubble so that $|\overline\theta_{2}-\overline\theta_{r_b}|(b)>\epsilon(n,k)$ for all $b\in \bb$.
\end{theorem}
\begin{remark}	
	$|\overline\theta_{1}-\overline\theta_{r_b}|(b)>\epsilon(n,k)$ is a nontriviality condition, which roughly says that if the bubble is not really an annular region in disguise, then there is a definite energy drop blow the bubble.
\end{remark}
\vspace{.25cm}

Our primary structure theorem for bubble regions is now the following.

\begin{theorem}[Structure of Bubble Regions]\label{t:bubble_region}
	Let $A$ be a stationary Yang-Mills connection on a $\delta$-bubble region $\cB = B_{\delta^{-1}}\setminus \overline B_{r_b}(\bb )$ with $\fint_{B_{2\delta^{-1}}} |F_A|^2 \leq \Lambda$ and $\fint_{B_{2\delta^{-1}}} |F[\cL]|^2 \leq \delta'$.  For each $\epsilon>0$ if $\delta<\delta(n,k,\Lambda,\epsilon)$ and $\delta'<\delta'(n,k,\Lambda,\delta)$  then:
	\begin{enumerate}
	\item If $b_i\in\bb$ with $r_i\leq r\leq \delta^{-1}$ then $\sum_{r_j\in B_r(b_i)} r_j^{n-4} < N(n,k,\Lambda) r^{n-4}$.
	\item For $q\in \cL\cap B_1$ let $b_{q,j}=\cL^\perp_q\cap \cL_{b^\perp_j}$, then $\big|\theta_{\delta^{-1}}(q)-\omega_{n-4}\int_{\cB_q}|F_A|^2 - \sum \theta_{r_j}(b_{q,j})\big|<\epsilon$.
	\item For $q\in \cL\cap B_1$ $\exists$ $c_{i}\in \cB_q\cap B_1$ and $s_{i}>0$ with $\#\{c_i\}\leq N(n,k,\Lambda)$ and $s_i^{4-n}\int_{B_{s_i}}|F|^2>\epsilon(k)$ such that if $R\geq R(n,k,\Lambda,\epsilon)$ then $\big|\int_{\cB_q\cap \bigcup B_{R s_i}(c_i)}|F_A|^2 - \int_{\cB_q}|F_A|^2\big|<\epsilon$.
	\item We have the estimate $\delta^{n-4}\int_{\cB} |\nabla^2 F_A|<C(n,k,\Lambda,\delta)$.
	\end{enumerate}
\end{theorem}
\begin{remark}
In fact, it will be a consequence of Theorem \ref{t:main_L1_hessian} that we have the better estimate $\delta^{n-4}\int_{\cB} |\nabla^2 F_A|<C(n,k,\Lambda)$, without the $\delta$-dependence.  However, to prove Theorem \ref{t:main_L1_hessian} it will be enough to first prove this weaker result.	
\end{remark}
\begin{remark}
	$(2)$ is the energy identity for bubbles and tells us that we can compute the energy at a point in terms of the energy of a slice.  
\end{remark}
\begin{remark}	
	$(3)$ is a concentration condition and tells us that the energy of the bubble can separate into at most $N$ clumps, independent of $\delta$.
\end{remark}
%Finally we end with the following, which tells us how to compute the energy 

\vspace{.25cm}

\subsection{Existence of Bubble Regions}
 
In this subsection we prove Theorem \ref{t:bubble_existence}.  Thus let us fix $\delta>0$ and assume $\fint_{B_{2\delta^{-1}}}|F[\cL]|^2<\delta'$, which will be fixed later.  For $0<\delta''<\delta$ fixed observe that $\forall$ $x\in B_{\delta^{-1}}$ there exists a radius $r_x>0$ with $\bar r(\delta'',\Lambda)<r_x<\delta^5$ such that $|\theta_{2r_x}-\theta_{\delta'' r_x}|(x)<\delta''$.  Indeed, to see this consider the sequence of radii $s_a\equiv (\delta'')^a \delta^5$ and note for all $N\in \dN$ that
\begin{align}
\sum_1^N |\theta_{s_a}-\theta_{s_{a+1}}|(x)\leq \Lambda\, .	
\end{align}
Therefore for $N=\Lambda(\delta'')^{-1}+1$ we see that for one of the radii $s_0,\ldots s_N$ we must have $|\theta_{s_a}-\theta_{s_{a+1}}|(x)<\delta''$ as claimed, otherwise by monotonocity of $\theta$ we contradict the above sum.  Given this we can choose $\delta'<\delta'(\Lambda,\delta'')$ so that we can also be assured $r_x^{4-n}\fint_{B_{2r_x}(x)}|F[\cL]|^2<\delta''$, so that $B_{2r_x}(x)$ is $(n-4,\delta'')$-symmetric.  We now pick $\delta''\leq \delta''(n,k,\Lambda,\delta)$ so that Theorem \ref{t:symmetry_implies_weakly_flat} holds with $\delta$.\\

To finish the construction of $\cB$ let us consider for $x\in \cL^\perp_p\cap B_1$ the covering $\{B_{r_x}(x)\}$ of $\cL^\perp_p\cap B_1$.  Let $\{B_{r^\perp_j}(x^\perp_j)\}$ be a Vitali subcovering with $x^\perp_j\in \cL^\perp\cap B_1$ such that $r^\perp_j = r_{x_j}$ and with $\{B_{r^\perp_j/10}(x^\perp_j)\}$ disjoint.  By translating the centers balls over $\cL$ and picking another Vitali subcovering we can extend this to a covering $\{B_{r'_j}(x'_j)\}$ of $B_1$ such that for each $x'_j$ there exists $x^\perp_i$ with $x'_j\in \cL_{x^\perp_i}$ and $r'_j=r^\perp_i$.  Now by using the uniform lower bound on the radii we can apply Theorem \ref{t:symmetry_implies_weakly_flat} to each ball $B_{r^\perp_j}(x^\perp_j)$ in order to conclude that either $r_A(x^\perp_j) > \frac{1}{2}r^\perp_j\geq \bar r(\Lambda,\delta)$ on $B_{r^\perp_j}(x^\perp_j)$ or that $B_{r^\perp_j}(x^\perp_j)$ is $\delta/2$-weakly flat with $\theta_{r^\perp_j}(x^\perp_j)> \epsilon_{n,k}$.  Let us now define $\bb^\perp=\{b^\perp_j\}\subseteq \{x'_j\}$ to be the collection of balls for which this second condition holds, and let $\bb\subseteq \{x'_j\}$ be the subset for which $b_j\in \cL_{b^\perp_i}$ for some $b^\perp_i\in \bb^\perp$.  We then define $\cB\equiv B_{\delta^{-1}}\setminus \bigcup B_{r_j}(b_j)$ as expected.  Note that on $\cB$ we have the regularity estimate $r_A\geq \frac{1}{2}\min r'_j \geq \bar r(\Lambda,\delta)$.  Therefore with $\delta'<\delta'(n,k,\Lambda,\delta)$ we have that $\cB$ is indeed a $\delta$-bubble region.\\

Let us now focus on proving $|\theta_{2}-\theta_{r_j}|(b_j)>\epsilon(n,k)$ under the assumption that for some $b\in\bb$ and $\delta^4<r<1$ we have that $B_r(b)$ is not $\delta$-weakly flat, at least if we choose $\delta'<\delta'(n,k,\Lambda,\delta)$ sufficiently small.  So assume this is not the case for any $\delta'$, then we can find a sequence of connections $\cA_i$ with bubble regions $\cB_i\equiv B_1\setminus B_{r_{i,x}}(\bb_{i})$ as above with $\fint_{B_1}|F[\cL_i]|^2\to 0$.  After passing to a subsequence we can limit $A_i\to A$ with $|F_{A_i}|^2dv_{g_i}\to |F_A|^2dv_g + \nu$ and $\cB_i\to \cB=B_1\setminus B_{r_x}(\bb)$, and by using Theorem \ref{t:defect_measure_symmetries} we have that $A$ defines a smooth Yang Mills connection on $\cL^\perp$ and that $\nu$ is invariant under translation by $\cL$.  By our contradicting assumption we have that
\begin{align}
\big|(|F_A|^2dv_g + \nu)[B_{2}(b)] - (|F_A|^2dv_g + \nu)[B_{r_b}(b)]\big|\leq \epsilon(n,k)\, .	
\end{align}
By choosing $\epsilon(n,k)$ sufficiently small, from the $\epsilon$-regularity of Theorem \ref{t:eps_reg}, we then have that $\nu\equiv 0$ on $A_{r_b,2}(b)$ with $d_x^2 |F_A|< \epsilon'(n,k)$ on $A_{r_b,2}(b)$, where $d_x=d(x,b)$ and $\epsilon'(n,k)$ is from Theorem \ref{t:annular_n=4}.  Indeed, since $B_{\delta^{-1}}(p)$ is $\delta$-weakly flat we even have that $d_x^2 |F_A|< \epsilon'(n,k)$ on $A_{r_b,\delta^{-1}}(b)$.  Now let us apply Theorem \ref{t:eps_gauge_existence}, and using that $B_{\delta^{-1}}(b)$ and $B_{r_b}(b)$ are $\delta$-weakly flat, we then get the improved estimate for $x\in A_{r_b,\delta^{-1}}$:
\begin{align}
	d_x^2\,|F_A|(x)< C(k)\delta\Big(\big(\delta\,d_x\big)^{\alpha}+\Big(\frac{r_b}{d_x}\Big)^\alpha\Big)\, ,
\end{align}
Now recall our assumption is that for some $\delta^3<r<\delta^{-1}$ we have that $B_r(b)$ is not $\delta$-weakly flat.  In particular, this implies in our case that for some $x\in A_{\delta^4,10}(b)$ we must have that $d_x^2\,|F_A|(x)>\delta$.  However, for $\delta<\delta(k)$ and $r_b<\delta^5$ as constructed we see from the above estimate that this is not possible, and thus we have found our desired contradiction and proved the Theorem. $\square$
\vspace{.25cm}

\subsection{Structure of Bubble Regions}

In this subsection we give a proof of Theorem \ref{t:bubble_region}.  The proof of the content estimate $(1)$ follows immediately from $(b2)$ and $(b4)$ in the definition of a $\delta$-bubble region.  Indeed, using the lower bound $r_j^{n-4}\int_{B_{r_j}(x_j)}|F_A|^2 > \epsilon(n,k)$ together with the fact that $B_{r_j}(x_j)$ are $\delta$-weakly flat we immediately have lower bounds on the slightly smaller balls
\begin{align}
	r_j^{n-4}\int_{B_{r_j/10}(x_j)}|F_A|^2 > \frac{1}{2}\epsilon(n,k)\, .
\end{align}
Combining this with the disjoint property of $\{B_{r_j/10}(x_j)\}$ we have 
\begin{align}
\sum r_j^{n-4}\leq \sum 2\epsilon(n,k)^{-1}\int_{B_{r_j/10}}|F_A|^2 \leq C(n,k)\int_{B_2}|F_A|^2\leq C(n,k)\Lambda\, ,
\end{align}
which proves the content estimate.

Let us now focus on the hessian estimate of $(4)$.  Indeed, for this we use $(b3)$ to see that $r_A>\bar r(\Lambda,\delta)$ on $\cB\subseteq B_{\delta^{-1}}$.  Standard elliptic estimates then give us that $|\nabla^2 F_A|<C(n,k,\Lambda,\delta)$ pointwise on $\cB$, which in particular implies the weaker $L^1$ estimate 
\begin{align}
	\int_{\cB} |\nabla^2 F_A| \leq C(n,k,\Lambda,\delta)\Vol(B_{\delta^{-1}})\, ,
\end{align}
as claimed. \\

We are now left with proving that $(2)$ and $(3)$ hold for $\delta'$ sufficiently small.  So assume this is not the case for any $\delta'$, then we can find a sequence of connections $\cA_i$ with bubble regions $\cB_i\equiv B_1\setminus B_{r_{i,x}}(\bb_{i})$ as above with $\fint_{B_1}|F[\cL_i]|^2\to 0$.  After passing to a subsequence we can limit $A_i\to A$ with $|F_{A_i}|^2dv_{g_i}\to |F_A|^2dv_g + \nu$ and $\cB_i\to \cB=B_1\setminus B_{r_x}(\bb)$, and by using Theorem \ref{t:defect_measure_symmetries} we have that $A$ defines a smooth Yang Mills connection on $\cL^\perp$ and that $\nu$ is invariant under translation by $\cL$. 

Let us first focus on $(2)$, and it is enough to prove this for $q=p$ as the other cases are verbatim.  Indeed, viewing $A$ as a connection on the four dimensional space $\cL^\perp$ we have the equality
\begin{align}\label{e:bubble_existence:1}
&\nu[B_{\delta^{-1}}]+\int_{B_{\delta^{-1}}(0^\perp)}|F_A|^2 = \sum\Big(\nu[B_{r^\perp_j/10}(b^\perp_j)]+ \int_{B_{r^\perp_j}(b_j^\perp)}|F_A|^2\Big)+ \int_{\cB\cap \cL^\perp} |F_A|^2\, .
\end{align}
Using that $\theta_{\delta^{-1}}(p_i)\to \nu[B_{\delta^{-1}}]+\int_{B_{\delta^{-1}}}|F_A|^2$ with $\theta_{r_{i,j}}(b^\perp_{i,j})\to \nu[B_{r_j/10}]+\int_{B_{r_j}(b_j^\perp)}|F_A|^2$ and $\omega_{n-4}\int_{\cB_{p_i,i}}|F_{A_i}|^2\to \int_{\cB\cap \cL^\perp} |F_A|^2$ this shows $(2)$ must hold for sufficiently far in the sequence.  Thus we can concentrate on $(3)$.

To prove $(3)$ let us apply Theorem \ref{t:prelim:bubbles} to the limit $A$ on $\cL^\perp\cap B_4$ to get points $c^A_1,\ldots, c^A_N\in \cL^\perp\cap B_1$.   By construction we have for $R\geq R(n,k,\Lambda,\epsilon)$ that
\begin{align}
\int_{B_2\cap B_{Rs_j}(c_j)}|F_A|^2 < \frac{\epsilon}{4}\, .	
\end{align}
Since $B_{\delta^{-1}}(p)$ is $\delta$-weakly flat we may apply Theorem \ref{t:annular_n=4} in order to see that $\int_{A_{1,\delta^{-1}}(p)}|F_A|^2 < C(k)\delta<\frac{\epsilon}{4}$, and hence
\begin{align}
\int_{B_{\delta^{-1}}\cap B_{Rs_j}(c_j)}|F_A|^2 < \frac{\epsilon}{2}\, .	
\end{align}
Finally, observing that $\text{supp}[\nu]\subseteq \bigcup_{\bb} B_{r_b}(b)$ this gives us for sufficiently far in the sequence that
\begin{align}
\int_{\cB_p\cap B_{\delta^{-1}}\cap B_{Rs_j}(c_j)}|F_{A_i}|^2 < \epsilon\, ,	
\end{align}
which shows that $(3)$ holds and thus finishes the proof. $\square$

\vspace{.5cm}

\section{\texorpdfstring{$\delta$-Annular Regions}{delta-Annular Regions}}\label{s:annular}

In this section we consider the second piece of our quantitative decompositions, namely the $\delta$-annular regions.  These are the regions which will turn out to be the most challenging to analyze, and the next several sections of this paper will be dedicated to proving the results stated in this section.  \\

Quantitative bubble regions have the property that they have large curvature, but only on bounded domains.  On the other hand, quantitative annular regions have small curvature, but over potentially an infinite number of scales.  Recalling the notion of a weakly flat region in Definition \ref{d:weak_flat} we define an annular region in the following manner:\\

\begin{definition}\label{d:annulus}
If $K_M<\delta$ from \eqref{e:YM_assumptions2}, then we call $\cA\subseteq B_2(p)$ a $\delta$-annular region if there exists a closed subset $\cC = \cC_0\cup \cC_+=\cC_0\cup\{x_i\}$, a radius function $r:\cC\to \dR^+$ with $0<r_x\leq \delta$ on $\cC_+$ and $r_x=0$ on $\cC_0$, and a $n-4$ subspace $\cL$ such that $\cA \equiv B_2\setminus \overline B_{r_x}(\cC)$ satisfies
\begin{enumerate}
	\item[(a1)] $\{B_{\tau^2 r_x}(x)\}$ are pairwise disjoint.
	\item[(a2)] For each $x\in \cC$ and $r_x\leq r\leq 2$ we have that $B_r(x)$ is $\delta$-weakly flat wrt $\cL_x\equiv \cL+x$.
	\item[(a3)] For each $x\in \cC$ and $r_x\leq r$ with $B_{2r}(x)\subseteq B_2$ we have that $\cL_{x}\cap B_r \subseteq B_{\tau r}(\cC)$ and $\cC\cap B_r\subseteq B_{\delta r}(\cL_x)$.
	\item[(a4)] $|\Lip\, r_x|\leq \delta$.
\end{enumerate}	
For each $\tau\leq s\leq 1$ we define the regions $\cA_s\equiv B_2\setminus \overline B_{s\cdot r_x}(\cC)$ as well as the wedge regions \newline $\cW^\theta(x)\equiv \big\{y\in A_{r_x/2,2}(x): d(y,\cL_x)\geq \cos(\theta)\, d(y,x)\big\}$ and $\cW^\theta_r(x) \equiv \cW^\theta(x)\cap A_{\cos\theta\, r,r/\cos\theta}(x)$.
\end{definition}
\begin{remark}
The constant $\tau=\tau_n = 10^{-10n}\omega_n$ is a dimensionally chosen constant designed to neutralize any errors obtained overlaps in covering constructions.
\end{remark}
\begin{remark}
For a smooth connection we have $\cC_0=\emptyset$.  For nonsmooth stationary connections one must allow for the possibility that $\cC_0\neq \emptyset$.	
\end{remark}

\begin{remark}\label{r:annular_intersection}
Note that $(a1)$ and $(a4)$ imply that $B_{10r_x}(x)$ intersects at most $C(n)$ other balls in the covering, all of which have radii which are in the range $[\frac{1}{2}r_x,2r_x]$.	
\end{remark}

\vspace{.5cm}

Associated to each annular region is its packing measure.  In the same way in which $\cC$ approximates the $n-4$ defect measure of the connection, we have that the packing measure approximates the $n-4$ Hausdorff measure on the support of this defect measure.  Precisely we have the following:\\

\begin{definition}
	Let $\cA \equiv B_2\setminus \overline B_{r_x}(\cC)$ be an annular region, then we define the associated packing measure
	\begin{align}
    \mu=\mu_\cA \equiv \sum_{x\in \cC_+} r_x^{n-4}\delta_{x} + \lambda^{n-4}|_{\cC_0}\, ,
\end{align}
where $\lambda^{n-4}|_{\cC_0}$ is the $n-4$-dimensional Hausdorff measure restricted to $\cC_0$.
\end{definition}
\vspace{.5cm}

The goal of this section is two fold.  We will first discuss some properties of annular regions.  Our three main properties about annular regions will be to show that the packing measure is Ahlfor's regular, and that in the annular region we have apriori $L^1$ hessian and $L^2$ curvature bounds.  These estimates will be the eventual key to both the global $L^1$ hessian estimate of Theorem \ref{t:main_L1_hessian} and the energy identity of Theorem \ref{t:main_energy_quantization}.  We will prove the Ahlfor's regularity statement in this section, however the curvatures estimates will not be proved until later in the paper, as there is a lot of new technical constructions needed in their proofs.

Our second main goal will be to prove the existence of annular regions.  In order for an annular region to be useful in the end analysis, we will need to know many exist.  In this section we will give some basic criteria used to construct {\it maximal} annular regions.  The maximal property of the constructed annular regions will be crucial in the proof of the annulus/bubble decomposition in Section \ref{s:bubble_decomp}.\\

Let us now begin by stating our main structural result on the properties of annular regions:

\begin{theorem}[Structure of Annular Regions]\label{t:annular_region}
	Let $A$ be a stationary Yang-Mills connection on a $\delta$-annular region $\cA = B_2\setminus \overline B_{r_x}(\cC)$ satisfying \eqref{e:YM_assumptions2} and $\fint_{B_4} |F_A|^2 \leq \Lambda$.  For each $\epsilon>0$ if $\delta<\delta(n,\Lambda,k,\epsilon)$ we then have:
	\begin{enumerate}
	\item For each $x\in \cC$ and $r_x<r<4$ we have that $A(n)^{-1} r^{n-4}\leq \mu\big(B_r(x)\big)\leq A(n) r^{n-4}$.
	\item We have the estimate $\int_{\cA\cap B_1} |\nabla^2 F|<\epsilon$.
	\item We have the estimate $\int_{\cA\cap B_1} |F_A|^2 < \epsilon$.
	%\item Let $\cC'\equiv\big\{x\in\cC : \exists\, q\in\cL $ with $\cL^\perp_q\cap \bar B_{r_x}(x)\neq\emptyset$ and $\int_{\cL^\perp_q}|F_A|^2>\epsilon\big\}$, then $\mu(\cC')<\epsilon$.
	\end{enumerate}
\end{theorem}
\vspace{.5cm}

\begin{comment}
An interesting corollary of this structure result is the following, which one consider as an important ingredient toward the proof of the energy identity:\\

\begin{corollary}\label{c:annular_region}
	Let $A$ be a stationary connection on a $\delta$-annular region $\cA = B_2\setminus \overline B_{r_x}(\cC)$ wrt $\cL$ satisfying \eqref{e:YM_assumptions2} and $\fint_{B_4} |F_A|^2 \leq \Lambda$.  For each $\epsilon>0$ if $\delta<\delta(n,\Lambda,k,\epsilon)$ we then have the following estimate:
	\begin{align}
		\int_{B_1} \big|\theta_1-\theta_{r_x}\big|(x)\,d\mu = \int_{B_1} \Big|\int_{B_1(x)}|F_A|^2- r_x^{4-n}\int_{B_{r_x}(x)}|F_A|^2\Big|\,d\mu<\epsilon
	\end{align}
\end{corollary}
\vspace{.5cm}
\end{comment}

With this in hand let us see what criteria may be used to build annular regions.  A key result in the annulus/bubble decomposition of Theorem \ref{t:bubble_decomposition} is the $n-4$ content bound on the number of pieces to the decomposition.  It is worth noting that for such an estimate to hold, one must be quite careful about the construction of the annular regions.  Indeed, if one were to build annular regions which were much smaller than they need to be, it is possible the content estimate would fail.  Therefore, we must also analyze what it means to build {\it maximal} annular regions, a concept which will be made precise in the following Theorem:

\begin{theorem}[Existence of Annular Regions]\label{t:annular_existence}
Let $A$ be a stationary Yang-Mills connection with $\int_{B_4}|F_A|^2\leq \Lambda$.  For each $\epsilon>0$ and $0<\delta<\delta(n,k,\Lambda,\epsilon)$ there exists $\delta'(n,k,K,\Lambda,\delta)$ such that if 
\begin{enumerate}
\item $\int_{B_8} |F[\cL]|^2<\delta'$ ,
\item For each $4^{-1}\delta\leq r\leq 4$ we have that $B_r(p)$ is $\delta$-weakly flat wrt $\cL$,
\end{enumerate}
then $\exists$ a $\delta$-annular region $\cA=B_2\setminus \overline B_{r_x}(\cC)$ wrt $\cL$ such that if we consider the set
\begin{align}
	\cC^c\equiv \begin{cases}
 	x\in \cC:& B_{r}(x) \text{ is }\delta\text{-weakly flat for }\delta^4 r_x\leq r\leq r_x\, ,\text{ or }\\
 	x\in\cC :& \exists\, q\in\cL \text{ with }\cL^\perp_q\cap \bar B_{r_x}(x)\neq\emptyset \text{ and }\int_{\cL^\perp_q\cap \cA}|F_A|^2>\epsilon\, ,
 \end{cases}
	%\{x\in \cC: B_{r}(x) \text{ is }\delta\text{-weakly flat for }\delta^4 r_x\leq r\leq r_x\}\, ,
\end{align}
then we have the estimate $\mu(\cC^c\cap B_1)<\epsilon$.
\end{theorem}
\vspace{.5cm}

\begin{comment}
\subsection{Basic Notation and Results}

In this subsection we make explicit some natural notions associated to an annular region, and we also state some basic structure results.

\begin{lemma}\label{l:annular_reg_scale}
Let $A$ be a stationary Yang-Mills connection on a $\delta$-annular region $\cA = B_2\setminus \overline B_{r_x}(\cC)$ satisfying \eqref{e:YM_assumptions2} and $\fint_{B_2} |F_A|^2 \leq \Lambda$.  If $\delta<\delta(n,k,\Lambda)$ then we have for each $x\in \cA_{10^{-6}}$ the estimate $\frac{1}{2}d(x,\cC)\leq r_A(x)\leq 2 d(x,\cC)$.
\end{lemma}
\vspace{.5cm}
\end{comment}

\subsection{Proof of Ahlfor's Regularity of Annular Region}\label{ss:annular_ahlfors_reg}

In this subsection we prove the Ahlfor's regularity of Theorem \ref{t:annular_region}.1 .  Thus throughout this section $\cA=B_2\setminus \overline B_{r_x}(x)$ is a $\delta$-annular region and $\mu$ is its packing measure.  To prove the result let us begin with the following:\\

{\bf Claim: } The projection mapping $\pi:\cC\to \cL$, from the center points to the annular best plane $\cL$, is a $1+\epsilon_n$-bilipschitz map where $\epsilon_n<100^{-1}$.  That is, for $x,y\in \cC$ we have that $(1-\epsilon_n)|\pi(x)-\pi(y)| \leq d(x,y)\leq (1+\epsilon_n)|\pi(x)-\pi(y)|$.\\

To prove the claim let us consider $x,y\in\cC$ and let $r\equiv d(x,y)$.  Then by condition $(a3)$ we have for $\cL_x\equiv \cL+x$ that $\cL_x\cap \overline B_r\subseteq B_{\tau_n r}(\cC)$ and $\cC\cap \overline B_r\subseteq B_{\tau_n r}(\cL)$.  In particular, this gives us that
\begin{align}
|\pi(x)-\pi(y)|	-\tau_n r \leq d(x,y) \leq |\pi(x)-\pi(y)| +\tau_n r\, .
\end{align}
However, we have chosen our scale so that $r\equiv d(x,y)$, and therefore by rearranging we have
\begin{align}
	\frac{1}{1+\tau_n}|\pi(x)-\pi(y)|\leq d(x,y)\leq \frac{1}{1-\tau_n}|\pi(x)-\pi(y)|\, ,
\end{align}
which finishes the proof of the claim. $\square$\\

Now to finish the proof let us pick $x\in \cC$ with $r_x\leq r<1$ such that $B_{2r}(x)\subseteq B_2$.  Let us first prove the upper bound on $\mu(B_r(x))$.  First note that by the bilipschitz condition we have that
\begin{align}
\pi(\cC\cap B_r(x))\subseteq B_{2r}(\pi(x))\, .	
\end{align}
Further, since the collection of balls $\{B_{\tau^2 r_y}(y)\}$ are all disjoint, if we again use the bilipschitz condition we must have that the image balls $\{B_{\tau^3 r_y}(\pi(y)\}$ are disjoint in $\cL$.  But then this give us
\begin{align}
\mu(B_r(x)) &\equiv \sum_{y\in\cC_+\cap B_r(x)} r_y^{n-4} +\lambda^{n-4}\big(\cC_0\cap B_r(x)\big) \leq \tau^{-3(n-4)}_n\Big(\sum_{y\in\cC_+\cap B_r(x)} (\tau_n^3\,r_y)^{n-4}\Big)+\lambda^{n-4}\big(\cC_0\cap B_r(x)\big)\notag\\
&\leq C(n)\Big(\sum_{x\in\cC_+\cap B_r}\Vol(B_{\tau^3 r_y}(\pi(y)))+Vol(\pi(\cC_0\cap B_r))\Big)\leq C(n)\Vol(B_{2r}(\pi(x)))\leq A(n)r^{n-4}\, ,
\end{align}
where we have used the bilipschitz condition on $\pi$ multiple times and that $\pi(B_r(x))\subseteq B_{2r}(\pi(x))$.\\

In order to prove the lower bound we start with the following claim:\\

{\bf Claim: } For each $x\in \cC$ and $r_x\leq r<1$ with $B_{2r}(x)\subseteq B_2$ we have that $B_{r_y}\big(\pi(\cC\cap B_r(x))\big)=\bigcup_{y\in \pi(\cC\cap B_r(x))} B_{r_y}(y)\supseteq B_{r/2}(\pi(x))$.\\

To prove the result let us assume it is false and let $y\in B_{r/2}(\pi(x))\setminus B_{r_y}\big(\pi(\cC\cap B_r(x))\big)$.  Let us choose $z\in \pi(\cC\cap B_r(x))$ such that
\begin{align}
z\in \text{argmin}\{ d(z,y):z\in \pi(\cC\cap B_r(x))\}	\, .
\end{align}

Let $s\equiv d(y,z)$, and note that by our assumption on $y$ we have that $s\geq r_z$.  But if we apply condition $(a3)$ to $B_{2s}(z)$, then we know that $\pi(\cC\cap B_r(x))$ is $2\tau_n s$ dense in $B_{2s}(\pi(z))\subseteq \cL$, which contradicts that $z$ is the closest point of $\pi(\cC)$ to $y$ and satisfies $d(y,z)=s$.  This proves the claim. $\square$\\

With the claim in hand we can now easily prove the lower volume bound.  Indeed, we have:
\begin{align}
\mu(B_r(x)) = \sum_{y\in \cC\cap B_r} r_y^{n-4} \geq C(n)^{-1}\sum_{y: \pi(y)\in B_{r/2}(\pi(x))} \Vol(B_{r_y}(\pi(y)))\geq C(n)^{-1}\Vol(B_{r/2}(\pi(x)))\geq A(n)^{-1} r^{n-4}\, ,
\end{align}
which finishes the proof of the lower bound. $\square$ \\

\begin{comment}
\subsection{Proof of Corollary \ref{c:annular_region}}

In this subsection we prove Corollary \ref{c:annular_region} under the assumption that we have proved the structural results of Theorem \ref{t:annular_region}.  Let us begin by associating to $\cC$ the function 
\begin{align}
f(x)\equiv r_x^{n-4}\int_{B_{r_x}(x)}|F_A|^2\, .	
\end{align}

First note that since we are on a $\delta$-annular region we have the estimate
\begin{align}
	|\theta_{1}-\theta_{1/10}|(x) < 	C(n)\delta\, ,\;\; \forall x\in\cC\, .
\end{align}
Further, by combining this with $(3)$ of Theorem \ref{t:annular_region} we have that
\begin{align}
\Big|\theta_{1}(x)-\fint_{B_{r_x}(\cC)\cap B_{1}(p)}|F_A|^2\Big| < 	10^{-1}\epsilon\, .
\end{align}
\vspace{.25cm}
\end{comment}

\subsection{Existence of Annular Regions}

In this section we deal with the issue of proving the existence of a $\delta$-annular region which is maximal in a suitable sense.  Let us pick some $\delta''>0$ which will be fixed later, and now we begin by defining the following for each $x\in B_2$:
\begin{align}
&s_x\equiv \inf_{s\leq 1}\Big\{\forall s\leq r\leq 1 \text{ } r^{4-n}\int_{B_r(x)}|F[\cL]|^2<\delta''\Big\}\, ,	\notag\\
&s'_x\equiv \inf_{s\leq 1}\Big\{\forall \delta^{3} s\leq r\leq 1 \text{ }B_r(x) \text{ is $\delta$-weakly flat}\Big\}\, ,\notag\\
&t_x\equiv \max\big\{s_x,s'_x\big\}\, .
\end{align}

Now for each $q\in \cL\cap B_1$ let $y_q\in \cL_q^\perp$ and $t_q\geq 0$ be defined by
\begin{align}
&t_q\equiv \min_{y\in \cL_q^\perp} t_y\, ,\	\notag\\
&y_q\in \arg\min_{y\in \cL_q^\perp} t_y\, .
\end{align}

Now let us define $\{B_{t_i}(y_i)\}$ be a maximal subset of $\{B_{t_q}(y_q)\}$ such that $\{B_{10^{-1}t_i}(y_i)\}$ are disjoint.  Let us decompose this collection into two subcollections:
\begin{align}
	\{B_{t_i}(y_i)\} = \{B_{t'_i}(y'_i)\}\cup \{B_{\tilde t_i}(\tilde y_i)\}\, ,
\end{align}
where $\{B_{t'_i}(y'_i)\}$ is the subcollection such that 
\begin{align}
t'_i	= \max\{s_{y'_i}, s'_{y'_i}\} = s'_{y'_i}\geq 10 \sup_{y\in B_{t'_i}(y'_i)}s_{y}\, ,
\end{align}
and $\{B_{\tilde t_i}(\tilde y_i)\}$ are the remaining balls.  Notice that if $\delta'<\delta'(n,\delta, \delta'')$ then by a standard maximal function argument we have the estimate
\begin{align}\label{e:annular_existence:1}
	\sum \tilde t_i^{\,\,n-4} \leq \delta^2\, .
\end{align}

Let us now consider the following claim:\\

{\bf Claim: }If $\delta''<\delta''(n,k,\Lambda,\delta)$ and $y\in B_{1000 t'_{i}}(y'_i)$, then $c(k)t'_i< s'_y\leq t_y$.  Further, if $y\in \{y_j\}$ is a ball center then we have the two sided estimate $c(k)t'_i< t_{y}< 10^4 t'_i$. \\

The upper bound $t_{y}< 10^4 t'_i$ in the case $y$ is a center point follows from the Vitali condition, therefore we will focus on the lower bound $c(k)t'_i< s'_y$ for general $y$, which will itself follow from Theorem \ref{t:weakly_flat_structure}.  Note by the definition of $t'_i$ we have that $B_{\delta t'_i}(y'_i)$ is $\delta$-weakly flat, but $B_r(y'_i)$ is not $\delta$-weakly flat for some $\frac{1}{2}\delta^{3} t'_i <r<\delta^{3} t'_i$.

Now let us look at the ball $B_{t'_i}(y)$.  If we assume $s'_y\leq c(k) t'_i$, then we have for all $c(k)\delta^{3} t'_i\leq r\leq c(k)^{-1}\delta^{3} t'_i$ that $B_{r}(y)$ is $\delta$-weakly flat.  Using Theorem \ref{t:weakly_flat_structure} with $\kappa=10^{-1}$ we have for $c(k)$ sufficiently small that $B_{r}(y)$ is $10^{-1}\delta$-weakly flat for all $10^{-1}\delta^{3} t'_i\leq r\leq 10\delta^{3} t'_i$.  In particular, for $\delta''\leq\delta''(n,k,\Lambda,\delta)$ we have that there must exist a point $z\in \cL^\perp_{y'_i}\cap B_{t'_i}$ such that $s'_z\leq 2^{-1}t'_{i}$.  However, by the definition of $B_{t'_i}(y'_i)$ we also have that $s_z\leq 10^{-1}t'_{i}$, and therefore $t_z\leq 2^{-1}t'_{i}$.  This contradicts that $t_{y'_i} = \min_{u\in\cL^\perp_{y'_i}}t_u$, and thus proves the Claim. $\square$\\

To continue with the proof of Theorem \ref{t:annular_existence} let us define the radii function
\begin{align}
r_x \equiv \begin{cases}
  \delta^2 t_i &\text{ if }x\in B_{\delta t_i}(y_i)	\, ,\notag\\
  \delta d(x,\{y_i\}) &\text{ if }x\notin \bigcup B_{\delta t_i}(y_i)\, .
 \end{cases}
\end{align}
Note that $|\Lip\, r_x|<\delta$.  Now we play a similar game as in the previous covering and define for each $q\in \cL\cap B_1$ the point $x_q\in \cL_q^\perp$ and $r_q\geq 0$ by
\begin{align}
&r_q\equiv \min_{x\in \cL_q^\perp} r_x\, ,	\notag\\
&x_q\in \arg\min_{x\in \cL_q^\perp} r_x\, .
\end{align}

Now for some $x_q$ let $y_i$ be the center point such that either $x_q\in B_{\delta t_i}(y_i)$ or $r_q = \delta d(x,y_i)$, so that in either case we have that $r_q\geq \delta^2 t_i$.  Since $B_r(y_i)$ is $\delta$-weakly flat for $\delta^3 t_i\leq r\leq c(k)^{-1}$ we have by Theorem \ref{t:weakly_flat_structure} the improved estimate that $B_{r}(y_i)$ is $10^{-1}\delta$-weakly flat for $c(k)\delta t_i\leq r\leq 1$.  In particular, we have that $B_{r}(y_i)$ is $10^{-1}\delta$-weakly flat for $r_q\leq r\leq 1$.  If $\delta''\leq \delta''(n,k,\Lambda,\delta)$ is sufficiently small we then conclude that $B_{r}(x_q)$ is itself $\delta$-weakly flat for $r_q\leq r\leq 1$.  

We now define our annular region by $\cC\subseteq \{x_q\}$ so that $\{B_{r_x}(x)\}_{x\in\cC}$ is a maximal subset of $\{B_{r_q}(x_q)\}_{q\in\cL}$ such that $\{B_{\tau^2 r_x}(r_x)\}$ are disjoint.  It is a straightforward, if somewhat tedious, exercise using the remarks of the previous paragraph to check for $\delta''\leq \delta''(n,k,\Lambda,\delta)$ that this defines a $\delta$-annular region.  We will focus then on the nontriviality of this annular region.  That is, if we consider the sets
\begin{align}
	&\cC^{c,1}\equiv \big\{
 	x\in \cC: B_{r}(x) \text{ is }\delta\text{-weakly flat for }\delta^4 r_x\leq r\leq r_x\, \big\}\, ,\notag\\
 	&\cC^{c,2}\equiv \big\{x\in\cC : \exists\, q\in\cL \text{ with }\cL^\perp_q\cap \bar B_{r_x}(x)\neq\emptyset \text{ and }\int_{\cL^\perp_q}|F_A|^2>\epsilon\big\}\, ,
\end{align}
then we want to see for $\delta'\leq \delta'(n,k,\Lambda,\delta)$ that we have the estimates $\mu\big(\cC^{c,1}\big)$, $\mu\big(\cC^{c,2}\big)<\frac{1}{2}\epsilon$.  We begin by estimating $\cC^{c,1}$.  To accomplish this let us consider $x\in \cC\cap B_{t'_i}(y'_i)$, then by using the two sided bound of the previous Claim, and particular that $c(k) t'_i\leq t_j\leq c(k)^{-1}t'_i$ for any other ball center in $B_{2t'_i}(y'_i)$, we have the estimate 
$$c(k)\delta\, t'_i\leq r_x\leq c(k)^{-1}\delta\, t'_i\, .$$

Additionally, we know by the Claim that $s_x\geq c(k)t'_i$, which is to say for some radius $r\geq \frac{1}{2}c(k)\delta^3 t'_i$ we must have that $B_r(x)$ is not $\delta$-weakly flat.  Combining this with the above estimate we see that for some radius $r\geq c(k)\delta^2 r_x$ that $B_r(x)$ is not $\delta$-weakly flat, which for $\delta<\delta(k)$ implies that $x\not\in \cC^{c,1}$, and in particular gives us the inclusion
\begin{align}
\cC^{c,1}\subseteq \bigcup B_{\tilde t_j}(\tilde y_j)\, .	
\end{align}
Finally, using \eqref{e:annular_existence:1} and the Ahlfor's regularity on $\mu$ proved in Section \ref{ss:annular_ahlfors_reg} we have the estimate
\begin{align}
\mu\big(\cC^{c,1}\big)\leq \sum \mu\big(B_{\tilde t_j}(\tilde y_j)\big)	\leq C(n)\sum \tilde t^{\, n-4}_j \leq C(n)\delta^2<\delta\, .
\end{align}
Now we focus on estimating $\cC^{c,2}$, which will itself depend on the curvature estimate of Theorem \ref{t:annular_region}.3.  So let us choose $\delta<\delta(n,k,\Lambda,\epsilon')$ such that Theorem \ref{t:annular_region} holds with $\epsilon'>0$.  Note then that we have
\begin{align}
\int_{B_{1}(0^\cL)}\int_{\cL^\perp_q\cap \cA}|F_A|^2 = \int_{\cA\cap B_1} |F_A|^2 <\epsilon'	\, .
\end{align}
Using the first claim of subsection \ref{ss:annular_ahlfors_reg}, where it is shown that the projection $\pi_\cL$ is uniformly bilipschitz on $\cC$, this implies that
\begin{align}
\int_{B_1}\Big(\fint_{B_{r_x}(\pi(x))}\int_{\cL^\perp_q\cap \cA} |F_A|^2 \Big)\, d\mu[x] < C(n)\epsilon'\, .	
\end{align}
In particular, if $\epsilon'\leq c(n)\epsilon^2$ then away from a set $\cC'\subseteq \cC\cap B_1$ with $\mu(\cC')<\frac{1}{2}\epsilon$ we have that $\fint_{B_{r_x}(\pi(x))}\int_{\cL^\perp_q\cap B_1} |F_A|^2<10^{-1}\epsilon$.  Now let us use that $B_r(x)$ is $\delta$-weakly flat for each $r>r_x$ combined with standard elliptic estimates to see that $|\nabla F_A|<C(n)\delta r^{-3}$ on $A_{r,r/2}(\cL_x)$.  This in particular gives us for each $x\in \cC\cap B_1$ that
\begin{align}
	\Big|\max_{q\in B_{r_x}(\pi(x))}\int_{\cL^\perp_q\cap \cA}|F_A|^2 - \fint_{B_{r_x}(\pi(x))}\int_{\cL^\perp_q\cap \cA}|F_A|^2\Big| <C(n)\delta\,r_x\, .
\end{align}
For $x\not\in \cC'$ this gives us that
\begin{align}
	\max_{q\in B_{r_x}(\pi(x))}\int_{\cL^\perp_q\cap \cA}|F_A|^2\leq  \fint_{B_{r_x}(\pi(x))}\int_{\cL^\perp_q\cap \cA}|F_A|^2+ C(n)\delta\,r_x<10^{-1}\epsilon+C(n)\delta^2<\epsilon\, .
\end{align}
Therefore we have that $\cC^{c,2}\subseteq \cC'$ and thus $\mu(\cC^{c,2})<\frac{1}{2}\epsilon$, which finishes the proof of the Theorem. $\square$
\vspace{.5cm}

\section{\texorpdfstring{Harmonic $\epsilon$-Gauge}{Harmonic epsilon-Gauge}}\label{s:eps_gauge}

Standard in any gauge problem is the need to choose a good coordinate system in order to study the equations.  In the context of Yang Mills the standard gauge condition one prefers is the Coulomb gauge.  Unfortunately, such a gauge will only exist locally and in general only when the underlying geometry is quite simple (e.g. when the curvature of the connection is small).

A key object of study in this paper are annular regions.  Annular regions $\cA\subseteq B_2$ are regions for which the connection looks very flat in a weak sense, but there is some curvature concentration on $B_2$ which is only visible on small scales inside the singular balls.  A Coulomb gauge will in general certainly not exist on the whole ball.

Instead in this section we will introduce a form of linearized Coulomb gauge associated to the induced vector bundle $E\to M$ coming from the orthogonal representation of $G\subseteq \SO(k)$.  This gauge will exist and solve an equation on the whole ball $B_2$, a point which will be important and useful in the analysis.  On the other hand, this linearized gauge will only form a a legitimate (vector bundle) gauge on part of the ball.  Recall that if $A$ is a Yang-Mills connection on $P$ then $E$ is equipped with a metric connection $\nabla_A$.  In particular, we have the associated Laplace operator $\Delta_A:\Gamma(E)\to\Gamma(E)$.  Let us begin by defining our notion of an $\epsilon$-gauge on $E$:\\

\begin{definition}\label{d:eps_gauge}
We say that sections $V^1,\ldots,V^k\in\Gamma(B_r,E)$ form a harmonic $\epsilon$-gauge on $B_r(x)$ if the following hold:
\begin{enumerate}
\item $\Delta V^a = 0$.
\item $|V^a|\leq 1+\epsilon$	.
\item $\fint_{B_r}|\langle V^a,V^b\rangle - \delta^{ab}|<\epsilon$.
\item $r^2\fint_{B_r}|\nabla V|^2<\epsilon^2$.
\end{enumerate}
\end{definition}

The goal of this section is to prove the existence of $\epsilon$-gauges on balls $B_2$ which admit annular regions $\cA\subseteq B_2$, and to prove that $V^a$ forms an actual vector bundle gauge over the whole annular region $\cA$.  The difficulty of this second statement is that apriori the sections $V^a$, which have bounded norm, may have norm tending to zero near the singular balls or may be becoming linearly dependent.  While one cannot say this doesn't happen, in fact in the whole ball it must happen, we will see that it cannot happen faster than at a small polynomial rate, and thus the sections remain a basis.  In fact, we will show something much stronger, we will see that for every $x\in\cC$ and $r\geq r_x$ that there exists a $k\times k$ matrix $T$ such that $T\circ V$ is an $\epsilon$-gauge on $B_r(x)$.  The idea for this is related to the ideas of \cite{ChNa_Codim4}.  We will also discuss some applications of these estimates which will be useful later in the paper.

In the next section we will tackle the more refined estimates on $\epsilon$-gauge's, which will tell us for {\it most} points $x\in\cC$ that for any ball $B_r(x)$ with $r\geq r_x$, the $V^a$ remain an $\epsilon$-gauge, even without transformation.  This result will be crucial in the proof of the energy identity and $L^1$ hessian estimate.  However, our first main result of this section is the following, which begins by showing the existence of harmonic $\epsilon$-gauge's on sufficiently symmetric balls:

\begin{theorem}\label{t:eps_gauge_existence}
Let $A$ be a stationary Yang-Mills connection with $\fint_{B_4} |F_A|^2 \leq \Lambda$ and $B_4(p)$ a $\delta$-weakly flat ball.  For each $\epsilon>0$ if $\delta<\delta(n,k,\Lambda,\epsilon)$, then there exists a harmonic $\epsilon$-gauge $V^a\in \Gamma(B_2,E)$.
\end{theorem}
\vspace{.5cm}

With the existence of an $\epsilon$-gauge established, we would like to understand how close to being a gauge the $V^a$ actually are.  The $L^2$ condition only concludes that the the $V^a$ form a gauge on a set of almost full measure.  This is actually pretty weak, and we would like to understand that the $V^a$ form a gauge on all of $\cA$.  In fact, it will be important for us to understand much more than this and have some effective control on the behavior of the $V^a$ at small scales.  Motivated by the transformation theorem which is to come, let us make the following definition:

\begin{definition}\label{d:transformation}
	Let $A$ be a stationary Yang-Mills connection on a $\delta$-annular region $\cA\equiv B_2\setminus \overline B_{r_x}(\cC)$, and let $V^a\in \Gamma\big(B_4,E\big)$ be an $\epsilon$-gauge.  Given $x\in \cC$ with $r_x\leq r\leq 2$ and $\theta\equiv 10^{-2}$ we define the $k\times k$ symmetric transformation matrix:
	\begin{align}
		T=T(x,r) = \Bigg(\fint_{W_r^\theta(x)} \langle V^a, V^b\rangle\Bigg)^{-1/2}\, .
	\end{align}
\end{definition}
\begin{remark}
Recall that the wedge regions are defined by $\cW^\theta(x)\equiv \big\{y\in B_2: d(y,\cL_x)\geq \cos(\theta)\, d(y,x)\big\}$ and $\cW^\theta_r(x) \equiv \cW^\theta(x)\cap A_{\cos\theta\, r,r/\cos\theta}(x)$.
\end{remark}
\begin{remark}
Apriori $T$ may have infinite eigenvalues as defined, however we will see in Theorem \ref{t:eps_gauge_transformation} below that this is not the case.
\end{remark}

\begin{remark}
For the sake of the theorems of this section one could have made the slightly simpler definition 	$T= \big(\fint_{B_{r}(x)} \langle V^a, V^b\rangle\big)^{-1/2}$, where one averages over a ball instead of a small portion of it.  Taking the average away from the singular set becomes important in the next section, when we try and control in a more refined manner the behavior of $T$.  In this case, if one were to average over all of $B_r$, then this adds small errors at every scale which may potentially pile up.
\end{remark}
\vspace{.25cm}

The following sums up the the use of the transformation matrices:

\begin{theorem}\label{t:eps_gauge_transformation}
Let $A$ be a stationary Yang-Mills connection on a $\delta$-annular region $\cA\subseteq B_2(p)$ satisfying \eqref{e:YM_assumptions2} and $\fint_{B_4} |F_A|^2 \leq \Lambda$, and let $V^a\in \Gamma(B_4,E)$ be a $\delta$-gauge.  For each $\epsilon>0$, if $\delta<\delta(n,\Lambda,\epsilon)$, then for all $B_{2r}(x)\subseteq B_2$ with $x\in \cC$ we have that
\begin{enumerate}
	\item $T(x,r)$ is nondegenerate, in fact $ 1-\epsilon\leq T\leq (1+\epsilon)r^{-\epsilon}$.
	\item $\tilde V^a \equiv (T\circ V)^a=T^a_b V^b$ is an $\epsilon$-gauge on $B_r(x)$.
\end{enumerate}
\end{theorem}
\vspace{.25cm}

The first immediate corollary is that $V^a$ is everywhere nondegenerate on $\cA$ and thus defines a vector bundle gauge.	 The second immediate corollary, which is itself not otherwise obvious, is that by using the uniform lower bound on $T(x,r)$ and standard elliptic estimates we have for each $x\in \cA_{10^{-6}}$ the pointwise estimates
\begin{align}\label{e:eps_gauge_pointwise_estimates}
&d(x,\cC)|\nabla V|(x) < C(n)\epsilon\, ,\notag\\
&d(x,\cC)^2|\nabla^2 V|(x) < C(n)\epsilon\, .	
\end{align}

Indeed, the fact that $|\nabla V|$ and $|\nabla^2 V|$ are scale invariantly bounded follows from the global $L^\infty$ estimate on $|V|$. However, the smallness of this bound is more subtle to prove, and will be directly used in Section \ref{s:eps_gauge_annular_estimates}.

\vspace{.25cm}

\subsection{Computation Properties and Basic Estimates of Harmonic Sections}\label{ss:computation_harmonic}

In this subsection we record some basic computation properties of harmonic sections, as well as some basic estimates over $\epsilon$-regularity regions.  Let us begin with the following computations for harmonic sections over a vector bundle $E$ equipped with a Yang Mills connection:

\begin{lemma}\label{l:harmonic_computation}
 Let $V\in\Gamma(B_r,E)$ satisfy $\Delta V=0$, where $E$ is equipped with a Yang Mills metric connection $\nabla_A$.  Then the following hold:
 \begin{enumerate}
 \item $\Delta |V|^2 = 2|\nabla V|^2 \geq 0$.
 \item $\Delta |\nabla V|^2 = 2|\nabla^2 V|^2 + 2F(\nabla V,\nabla V)+2Rc(\nabla V,\nabla V)$.	
 \end{enumerate}
\end{lemma}
\begin{remark}
We leave both as an easy exercise, but let us point out that the Yang-Mills condition plays a role in the second equation, as typically there is a $\text{div} F$ term which appears in this computation. 	
\end{remark}

\vspace{.25cm}

Let us now also record a basic estimate of $\epsilon$-gauges inside $\epsilon$-regularity regions:

\begin{lemma}\label{l:harmonic_reg_eps_reg}
 Let $V\in\Gamma(B_{2r},E)$ satisfy $\Delta V=0$, where $E$ is equipped with a Yang Mills metric connection $\nabla_A$.  Assume that we have the regularity scale estimate $r_A\geq 4r$.  Then for $k\geq 1$ we have the estimate $r^k\sup_{B_r}|\nabla^k V|\leq C(n,k)\,r\fint_{B_{2r}}|\nabla V|$.
\end{lemma}
\begin{proof}
	Let us briefly outline this because it uses the Yang-Mills conditions in two ways.  Since this is a scale invariant estimate we can assume $r=1$ without any loss.  First, as in Lemma \ref{l:harmonic_computation} we may compute $\Delta \nabla V = F(\nabla V)$ using the Yang-Mills condition, and thus on $B_{3}(x)$ we have the inequality
	\begin{align}
	\Delta|\nabla V|\geq -C(n)|\nabla V|\, ,	
	\end{align}
	so that $|\nabla V|$ satisfies a subharmonic inequality.  We may therefore use the mean value theorem for nonnegative functions satisfying the above in order to conclude that $\sup_{B_{5/2}}|\nabla V|\leq C(n)\fint_{B_{2}}|\nabla V|$.  Now we may use standard elliptic estimates on the equation $\Delta \nabla V = F(\nabla V)$ in order to conclude the result.
\end{proof}

\vspace{.25cm}

\subsection{Proof of \texorpdfstring{$\epsilon$}{epsilon}-Gauge Existence of Theorem \ref{t:eps_gauge_existence}}

We will prove the result by contradiction.  Indeed, let us assume for some $\epsilon>0$ that no such $\delta(n,k,\Lambda,\epsilon)$ exists.  Then we can find a sequence of Yang-Mills connections $A_i$ with $B_4(p_i)$ $\delta_i$-weakly flat balls with $\delta_i\to 0$ for which there does exist an $\epsilon$-gauge $V^a$ on $B_2$.\\

After possibly composing with a rotation, there is no harm in assuming each weakly flat ball is with respect to the $n-4$ plane of symmetry $\cL_{i}\equiv \cL\equiv \dR^{n-4}\times\{0\}$.  Note that since $\delta_i\to 0$, we have that the underlying manifolds are converging
\begin{align}
	B_{4}(p_i)\stackrel{C^{1,\alpha}}{\longrightarrow} B_4(0)\subseteq\dR^n\, .
\end{align}
Additionally, we have on $B_{4}\setminus B_{\delta_i}(\cL)$ that the curvature satisfies $|F_i|\to 0$.  Thus if we restrict the bundles $E_i\to B_{\delta_i^{-1}}\setminus B_{\delta_i}(\cL)$ then after passing to a subsequence we can limit
\begin{align}
&E_i\to E\, ,\notag\\
&A_i\to A\, ,	
\end{align}
where $E\to \dR^n\setminus \cL$ and $A$ is a flat connection on $E$.  In particular, we can pick global parallel sections $E^1,\ldots, E^k$ on $E$, and thus canonically extend $E$ to the trivial bundle $\dR^k\times \dR^n\to \dR^n$.  

Let us consider the convergence $B_{\delta^{-1}_i}\to \dR^n$ and $E_i\to E$ in slightly more detail, as it will be useful in the coming analysis.  Specifically, our convergence tells us that for all $i$ sufficiently large we may find a $C^{2,\alpha}$ diffeomorphism and $G$-bundle maps such that the following diagram commutes: 
\begin{align}
\xymatrix{
E_i\ar[rr]^{\varphi_E}\ar[d]_{\pi_i} &   & E \ar [d]^{\pi}\\
 B_{\delta_i^{-1}}(p_i)\setminus B_{\delta_i}(\cL)\ar[rr]^{\varphi} && \dR^n\setminus B_{\delta_i}(\cL)}	
\end{align}

Let us now define the sections $V^1_i,\ldots, V^k_i\in \Gamma(E_i,B_3)$ as the unique harmonic sections $\Delta V^a_i = 0$ in $B_3(p_i)	$ which satisfy the Dirichlet boundary values

\begin{align}
V^a_i(x) = \begin{cases}
 	\varphi_E^*E^a & \text{ if }x\in  \partial B_{3}(p_i)\setminus B_{\delta_i}(\cL)\, ,\notag\\
 	0 & \text{ if }x\in  \partial B_{3}(p_i)\cap B_{\delta_i}(\cL)\, .
 \end{cases}
\end{align}

Now let us see that for $i$ sufficiently large this defines our $\epsilon$-gauge.  Let us begin by showing the pointwise estimate $|V^a_i|\leq 1$.  Indeed, we have on $\partial B_3$ that $|V^a_i| \leq |\varphi_E^*E^a|\equiv 1$.  On the other hand, by Lemma \ref{l:harmonic_computation} we have that $|V|$ is a subharmonic function.  Therefore by a maximum principle we have the estimate $|V^a_i|\leq 1$ on all of $B_3$.\\

We will want to show the gradient estimate $\fint_{B_4} |\nabla V^a_i|^2<\epsilon^2$.  Let us begin by proving the related, but apriori weaker, estimate that if $B_{8r}(x)\subseteq B_3$ then $r^2\fint_{B_r} |\nabla V^a_i|^2<C(n)$.  Indeed, let $\phi:B_{8r}(x)\to \dR$ be a smooth cutoff function with 
\begin{align}\label{e:harmonic_existence:cutoff}
&\phi\equiv 1 \text{ on }B_{4r}\, ,\;\;\;\;	\phi\equiv 0 \text{ outside }B_{7r}\, ,\notag\\
&r\,|\nabla \phi|\, , r^2\,|\nabla^2 \phi|<C(n)\, .
\end{align}
Thus we can compute
\begin{align}\label{e:harmonic_existence:1}
\fint_{B_{4r}}|\nabla V^a_i|^2  &\leq C(n)r^{-n}\int \phi |\nabla V^a_i|^2 = C(n)r^{-n}\int \phi \Delta |V^a_i|^2 \notag\\
&= C(n)r^{-n}\int \Delta\phi |V^a_i|^2\leq   C(n)r^{-2}\sup_{B_{7r}}|V^a_i|^2\leq C(n)r^{-2}\, ,
\end{align}
where in the middle equality we have used Lemma \ref{l:harmonic_computation}.1.

The global $L^\infty$ bound on $V^a_i$ combined with the curvature bound $|F_i|\to 0$ on $B_{4}(p_i)\setminus B_{\delta_i}(\cL)$ tells us that after passing to a subsequence we can limit
\begin{align}
V^a_i\to V^a\in \Gamma(E,B_3(0^n)\setminus \cL)\, ,	
\end{align}
where the convergence is in $C^{1,\alpha}$ on compact subsets of $\overline{B_3}(0^n)\setminus \cL$.  Using $|V^a_i|\leq 1$, $\Delta V^a_i=0$ and \eqref{e:harmonic_existence:1} we therefore obtain
\begin{align}
	&|V^a|\leq 1 \, ,\notag\\
	& V^a\in H^1_{\text{loc}}(B_3)\, ,\notag\\
	&\Delta V^a = 0 \text{ on }B_3\setminus \cL\, .
\end{align}
Note that we have already identified $E$ as the trivial $\dR^k$ bundle over $\dR^n$, and thus we may view $V^a:B_3\setminus \cL\to \dR^k$ canonically.\\

Using the $H^1_{\text{loc}}$ estimate above and that the set $\cL$ has capacity zero we can therefore conclude that $V^a$ extends a smooth solution of $\Delta V^a=0$ over all $B_3$.  On the other hand, since the convergence is smooth on compact subsets of $\overline{B_3}\setminus \cL$ we know that $V^a = E^a$ on $\partial B_3$.  The $E^a$ are themselves harmonic, so by uniqueness if $V^a=E^a$ on $\partial B_3$ then $V^a=E^a$ on all of $B_3$. Thus we have concluded that
\begin{align}
V^a_i \to E^a\, ,	
\end{align}
where the convergence is $C^{1,\alpha}$ on compact subsets of $\overline{B_3}\setminus \cL$.  In particular, we have that
\begin{align}
	\langle V^a_i,V^b_i\rangle \to \delta^{ab} \text{ pointwise on } B_3\setminus \cL\, .
\end{align}
Using the pointwise bound $|V^a_i|\leq 1$ and that $\Vol(B_3\cap B_r(\cL))\to 0$ as $r\to 0$ one therefore easily concludes the $L^2$ almost orthogonality
\begin{align}
	\fint_{B_3} \big|\langle V^a_i,V^b_i\rangle - \delta^{ab} \big| \to 0\, .
\end{align}

Finally, if we can show that $\fint_{B_2}|\nabla V^a_i|^2 \to 0$, then we will proved that for all $i$ sufficiently large $V^a_i$ forms an $\epsilon$-gauge, which is our desired contradiction.  Thus, for any $x\in B_2$ let us consider the ball $B_{r}(x)=B_{1/16}(x)$ with $B_{8r}\subseteq B_3$, and let us consider the cutoff function $\phi$ from \eqref{e:harmonic_existence:cutoff}.  Then we can make a slightly more refined version of the previous computation in order to conclude
\begin{align}
	\fint_{B_r(x)}|\nabla V^a_i|^2  &\leq C(n)\int \phi |\nabla V^a_i|^2 = C(n)\int \phi \Delta |V^a_i|^2 \notag\\
	&=C(n)\int \phi \Delta \Big(|V^a_i|^2-1\Big)=C(n)\int \Delta \phi \Big(|V^a_i|^2-1\Big)\notag\\
	&\leq C(n)\fint_{B_3}\big| |V^a_i|^2-1\big| \to 0\, .
\end{align}
Since this hold for any ball $B_{1/16}(x)$ with $x\in B_2$ we have $\fint_{B_2}|\nabla V^a_i|^2\to 0$ as claimed, which proves our desired contradiction and thus proves the theorem. $\square$
\vspace{.5cm}

\subsection{Proof of Transformation Theorem \ref{t:eps_gauge_transformation}}

The proof will be by contradiction.  Therefore let us assume for some $\epsilon>0$ the result fails, and thus we can find a sequence of $\delta_i$-annular regions $\cA_i\subseteq B_2(p_i)$ with $\delta_i$-gauges $V_i^1,\ldots, V^k_i\in \Gamma(B_4(p_i),E_i)$ such that the result fails for each $i$ with $\delta_i\to 0$.  Let us choose $x_i\in \cC_i$ to be one of the points for which the result fails, and define
\begin{align}
r_i\equiv \min\{r_{x_i}\leq r<2: \forall\; r\leq s\leq 2 \text{ we have that $(1)$ and $(2)$ hold on } B_s(x_i)\}\, .	
\end{align}
We have by assumption that $r_i>r_{x_i}$, and therefore either $(1)$ or $(2)$ must fail for some radius $r>\frac{1}{2}r_i$.  Notice that $r_i\to 0$ since $\delta_i\to 0$.  Let $T_i\equiv T(x_i,r_i)$ and let us denote $\tilde V^a_i \equiv T_i\circ V_i$ to be the transformed sections.  Since the result holds for $r_i$ we have that $\tilde V^a_i$ is an $\epsilon$-gauge on $B_{r_i}(x_i)$, but note that we cannot have that $\tilde V^a_i$ is a $10^{-2n}\epsilon$-gauge on $B_{r_i}(x_i)$.  If this were to hold, then it is clear that $(1)$ and $(2)$ must still be satisfied for all $r\geq \frac{1}{2} r_i$, which is not the case.

Therefore let us rescale the geometry so that $B_{r_i}(x_i)\to B_2(\tilde x_i)$, so that $\tilde V^a_i\in \Gamma(B_{r_i^{-1}}(\tilde x_i),E_i)$ is an $\epsilon$-gauge on $B_2(\tilde x_i)$.  After rotation we may assume the best plane $\cL_i\equiv \cL$ for each annular region $\cA_i$ is a constant.  Let us begin with the following claim:\\

{\bf Claim: }For each $1\leq r\leq r_i^{-1}$ and $0<s<1$ we have the estimates
\begin{enumerate}
\item $\sup_{B_r} |\tilde V_i|\leq C(n) r^\epsilon$.
\item $r^2\fint_{B_r(\tilde x_i)} |\nabla \tilde V^a_i|^2 \leq C(n) r^{2\epsilon}$.	
\item $\int_{B_r\cap B_s(\cL)} |\nabla \tilde V^a_i|^2 \leq C(n) r^{n-4+2\epsilon} s^2$.
\end{enumerate}

To prove the claim let us observe that for every $1\leq r\leq r_i^{-1}$ we have by the definition of $r_i$ that if $\tilde T \equiv \Big(\fint_{W^\theta_{r}(\tilde x_i)}\langle \tilde V^a_i,\tilde V^b_i\rangle\Big)^{-1/2}$, then by condition $(2)$ we have that $\tilde T\circ \tilde V^a_i$ is an $\epsilon$-gauge on $B_r(x_i)$.  In particular,
\begin{align}\label{e:transformation:1}
	\sup_{B_r}|\tilde T\circ \tilde V^a_i|\leq 1+\epsilon\, .
\end{align}
However, by condition $(1)$, rescaled since our original ball $B_{r_i}$ now has radius $2$, we have the estimate
\begin{align}
	\frac{1}{2} r^{-\epsilon}\leq \tilde T\leq 2 r^{-\epsilon}\, .
\end{align}
Plugging this into \eqref{e:transformation:1} we arrive at the first estimate of the claim.  
The second and third estimates are proved by a verbatim argument, so let us focus on the second.  Choose a cutoff function $\phi$ so that $\phi\equiv 1$ on $B_r(\tilde x_i)$, $\phi\equiv 0$ outside of $B_{2r}(\tilde x_i)$ and $r\,|\nabla\phi|$, $r^2\,|\nabla^2\phi|\leq C(n)$.  Multiplying both sides of Lemma \ref{l:harmonic_computation}.1 by $\phi$ and integrating we arrive at
\begin{align}
	\fint_{B_r}|\nabla \tilde V^a_i|^2  &\leq C(n)r^{-n}\int \phi |\nabla \tilde V^a_i|^2 = C(n)r^{-n}\int \phi \Delta |\tilde V^a_i|^2 \notag\\
	&=C(n) r^{-n}\int \Delta \phi |\tilde V^a_i|^2\notag\\
	&\leq C(n)r^{-2}\sup_{B_{2r}}|\tilde V^a_i|^2 \leq C(n) r^{-2}r^{2\epsilon}\, ,
\end{align}
which finishes the proof of the claim. $\square$\\

Now as in the proof of Theorem \ref{t:eps_gauge_existence} using that $\delta_i,r_i\to 0$ we may pass to a subsequence in order to limit our spaces 
\begin{align}
	&B_{r_i^{-1}}(\tilde x_i)\to \dR^n\, ,\notag\\
	&E_i\to E\equiv \dR^k\times \dR^n\, ,
\end{align}
where $E$ is apriori a flat bundle over $\dR^n\setminus \cL$ which may be canonically extended to the trivial bundle $\dR^k\times \dR^n$.  Using the estimates of the previous claim we may also pass to a subsequence to also limit
\begin{align}
\tilde V^a_i \to \tilde V^a \in \Gamma\big(E,\dR^n\big)\, ,
\end{align}
where the convergence is smooth on $\dR^n\setminus \cL$.  By using $(3)$ of the claim and that $\tilde V^a_i\to \tilde V^a$ smoothly on $\dR^n\setminus \cL$ we also see that
\begin{align}\label{e:transformation:2}
\fint_{B_r(x)} |\nabla \tilde V^a_i|^2 \to \fint_{B_r(x)} |\nabla \tilde V^a|^2\, ,
\end{align}
so that the $H^1$ norms converge.

Now using that $E$ is flat and trivial we may view $\tilde V^a:\dR^n\to \dR^k$, and by the estimates of the previous Claim we have
\begin{enumerate}
\item $\sup_{B_r} |\tilde V|\leq C(n) r^\epsilon$.
\item $r^2\fint_{B_r(x_i)} |\nabla \tilde V|^2 \leq C(n) r^\epsilon$.	
\end{enumerate}

Since the convergence of $\tilde V^a_i\to \tilde V$ is smooth on $\dR^n\setminus \cL$ we have that $\tilde V$ is harmonic on $\dR^n\setminus \cL$.  However, since $\tilde V\in H^1_{\text{loc}}$ by $(2)$ above and since $\cL$ is a set with zero capacity, we have that $\tilde V^a$ extends to a smooth harmonic function on all of $\dR^n$.

However, since $|\tilde V|$ is growing at most at a small polynomial rate, we have by Liouville's theorem that
\begin{align}
|\nabla \tilde V^a|\equiv 0\, .	
\end{align}
In particular, we have that $\langle \tilde V^a,\tilde V^b\rangle = constant$ for each $a,b$ .  However, by construction we also have that
\begin{align}
	\fint_{W^\theta_2(0)} \langle \tilde V^a,\tilde V^b\rangle = \delta^{ab}\, .
\end{align}
Combining these points we get that
\begin{align}
	\langle \tilde V^a,\tilde V^b\rangle = \delta^{ab}\, ,
\end{align}
on all of $\dR^n$.  In particular, $\tilde V^a$ is a $0$-gauge.  Recall now that while $\tilde V^a_i$ is an $\epsilon$-gauge on $B_2$, by the construction of $r_i$ it is not a $10^{-2n}\epsilon$-gauge.  However, using \eqref{e:transformation:2} and that $\tilde V^a$ is a $0$-gauge we  see that for $i$ sufficiently large this is our desired contradiction, and thus we have proved the Theorem. $\square$\\

\section{\texorpdfstring{$\epsilon$}{epsilon}-Gauge's on Annulus Regions}\label{s:eps_gauge_annular_estimates}

In the previous section we showed the existence of $\epsilon$-gauge's on $\delta$-annular regions and proved some basic estimates.  In this section we study more carefully the properties of such $\epsilon$-gauge's and prove our main analytic estimates.  There are two main results we wish to prove and discuss in this section.  The first is that we will see that on most of the $\delta$-annular region our $\epsilon$-gauge is a legitimate vector Coulomb gauge which is $\epsilon$-orthonormal.  Precisely:\\

\begin{theorem}\label{t:eps_gauge_on}
	Let $A$ be a stationary Yang-Mills connection on a $\delta$-annular region $\cA = B_2\setminus \overline B_{r_x}(\cC)$ and $\fint_{B_2} |F_A|^2 \leq \Lambda$, and let $V$ be a $\delta$-gauge on $B_4$.  For each $\epsilon>0$ if $\delta<\delta(n,k,\Lambda,\epsilon)$ then there exists a subset $\cC_\epsilon\subseteq \cC\cap B_1$ such that
	\begin{enumerate}
	\item $\mu\big(\cC_\epsilon\big)\geq (1-\epsilon)\mu\big(\cC\cap B_1\big)$.	
	\item For each $x\in \cC_\epsilon$ and $r_x\leq r\leq 1$ we have that $V$ is an $\epsilon$-gauge on $B_r(x)$.
	\end{enumerate}
\end{theorem}
\vspace{.5cm}

In fact the above result will eventually follow from the scale invariant gradient estimate discussed in the next theorem, which is where most of the work of this section focuses.  Precisely, we have the following:\\

\begin{theorem}\label{t:eps_gauge_gradient}
	Let $A$ be a stationary Yang-Mills connection on a $\delta$-annular region $\cA = B_2\setminus \overline B_{r_x}(\cC)$ and $\fint_{B_2} |F_A|^2 \leq \Lambda$, and let $V$ be a $\delta$-gauge on $B_4$.  For each $\epsilon>0$ if $\delta<\delta(n,k,\Lambda,\epsilon)$ then we have the following estimates:
	\begin{align}
		&\int_{\cA_{10^{-4}}\cap B_{3/2}} r_A^{-3}|\nabla V|\,,\,\,\, \int_{\cA_{10^{-4}}\cap B_{3/2}} r_A^{-2}|\nabla^2 V| \, ,\,\,\, \int_{\cA_{10^{-4}}\cap B_{3/2}} |\nabla^4 V| < \epsilon\, .\notag\\
	\end{align}
\end{theorem}
\vspace{.25cm}

\subsection{Annular Green's Function}\label{ss:annulus_greens}

In this subsection we introduce and study the Green's function $G_\cA$ associated to an annular region.  Recall that $G_x(y) \sim \alpha_n |x-y|^{2-n}$ is the standard Green's function, which is the solution of $-\Delta G_x = \delta_x$. Since we work under the assumption that $K\leq \delta$, by \eqref{e:YM_assumptions} and standard estimates, we have that there exists a constant $C$ for which $C^{-1} d(x,y)^{2-n} \leq G_x(y)\leq C d(x,y)^{2-n}$, and also $C^{-1}d(x,y)^{1-n}\leq \abs{\nabla G_x(y)}\leq C d(x,y)^{1-n}$. The annular version of the Green's function $G_\cA$ satisfies the following:\\

\begin{definition}
	Let $\cA = B_2\setminus \overline B_{r_x}(\cC)$ be a $\delta$-annular region with packing measure $\mu$.  Then we define:
\begin{enumerate}
\item The annular Green's function $G_\cA(y) \equiv \int G_x(y)\,d\mu[x]$, which is the global solution of $-\Delta G_\cA = \mu$.	
\item The annular distance function $b(y)=b_\cA(y)$ which is defined by the formula $G_\cA \equiv  b^{-2}$.
\end{enumerate}
\end{definition}
\vspace{.25cm}

Notice in the above that if one viewed $\cA$ as a perfect annulus $\cA\equiv B_2(0^n)\setminus \dR^{n-4}$, then $b_\cA(y) \propto d(y,\dR^{n-4})$ would be the distance to the singular set.  Therefore $b$ is our smooth approximation to such a distance.  Let us see that this is a fair interpretation in the general case:\\

\begin{lemma}\label{l:annular_distance}
Let $A$ be a stationary Yang-Mills connection on a $\delta$-annular region $\cA = B_2\setminus \overline B_{r_x}(\cC)$ satisfying $\fint_{B_2} |F_A|^2 \leq \Lambda$, and let $b(y)$ be the annular distance function.  Then if $\delta<\delta(n,k,\Lambda)$ then there exists $C(n)$ such that the following hold:
\begin{enumerate}
\item $C^{-1}\,d(y,\cC)<b(y)<C\,d(y,\cC)$ for all $y\in \cA_{10^{-6}}$.
\item $C^{-1}<|\nabla b|< C$ on $\cA_{10^{-6}}$.
\end{enumerate}
\end{lemma}
\begin{remark}
Recall that $\cA_s\equiv B_2\setminus \overline B_{s\cdot r_x}(\cC)$ is the extended annulus.	
\end{remark}

\begin{proof}
Let $y\in \cA_{10^{-6}}$ with $x\in\cC$ the closest point of $\cC$ to $y$ and $r\equiv 2d(x,y)$.  Let us denote the sequence of scales $r_\alpha\equiv 2^\alpha r$, then we can write
\begin{align}
G_\mu(y) = \int G_{z}(y)\,d\mu[z] \sim \int_{B_{r}(x)} G_{z}(y)\,d\mu[z]+ \sum_{\alpha\geq 0} \int_{A_{r_{\alpha+1},r_\alpha}(x)} G_{z}(y)\,d\mu[z]	\, .
\end{align}
	Thus our upper and lower bounds are derived from the estimates
\begin{align}
&G_\mu(y)\leq C(n)\sum_\alpha r_\alpha^{2-n}\cdot r_\alpha^{n-4} = C(n)\sum 2^{-2\alpha} r^{-2}\leq C(n) r^{-2} = C(n)d(x,\cC)^{-2}\, ,\notag\\
&G_\mu(y)\geq \int_{B_{r}(x)} \alpha_n d^{2-n}(y,z)\,d\mu[z]\geq C(n)^{-1}r^{2-n}r^{n-4} = C(n)^{-1}r^{-2}\, ,
\end{align}
where we have used the Ahlfor's upper bounds proved in theorem \ref{t:annular_region}.

The upper bound on the gradient estimate is proved similarly with
\begin{align}
|\nabla G_\mu|(y) &\leq 	C(n)\int_{B_{r}(x)} \alpha_n d^{1-n}(y,z)\,d\mu[z]+ C(n)\sum_{\alpha\geq 0} \int_{A_{r_{\alpha+1},r_\alpha}(x)} \alpha_n d^{1-n}(y,z)\,d\mu[z]\notag\\
&\leq C(n)\sum r_\alpha^{1-n}\cdot r_{\alpha}^{n-4}\leq C(n) r^{-3} = C(n) d(x,\cC)^{-3}\, .
\end{align}

The lower bound on the gradient takes a little bit more work.  Let us consider the radial vector at $y$ given by $v=\nabla d(x,y)$.  Note that for every $z\in \cC\cap B_r(x)$ that $\nabla_v G_{z}(y)> C^{-1}d(z,y)^{1-n}$.  Further, by condition $(a3)$ we have that for every $\beta>0$ if $\delta<\delta(\beta)$ then for every $z\in B_{r_\beta}(x)\cap \cC$ we have that $\nabla_v G_{z}(y)>0$.  In particular, if this holds for a given $\beta$ then we can estimate
\begin{align}
|\nabla G_\mu|(y) \geq \nabla_v G_\mu(y)\geq &\int_{B_{r}(x)} \nabla_v G_{z}(y)\,d\mu[z]+ \sum_{0\leq \alpha\leq \beta} \int_{A_{r_{\alpha+1},r_\alpha}(x)} \nabla_v G_{z}(y)\,d\mu[z]\notag\\
 &+ \sum_{\beta+1\leq \alpha} \int_{A_{r_{\alpha+1},r_\alpha}(x)} \nabla_v G_{z}(y)\,d\mu[z]\notag\\
 &>2C(n)^{-1} r^{1-n}r^{n-4} - C(n)\sum_{\alpha\geq \beta+1} r_\beta^{1-n}r_\beta^{n-4}\, ,\notag\\
 &\geq \big(2C(n)^{-1} - C(n) 2^{-\beta}\big)r^{-3}\, .	
\end{align}
Thus if $\beta=\beta(n)$ then we obtain the estimate
\begin{align}
	|\nabla G_\mu|(y) \geq C(n)^{-1} r^{-3}\, ,
\end{align}
which completes the proof of the Lemma.
\end{proof}

\vspace{.25cm}

The following straightforward but useful computations are at the heart of what we will use the annular distance functions for:\\

\begin{lemma}\label{l:Green_annular_computation}
	Let $\cA = B_2\setminus \overline B_{r_x}(\cC)$ be a $\delta$-annular region with packing measure $\mu$ and annular distance function $b(y)$.  If $f$ is a smooth function let us define $S(r)=r\cdot r^{-3}\int_{b=r} f |\nabla b|$.  Then we have
\begin{align}
		r\frac{d}{dr}\Big(r\frac{d}{dr}S\Big) =S + \int_{b=r} \Delta f \,|\nabla b|^{-1}\, .
\end{align}
\end{lemma}
\begin{proof}
Let us note that $\Delta b = \frac{3}{b}|\nabla b|^2$ and that the mean curvature of the level set $b=r$ is given by 
\begin{align}
H_{b=r} = \text{div}\Big(\frac{\nabla b}{|\nabla b|}\Big) = 3\frac{|\nabla b|}{b}-\frac{\langle \nabla b,\nabla|\nabla b|\rangle}{|\nabla b|^2}\, .
\end{align}
One can then compute
\begin{align}
r\frac{d}{dr} S = S+r^{-1}\int_{b\leq r} \Delta f\, .	
\end{align}
Applying $r\frac{d}{dr}$ again leads to the result.
\end{proof}

\vspace{.25cm}

\subsection{Super Convexity for Scale Invariant \texorpdfstring{$L^1$}{L1} Gradient}

In this subsection we derive a superconvexity estimate for the gradient $|\nabla V|$ of an $\epsilon$-gauge on an annular region.  This estimate will turn out to be the key technical tool in the proofs of Theorem \ref{t:eps_gauge_gradient} and Theorem \ref{t:eps_gauge_on}.  

We will begin by defining a convenient cutoff function associated to an annular region.  For this let $0\leq \phi(s)\leq 1$ be a fixed smooth cutoff with $\phi\equiv 0$ for $|s|\leq \frac{8}{10}$, $\phi\equiv 1$ for $|s|\geq \frac{9}{10}$ and with the estimates $|\frac{d^k}{ds^k}\phi|\leq C(k)$.  For each $x\in B_2$ and $0<r\leq 10$ we can then define $\phi_{x,r}(y)\equiv \phi\big(r^{-2}d^2(x,y)\big)$.  Associated to an annular region $\cA\subseteq B_2(p)$ we then define the cutoff
\begin{align}\label{e:superconvexity_cutoff}
	\phi_\cA(y)\equiv (1-\phi_{p,2})(y)\cdot\prod_{x\in\cC} \phi_{x,10^{-5}r_x}(y) = (1-\phi_{p,2}(y))\cdot \tilde \phi_\cA(y)\, .
\end{align}
Using $(a1)\to (a4)$ and Remark \ref{r:annular_intersection} it is easy to check the following properties of the cutoff
\begin{align}\label{e:superconvexity_cutoff:estimates}
    &\phi_\cA \equiv 1 \text{ in } \cA_{10^{-5}}\, ,\notag\\
	&\text{supp}\,|\nabla\tilde \phi_\cA| \subseteq B_{10^{-5}r_x}(\cC)\setminus B_{10^{-6}r_x}(\cC)\subseteq \cA_{10^{-6}}\, ,\notag\\
	&|\nabla^{(k)} \phi_\cA|\leq C(n,k)r_x^{-k} \text{ in each }B_{r_x}(x)\, .  
\end{align}

Let us begin with the main computation of this subsection:\\

\begin{proposition}\label{p:super_convexity}
Let $A$ be a stationary Yang-Mills connection on a $\delta$-annular region $\cA = B_2\setminus \overline B_{r_x}(\cC)$ satisfying $\fint_{B_2} |F_A|^2 \leq \Lambda$, and let $V$ be a $\delta$-gauge on $B_4$. Let us define the scale invariant quantity $S(r) = r\cdot r^{-3} \int_{b=r} |\nabla V|\,\varphi_\cA\, |\nabla b|$, then for each $\epsilon>0$ if $\delta<\delta(n,k,\epsilon,\Lambda)$ then:
\begin{align}
r\frac{d}{dr}\Big(r\frac{d}{dr}S\Big) \geq (1-\epsilon)^2S - e(r)\, ,
\end{align}
where $|e(r)| \leq  \epsilon\,\mu\big(\{x: C^{-1}r\leq r_x\leq Cr\}\big) + C r\ $ for $C=C(n)$.
\end{proposition}
\vspace{.25cm}
\begin{proof}
Let us consider $\epsilon'>0$, which will eventually be chosen by $\epsilon'=\epsilon'(n, \epsilon)$.  %Noting that $\frac{d}{dt}S = r\frac{d}{dr}S$ 
We can use Lemma \ref{l:Green_annular_computation} in order to compute
\begin{align}
r\frac{d}{dr}\Big(r\frac{d}{dr}S\Big)  = S+\int_{b=r} \Big(\Delta |\nabla V|\, \phi + 2\langle\nabla|\nabla V|,\nabla\phi\rangle +|\nabla V|\Delta \phi\Big)\,|\nabla b|^{-1}\, .
\end{align}
	Using Lemma \ref{l:harmonic_computation} and that our cutoff satisfies \eqref{e:superconvexity_cutoff} $\text{supp}\,\phi_\cA\subseteq \cA_{10^{-6}}$ we have that
\begin{align}
	\Delta |\nabla V|\, \phi \geq -C(n)\delta\, d(x,\cC)^{-2} |\nabla V|\phi\, .
\end{align}
If we combine this with Lemma \ref{l:annular_distance} and integrate we arrive at
\begin{align}
\int_{b=r} \Delta |\nabla V|\, \phi \geq -C(n)\delta\, b^{-2}\int_{b=r}|\nabla V|\,\phi \geq -\epsilon\, S\, .
\end{align}

In order to estimate the other error terms let us recall that we can write $\phi_\cA = \phi_{p,2}\cdot \tilde \phi_\cA$ as in \eqref{e:superconvexity_cutoff}.  Using Theorem \ref{t:eps_gauge_transformation} with $\epsilon'>0$ together with \eqref{e:eps_gauge_pointwise_estimates} and \eqref{e:superconvexity_cutoff:estimates} we can therefore write
\begin{align}
2|\langle\nabla|\nabla V|,\nabla\phi_\cA\rangle|(y)\leq C(n)\epsilon' b^{-2}+C(n)\epsilon' \sum_{x\in \cC\cap B_{18/10}} r_x^{-3}\chi[B_{C(n) r_x}(x)\cap \{C(n)^{-1}r_x<b(y)<C(n) r_x\}]\, ,\notag\\
|\nabla V|\,|\Delta\phi|(y) \leq C(n)\epsilon'b^{-1}+C(n)\epsilon' \sum_{x\in \cC\cap B_{18/10}} r_x^{-3}\chi[B_{C(n) r_x}(x)\cap \{C(n)^{-1}r_x<b(y)<C(n) r_x\}]\, .
\end{align}
Integrating these last two terms and using Lemma \ref{l:annular_distance} we arrive at
\begin{align}
&\int_{b=r}2|\langle\nabla|\nabla V|,\nabla\phi\rangle| \leq C(n)\epsilon'b+ C(n)\epsilon'\sum_{x\in B_{18/10}\cap \{r_x\in [C^{-1}r,Cr]\}} r^{n-4} \leq  C(n)\epsilon'\mu\big(\{C^{-1}r< r_x< C r\}\big)\, ,\notag\\
&\int_{b=r} |\nabla V|\,|\Delta\phi|\leq C(n)\epsilon'b^2+ C(n)\epsilon'\mu\big(\{C^{-1}r< r_x< C r\}\big)\, .
\end{align}
Choosing $\epsilon' = C(n)^{-1}\epsilon$ we have arrived at our conclusion.
\end{proof}

\subsection{Dini Estimates and Superconvexity}

In order to exploit Proposition \ref{p:super_convexity} we will apply a maximum principle and study solutions of the underlying superconvex equation.  The following tells us how to estimate the Dini integral of solutions:

\begin{proposition}\label{p:dini_ode}
For each $R>0$ and $e(r)$ the solution of $r\frac{d}{dr}\Big(r\frac{d}{dr}\bar S\Big)= (1-\epsilon)^2 \bar S - e(r)$ with $\bar S(0)=\bar S(R)=0$ satisfies the Dini estimate
\begin{align}
	\int_0^R \frac{\bar S(r)}{r} \leq \frac{1}{(1-\epsilon)^2}\int_0^R \frac{e}{s} + \frac{R^{-1+\epsilon}}{(1-\epsilon)^2}\int_0^R s^{-\epsilon}\, e\, .
\end{align}
\end{proposition}
\begin{proof}
	Observe that $r^{1-\epsilon}$ and $r^{-1+\epsilon}$ are solutions to the homogeneous equation.  With this one can check that an explicit solution to $r\frac{d}{dr}\Big(r\frac{d}{dr}\bar S\Big)= (1-\epsilon)^2 \bar S - e(r)$ under the conditions $\bar S(0)=\bar S(R)=0$ is given by
\begin{align}
	\bar S(r) = \frac{1}{2(1-\epsilon)}\Bigg(r^{-1+\epsilon}\int_0^r s^{-\epsilon}e + r^{1-\epsilon}\int_r^R s^{-2+\epsilon}e-\Big(\frac{r}{R}\Big)^{1-\epsilon}R^{-1+\epsilon}\int_0^R s^{-\epsilon}e\Bigg)\, .
\end{align}

From this we have the explicit computation
\begin{align}
\int_0^R \frac{\bar S(r)}{r} = 	\frac{1}{2(1-\epsilon)}\Bigg(\int_0^R r^{-2+\epsilon}\int_0^r s^{-\epsilon}e + \int_0^R r^{-\epsilon}\int_r^R s^{-2+\epsilon}e-\int_0^R r^{-\epsilon}R^{-2+2\epsilon}\int_0^R s^{-\epsilon}e\Bigg)\, .
\end{align}

Estimating each of these terms is similar, so let us just focus on the first.  Indeed, by changing the order of integration we arrive at
\begin{align}
	\int_0^R\int_0^r r^{-2+\epsilon} s^{-\epsilon}e &= \int_0^R\int_s^R r^{-2+\epsilon} s^{-\epsilon}e = \frac{-1}{1-\epsilon}\int_0^R \big(R^{-1+\epsilon}-s^{-1+\epsilon}\big)s^{-\epsilon} e \notag\\
	&= \frac{1}{1-\epsilon}\Big(\int_0^R s^{-1} e - R^{-1+\epsilon}\int_0^R s^{-\epsilon} e\Big)\, .
\end{align}
 Arguing in a verbatim manner with the other terms leads to the conclusion of the lemma.

\end{proof}

Our main corollary of the above is the following, which gives a Dini estimate for our $L^1$ Hessian:

\begin{corollary}\label{c:dini_estimate}
	Let $A$ be a stationary Yang-Mills connection on a $\delta$-annular region $\cA = B_2\setminus \overline B_{r_x}(\cC)$ satisfying $\fint_{B_2} |F_A|^2 \leq \Lambda$, and let $V$ be a $\delta$-gauge on $B_4$. Let us define the scale invariant quantity $S(r) = r\cdot r^{-3}\int_{b=r} |\nabla V|\,\varphi_\cA\, |\nabla b|$.  Then for each $\epsilon>0$ if $\delta<\delta(n,k,\Lambda,\epsilon)$ we have the estimate
\begin{align}
	\int_0^\infty \frac{S(r)}{r} < \epsilon\, .		
\end{align}
\end{corollary}
\begin{proof}
	Recall from Proposition \ref{p:super_convexity} that $S(r)$ solves the differential inequality
\begin{align}
r\frac{d}{dr}\Big(r\frac{d}{dr} S\Big)\geq (1-\epsilon)^2 S - e(r)\, ,	
\end{align}
where we have that
\begin{align}
e(r)\leq C(n)\,\mu\Big(\{x\in\cC\cap B_{19/10}: C(n)^{-1}r\leq r_x\leq C(n)r\}\Big) + C(n) r\, .	
\end{align}
Note also that by the construction of the cutoff $\phi_\cA$ that $S(0)=S(R)=0$ where $R\leq R(n)$.  Let us then consider the solution of
\begin{align}
	&r\frac{d}{dr}\Big(r\frac{d}{dr} \bar S\Big)= (1-\epsilon)^2 \bar S - e(r)\, ,\notag\\
	&\bar S(0)=\bar S(R)=0\, .
\end{align}
Note by a maximum principle applied to $\bar S-S$ we immediately yield the inequality $S(r)\leq \bar S(r)$.  On the other hand, let us note that
\begin{align}
\int_0^R \frac{e}{s}	 \leq C(n)\mu\{B_{18/10}\} + C(n) \leq C(n)\, ,
\end{align}
where in the last inequality we have used Theorem \ref{t:annular_region}.  Thus by applying Proposition \ref{p:dini_ode} with $\epsilon'>0$ we have the estimate
\begin{align}
\int_0^\infty \frac{S}{r} \leq \int_0^R \frac{\bar S}{r} \leq C(n)\epsilon'<\epsilon\, ,	
\end{align}
where in the last line we have chosen $\epsilon'\equiv C(n)^{-1}\epsilon$.  This completes the proof of the corollary.
\end{proof}

\subsection{\texorpdfstring{Proof of the $L^1$ gradient estimate of Theorem \ref{t:eps_gauge_gradient}}{Proof of the L1 gradient estimate of Theorem \ref{t:eps_gauge_gradient}}}

Using the coarea formula we may write
\begin{align}
	\int_{\cA\cap B_{3/2}} d_\cC^{-3}|\nabla V| \leq \int d_\cC^{-3}|\nabla V| \phi_\cA = \int_0^R\int_{b=r} d_\cC^{-3}|\nabla V| \phi_\cA\,|\nabla b|^{-1}\, ,
\end{align}
where $b$ is the $\mu$-Green's distance associated to $\cA$ and $R\leq R(n)$.  Using Lemma \ref{l:annular_distance} we can estimate this by
\begin{align}
\int_{\cA\cap B_{3/2}} d_\cC^{-3}|\nabla V| \leq C(n)  \int_0^R r^{-3}\int_{b=r} |\nabla V| \phi_\cA\,|\nabla b|	= C(n)\int_0^R \frac{S(r)}{r}\, .
\end{align}
Using Corollary \ref{c:dini_estimate} with $\epsilon'\equiv C(n)^{-1}\epsilon$ we then arrive at the desired estimate 
\begin{align}
\int_{\cA\cap B_{3/2}} d_\cC^{-3}|\nabla V| \leq C(n)\int_0^R \frac{S(r)}{r} < \epsilon\, ,
\end{align}
as claimed.  The other estimates then follow by combining this with Lemma  \ref{l:harmonic_reg_eps_reg}. $\square$

\subsection{Transformation Estimates}

Recall from Theorem \ref{t:eps_gauge_transformation} that if $x\in \cC$ then for every $r_x\leq r <1$ there exists a matrix $T=T(x,r)$, given explicitly in Definition \ref{d:transformation}, such that $T\circ V$ is still an $\epsilon$-gauge.  The key to Theorem \ref{t:eps_gauge_on} is to see that for most $x\in \cC$ and for all $r_x<r\leq 1$, this matrix is in fact close to the identity.  In this subsection we see how to use the gradient estimate of Theorem \ref{t:eps_gauge_gradient} in order to control the transformation matrix $T$, which will be used in the next section is order finish the proof of Theorem \ref{t:eps_gauge_on}.\\

The main technical result of this subsection is the following:

\begin{proposition}\label{p:transformation_estimate}
Given a $\delta$-annular region $\cA=B_2\setminus \overline B_{r_x}(\cC)$ and a $\delta$-gauge $V^a\in \Gamma(B_4,E)$, then for $x\in \cC$ and $r_x\leq r<1$ with $\int_{W^{\pi/4}(x)} r_A^{1-n} \,|\nabla V|<\epsilon_n$ we have for every $\epsilon>0$ that if $\delta<\delta(n,k,\Lambda,\epsilon)$ then
\begin{align}
  \big|Id- T(x,r)\big| < \epsilon+C(n) \int_{W^{\pi/4}(x)} r_A^{1-n} \,|\nabla V|\, .
\end{align}
\end{proposition}
\vspace{.5 cm}

In order to prove the above Proposition let us begin with the following Lemma, which is a more local version of the above:\\

\begin{lemma}\label{l:transformation_estimate}
Given a $\delta$-annular region $\cA=B_2\setminus \overline B_{r_x}(\cC)$ and a $\delta$-gauge $V^a\in \Gamma(B_4,E)$, then for $x\in \cC$ and $r_x\leq r<1$ we have for $\gamma\equiv \frac{4}{5}$ that
\begin{align}
  \big|Id- T(x,\gamma\cdot r)T(x,r)^{-1} \big| < C(n) r\fint_{W_r^{\pi/4}(x)} |\nabla \big(T_{x,r}\circ V\big)|
\end{align}
\end{lemma}
\begin{proof}
Let us denote $\tilde V^a \equiv T_{x,r}\circ V^a$, and so we are trying to estimate $\big|Id- \tilde T(x,\gamma\, r) \big|$ where $\tilde T(x,\gamma\cdot r)\equiv \Bigg(\fint_{W_{\gamma\cdot r}^\theta(x)} \langle \tilde V^a, \tilde V^b\rangle\Bigg)^{-1/2}$ with $\theta\equiv 10^{-2}$.  Using the definition of $\tilde V^a$ it is equivalent for us to estimate %Since $\tilde V^a$ is an $\epsilon$-gauge we know that $|\tilde T^{ab}-\delta^{ab}|<C(n)\epsilon$, and thus it is equivalent for us to estimate
\begin{align}\label{e:transformation_est:1}
\big| Id - \fint_{W_{\gamma\cdot r}^\theta(x)} \langle \tilde V^a, \tilde V^b\rangle\big| = \big| \fint_{W_{r}^\theta(x)} \langle \tilde V^a, \tilde V^b\rangle - \fint_{W_{\gamma\cdot r}^\theta(x)} \langle \tilde V^a, \tilde V^b\rangle\big| \, .	
\end{align}
To accomplish this let us observe the relations
\begin{align}
&B_{10^{-2} r}\Big(W^\theta_r(x)\Big), B_{10^{-2} r}\Big( W^\theta_{\gamma r}(x)\Big) \subseteq W^{\pi/4}_r(x)\, .
%&B_{10^{-2} r}\Big(W^\theta_r(x)\Big)\cap B_{10^{-2} r}\Big( W^\theta_{\gamma r}(x)\Big) = \emptyset\, .
\end{align}
We will use a Poincar\'e to then conclude the result.  Indeed, let us estimate
\begin{align}\label{e:transformation_est:2}
	\Big| \fint_{W_{r}^{\theta}(x)} \langle \tilde V^a, \tilde V^b\rangle &- \fint_{W_{r}^{\pi/4}(x)} \langle \tilde V^a, \tilde V^b\rangle\Big|\leq \fint_{W_{r}^\theta(x)}\big| \langle \tilde V^a, \tilde V^b\rangle - \fint_{W_{r}^{\pi/4}(x)} \langle \tilde V^a, \tilde V^b\rangle\big|\, ,\notag\\
	&\leq \frac{\Vol(W_r^{\pi/4}(x))}{\Vol(W_r^{\theta}(x))}\fint_{W_{r}^{\pi/4}(x)}\big| \langle \tilde V^a, \tilde V^b\rangle - \fint_{W_{r}^{\pi/4}(x)} \langle \tilde V^a, \tilde V^b\rangle\big|\, ,\notag\\
	&\leq C(n)\fint_{W_{r}^{\pi/4}(x)}\big| \langle \tilde V^a, \tilde V^b\rangle - \fint_{W_{r}^{\pi/4}(x)} \langle \tilde V^a, \tilde V^b\rangle\big|\, ,\notag\\
	&\leq C(n) r \fint_{W_{r}^{\pi/4}(x)}\big| \nabla \langle\tilde V^a, \tilde V^b\rangle\big|\, ,\notag\\
	&\leq C(n) r \fint_{W_{r}^{\pi/4}(x)}\big| \nabla \tilde V \big|\, ,
\end{align}
where in the last line we have used the $L^\infty$ bounds on $\tilde V$.  A verbatim computation also gives 
\begin{align}\label{e:transformation_est:3}
	\Big| \fint_{W_{\gamma r}^{\theta}(x)} \langle \tilde V^a, \tilde V^b\rangle &- \fint_{W_{r}^{\pi/4}(x)} \langle \tilde V^a, \tilde V^b\rangle\Big|\leq C(n) r \fint_{W_{r}^{\pi/4}(x)}\big| \nabla \tilde V \big|\, .
\end{align}
Combining \eqref{e:transformation_est:2} with \eqref{e:transformation_est:3} we are able to estimate \eqref{e:transformation_est:1} by
\begin{align}
	\big| Id &- \fint_{W_{\gamma\cdot r}^\theta(x)} \langle \tilde V^a, \tilde V^b\rangle\big| = \big| \fint_{W_{r}^\theta(x)} \langle \tilde V^a, \tilde V^b\rangle - \fint_{W_{\gamma\cdot r}^\theta(x)} \langle \tilde V^a, \tilde V^b\rangle\big|\, ,\notag\\
	&\leq \Big| \fint_{W_{r}^{\theta}(x)} \langle \tilde V^a, \tilde V^b\rangle - \fint_{W_{r}^{\pi/4}(x)} \langle \tilde V^a, \tilde V^b\rangle\Big|+ \Big| \fint_{W_{\gamma r}^{\theta}(x)} \langle \tilde V^a, \tilde V^b\rangle - \fint_{W_{r}^{\pi/4}(x)} \langle \tilde V^a, \tilde V^b\rangle\Big|\, ,\notag\\
	&\leq C(n) r \fint_{W_{r}^{\pi/4}(x)}\big| \nabla \tilde V \big|\, ,
\end{align}
which finishes the proof of the Lemma.
\end{proof}

With the Lemma in hand we now finish the proof of Proposition \ref{p:transformation_estimate}:

\begin{proof}[Proof of Proposition \ref{p:transformation_estimate}]
	
Let $x\in \cC$ be such that $\int_{W^{\pi/4}(x)} r_A^{1-n} \,|\nabla V|<\epsilon_n$ holds, where $\epsilon_n$ will be chosen shortly.  Let us consider the sequence of scales $s_\alpha\equiv \gamma^\alpha$ together with the associated matrices $T_\alpha\equiv T(x,s_\alpha)$.  We will only prove the result for the $T_\alpha$ with $s_\alpha\geq r_x$.  Using that the $T_\alpha\circ V^a$ are $\epsilon$-gauges on $B_{s_\alpha}(x)$ the result easily extends to all radii $r\geq r_x$.  Note also that since $V^a$ is a $\delta$-gauge we have for $\delta<\delta(n,k,\epsilon)$ that $|Id-T_0|<\epsilon$.
	
Now let us remark on the following.  If $\alpha\geq 0$ is such that $\abs{T_\alpha}\leq 2$, then by applying Lemma \ref{l:transformation_estimate} we have the estimate
\begin{align}
\big|T_{\alpha} - T_{\alpha+1}\big| &\leq C(n) s_\alpha \fint_{W_{s_\alpha}^{\pi/4}(x)}\big| \nabla V \big|\notag\\
&\leq C(n)\int_{W^{\pi/4}(x)\cap A_{2^{-1}s_\alpha,2 s_\alpha}} r_A^{1-n}\big| \nabla V \big|\, .
\end{align}
Iterating on this we see that if $T_\beta\leq 2$ for all $\beta\leq \alpha$ then we get the estimate
\begin{align}\label{e:p_transformation:1}
\big|T_{\alpha+1} - T_{0}\big| &\leq C(n)\sum_{\beta\leq \alpha}\int_{W^{\pi/4}(x)\cap A_{2^{-1}s_\beta,2 s_\beta}} r_A^{1-n}\big| \nabla V \big|\, ,\notag\\
&\leq C(n)\int_{W^{\pi/4}(x)\cap A_{2^{-1}s_\alpha,2}} r_A^{1-n}\big| \nabla V \big|\, ,\notag\\
&\leq C(n)\epsilon_n < 10^{-2}\, ,
\end{align}
where in the last line we have chosen $\epsilon_n$ sufficiently small in a manner which depends only on dimension.  

Combining this with the estimate $|T_0-Id|<\epsilon_n$ we conclude that if $s_\alpha\geq r_x$ is such that $\abs{T_\beta}\leq 2$ for all $\beta\leq \alpha$, then in fact if $s_{\alpha+1}\geq r_x$ then $\abs{T_\beta}\leq 2$ for all $\beta\leq \alpha+1$.  Therefore we have the estimate
\begin{align}
\abs { T_\alpha } \leq 2 \text{ for all }s_\alpha\geq r_x\, .	
\end{align}
In particular, we then we get that \eqref{e:p_transformation:1} holds for all $s_\alpha\geq r_x$, from which we get the estimate
\begin{align}
\big|T_{\alpha} - Id\big| &\leq |T_0-Id| + C(n)\int_{W^{\pi/2}(x)} r_A^{1-n}\big| \nabla V \big|\, ,\notag\\
&\leq \epsilon + C(n)\int_{W^{\pi/2}(x)} r_A^{1-n}\big| \nabla V \big|\, ,
\end{align}
which finishes the proof of the Proposition.
\end{proof}
\vspace{.5cm}

\subsection{Proof of Theorem \ref{t:eps_gauge_on}}

We finish the proof of Theorem \ref{t:eps_gauge_on} in this subsection.  Let us begin by observing with $r_\alpha\equiv 2^{-\alpha}$ that if $\supp{f}\subseteq \cA_{10^{-6}}$ then we have the following:
\begin{align}\label{e:eps_gauge_on:fatou}
\int\Big(\int_{W^{\pi/4}_x} d(x,\cC)^{1-n}&|f|\Big)\,d\mu[x] \leq \int\Big(\sum_{r_\alpha\geq r_x}\int_{W^{\pi/4}_x\cap A_{\frac{1}{2}r_\alpha,2r_{\alpha}}(\cC)} d(x,\cC)^{1-n}|f|\Big)\,d\mu[x]	\notag\\
&\leq C(n)\int\Big(\sum_{r_\alpha\geq r_x} r_\alpha^{1-n}\int_{W^{\pi/4}_x\cap A_{\frac{1}{2}r_\alpha,2r_{\alpha}}(\cC)} |f|\Big)\,d\mu[x]\, ,\notag\\
&\leq C(n)\sum_{r_\alpha}r_\alpha^{1-n}\mu\big(\{x:\frac{1}{4}r_\alpha<r_x<4r_\alpha\}\big)\int_{A_{\frac{1}{2}r_\alpha,2r_{\alpha}}(\cC)}|f|\, ,\notag\\
&\leq C(n)\sum_{r_\alpha}r_\alpha^{-3}\int_{A_{\frac{1}{2}r_\alpha,2r_{\alpha}}(\cC)}|f|\, ,\notag\\
&\leq \int_{B_2} d(x,\cC)^{-3}|f|\, .
\end{align}

With this in hand can finish Theorem \ref{t:eps_gauge_on}.  Indeed, by applying \eqref{e:eps_gauge_on:fatou} to $f=|\nabla V|\chi_{\cA\cap B_{3/2}}$ and by using Theorem \ref{t:eps_gauge_gradient} with $(\epsilon')^2$ we obtain the estimate
\begin{align}
	\int_{B_1}\Big(\int_{W^{\pi/4}_x} d(x,\cC)^{1-n}|\nabla V|\Big)\,d\mu[x] \leq C(n)\int_{\cA\cap B_{3/2}} r_A^{-3}|\nabla V|\leq C(n)(\epsilon')^2\, .
\end{align}
In particular, let us consider the set $\cC'\equiv\{x\in B_1\cap \cC:\int_{W^{\pi/4}_x} d(x,\cC)^{1-n}|\nabla V|<\epsilon'\}$, then we see that $\mu(B_1\setminus \cC')\leq C(n)\epsilon'$.  On the other hand, by applying Proposition \ref{p:transformation_estimate} with $\epsilon'$ we see that for $\epsilon'<c(n,k)\epsilon$ that $\cC'\subseteq \cC_\epsilon$, and thus we have finished the proof of the Theorem. $\square$

\vspace{.5cm}

\section{Proof of Curvature Estimates of Theorem \ref{t:annular_region}}\label{s:annular_curv_proof}

In this section we finish the proof of the $L^1$ hessian and $L^2$ curvature estimates on annular regions.  The strategy of the proof is to use the scale invariant gradient estimates on $V$ from Theorem \ref{t:eps_gauge_gradient} to show the estimates on $\cA$ wherever $V$ remains close to being an orthogonal basis.  By Theorem \ref{t:eps_gauge_on} this is everything except a set of small $n-4$ content, and therefore we can recover the rest and start the estimate over in a manner for which the inductive errors give rise to a geometric series.  More slowly, let us start with the following, which is the main tool in our inductive construction:

\begin{lemma}\label{l:annular_region}
Let $A$ be a stationary Yang-Mills connection on a $\delta$-annular region $\cA = B_2\setminus \overline B_{r_x}(\cC)$ satisfying $\fint_{B_2} |F_A|^2 \leq \Lambda$, and let $V$ be a $\delta$-gauge on $B_4$.  For each $\epsilon>0$ if $\delta<\delta(n,k,\Lambda,\epsilon)$ then there exists a collection $\{B_{r_j}(x_j)\}$ with $x_j\in \cC$, $r_j>r_{x_j}$ and $\tilde\cA = \cA\setminus \bigcup B_{r_j}(x_j)$ such that
\begin{enumerate}
\item $B_{r_i/10}(x_i)\cap B_{r_j/10}(x_j)=\emptyset$.
\item $\sum r_j^{n-4}<\epsilon$.
\item $\int_{\tilde\cA\cap B_1} |\nabla^2 F|$,  $\int_{\tilde\cA\cap B_1} |F|^2<\epsilon$
\end{enumerate}
\end{lemma}
\begin{proof}
	For each $x\in \cC$ and $\epsilon'>0$, which will be fixed later, let us define the radius 
\begin{align}
\bar r_x \equiv \min\{r_x\leq \bar r<2:|\langle V^a,V^b\rangle-\delta^{ab}|<\epsilon' \text{ in }W^{\pi/4}(x)\cap A_{\bar r,2}(x)\}\, .
\end{align}
With this let us consider the set
\begin{align}
\bar\cC\equiv\{x\in \cC\cap B_1:\bar r_x>r_x\}\, .	
\end{align}
Note from Theorem \ref{t:eps_gauge_on} that for $\delta<\delta(n,k,\Lambda,\epsilon')$ that $\mu(\bar \cC)<\epsilon'$.  Let us start with the following claim:\\

{\bf Claim: }We have the estimate $\mu(\bar \cC\cap B_{\bar r_x/10}(x))>c(n)\bar r_x ^{n-4}$.\\

Indeed, let us consider two cases.  If $\bar r_x<2 r_x$ then this is clear simply because $\mu(\bar \cC\cap B_{\bar r_x/10}(x))\geq \mu(x)>c(n) r_x^{n-4}\geq c(n) \bar r_x^{n-4}$.

On the other hand, if $\bar r_x\geq 2 r_x$, then let $y\in W^{\pi/4}(x)\cap \partial B_{\bar r_x}(x)$ be such that $|\langle V^a,V^b\rangle-\delta^{ab}|=\epsilon'$.  Let $x'\in \cC \cap \bar B_{\bar r_x/20}(x)$ be one of the points of $\cC \cap \bar B_{\bar r_x/20}(x)$ that is closest to $y$.  Note then for every $z\in B_{\bar r_x/100}(x')$ that $y\in \cW^{\pi/4}(z)$ and with $d(y,z)>r_z$.  This last inequality holds using $(a4)$ because $d(z,y)>\frac{3}{4}d(x,y)\geq \frac{3}{2}r_x>r_z$.  In particular, we see that $B_{\bar r_x/100}(x')\subseteq \bar \cC$ and thus $\mu(\bar \cC\cap B_{\bar r_x/10}(x))\geq \mu(B_{\bar r_x/100}(x'))>c(n) \bar r_x^{n-4}$ by the Ahlfor's regularity of Theorem \ref{t:annular_region}.  This finishes the proof of the Claim. $\square$ \\

Now let us choose a Vitali subcovering $\{B_{\bar r_j}(x_j)\}$ with $x_j\in \bar \cC$.  Then using the above claim we can estimate
\begin{align}
\sum \bar r_j^{n-4} \leq C(n)\sum \mu(\bar \cC\cap B_{\bar r_j/10}(x_j))\leq C(n)\mu(\bar\cC)\leq C(n)\epsilon'<\epsilon\, ,
\end{align}
where in the last inequality we have assumed $\epsilon'<C(n)^{-1}\epsilon$.  \\

We claim now that our desired curvature estimates hold on the set $\tilde\cA = \cA\setminus \bigcup B_{r_j}(x_j)$.  To prove this, note that, by Theorem \ref{t:eps_gauge_gradient}, we have the estimates
\begin{align}
\int_{\cA\cap B_1} |\nabla^4 V|+	\int_{\cA\cap B_1} |\nabla^2 V|^2+	\int_{\cA\cap B_1} r_A^{-3}|\nabla V|+	\int_{\cA\cap B_1} r_A^{-2}|\nabla^2 V| < \epsilon' \, .
\end{align}
Note that by the definition of curvature we have the estimates
\begin{align}
&|F(V)| \leq 2|\nabla^2 V|\, ,\notag\\
&|\nabla^2 F(V)|\leq 2|\nabla^4 V|+2|\nabla F|\,|\nabla V|+|F|\,|\nabla^2 V|\, .	
\end{align}
In particular, if we are at a point $y\in \cA$ such that $|\langle V^a,V^b\rangle-\delta^{ab}|<\epsilon'$ almost form an orthonormal basis, then we have the estimates
\begin{align}
&|F|(y) \leq 3\sum_a |\nabla^2 V_a|(y)\, ,\notag\\
&|\nabla^2 F|(y)\leq 3\sum_a \ton{|\nabla^4 V_a|(y)+r_A^{-3}|\nabla V_a|+ r_A^{-2}|\nabla^2 V_a|}\, ,	
\end{align}
where we have used the scale invariant estimates $r_A^2|F|$ ,$r_A^3|\nabla F| <<1$ in the above.  This estimate holds for each $y\in \bar \cA$, and therefore we have that
\begin{align}
&\int_{\bar\cA\cap B_1} |\nabla^2 F|\leq 3\sum_a\ton{\int_{\cA\cap B_1}|\nabla^4 V_a|+\int_{\cA\cap B_1}r_A^{-3}|\nabla V_a|+ \int_{\cA\cap B_1}r_A^{-2}|\nabla^2 V_a|} <5\epsilon'\, ,\notag\\	
&\int_{\bar\cA\cap B_1} |F|^2\leq 3\sum_a \ton{\int_{\cA\cap B_1}|\nabla^2 V_a|^2} <3\epsilon'\, ,
\end{align}
which for $\epsilon'<\frac{1}{10}\epsilon$ finishes the proof of the Lemma.
\end{proof}
\vspace{.25cm}

With the above in hand let us now finish the proof of Theorem \ref{t:annular_region}:

\begin{proof}[Proof of Theorem \ref{t:annular_region}]
	The proof of the curvature estimates of Theorem \ref{t:annular_region} is now just an inductive application of Lemma \ref{l:annular_region}.  Indeed, let us apply Lemma \ref{l:annular_region} with $\epsilon'=\epsilon'(n,\epsilon)$ which will be chosen later.  Then we can write
\begin{align}
	\cA\cap B_1 \subseteq \cA_0\cup \bigcup B_{r_{j_0}}(x_{j_0})\, ,		
\end{align}
 such that 
 \begin{align}
 	&\int_{\cA_0\cap B_1} |\nabla^2 F|\; ,\;\;\int_{\cA_0\cap B_1} |F|^2\, ,\;\sum r_{j_0}^{n-4}<\epsilon'\, .
 \end{align}
Now note that the restriction of $\cA$ to $B_{2r^0_j}(x^0_j)$ is a $\delta$-annular region for each $j$, and therefore we may again apply Lemma \ref{l:annular_region} to each of the balls $B_{2r^0_j}(x^0_j)$ to obtain
\begin{align}
	\cA\cap B_{r_{j_0}}(x_{j_0}) \subseteq \cA_{j_0,1}\cup \bigcup B_{r_{j_0,j_1}}(x_{j_0,j_1})\, ,		
\end{align}
such that
\begin{align}
 &\int_{\cA_{j_0,1}\cap B_{r_{j_0}}} |\nabla^2 F|\; ,\;\;\int_{\cA_{j_0,1}\cap B_{r_{j_0}}} |F|^2\, ,\;\sum_{j_1} r_{j_0,j_1}^{n-4}<\epsilon'\,r_{j_0}^{n-4}\, .	
\end{align}
Thus by defining 
\begin{align}
&\cA_1\equiv \cA_0\cup \bigcup_{j_0} \cA_{j_0,1}\, ,\notag\\
&\{B_{r_{j_1}}(x_{j_1})\} \equiv \bigcup_{j_0}\{B_{r_{j_0,j_1}}(x_{j_0,j_1})\, ,
\end{align}
we have the covering
\begin{align}
\cA\cap B_1 \subseteq \cA_1\cup \bigcup B_{r_{j_1}}(x_{j_1})
\end{align}
such that 
\begin{align}
	&\int_{\cA_1\cap B_1} |\nabla^2 F|\leq \epsilon'+\epsilon'\sum r_{j_0}^{n-4}\leq \epsilon'+(\epsilon')^2\, ,\notag\\
	&\int_{\cA_1\cap B_1} |F|^2\leq \epsilon'+(\epsilon')^2\, ,\notag\\
	&\sum r_{j_1}^{n-4}\leq \epsilon'\sum r_{j_0}^{n-4}\leq (\epsilon')^2\, .
\end{align}
Now we can observe that each $\cA\cap B_{2r_{j_1}}$ is a $\delta$-annular region and repeat this process again.  Again, if we repeat this process $I$ times then we arrive at a covering 
\begin{align}
\cA\cap B_1 \subseteq \cA_I\cup \bigcup B_{r_{j_I}}(x_{j_I})
\end{align}
such that 
\begin{align}
	&\int_{\cA_I\cap B_1} |\nabla^2 F|\; ,\;\;\int_{\cA_I\cap B_1} |F|^2\;\;\leq \sum_{j=1}^I(\epsilon')^j\; ,\;\;\sum r_{j_I}^{n-4}\leq (\epsilon')^I\, .
\end{align}
In particular, we clearly get the much weaker estimate that $r_{j_I}\to 0$ as $I\to \infty$, and thus we have that the $\cA_I$ form an exhaustion of $\cA$.  Passing to the limit we then get the estimate
\begin{align}
	&\int_{\cA\cap B_1} |\nabla^2 F|\; ,\;\;\int_{\cA\cap B_1} |F|^2\;\;\leq \sum_{j=1}^\infty(\epsilon')^j\leq 2\epsilon' \, .
\end{align}
Letting $\epsilon'<\frac{1}{2}\epsilon$ this concludes the proof of our Theorem.
\end{proof}
\vspace{.5cm}

\section{Quantitative Bubble Tree and Quantitative Energy Identity}\label{s:quant_bubbletree}

In this section we discuss our quantitative bubble tree and quantitative energy identity theorems.  As the names suggest, these are both quantitative and higher dimensional versions of the more classical notions from dimension four.  Both the classical bubble tree and energy identity's are traditionally discussed for limiting sequences of Yang Mills connections, however the quantitative versions hold on a fixed Yang Mills connection, and so from this point of view the results are new even for dimension four.  We will see that the quantitative nature of the estimates are crucial for the applications.\\

Let us begin with our quantitative bubble tree decomposition:

\begin{theorem}[Quantitative Bubble Tree]\label{t:quantitative_bubbletree}
Let $A$ be a stationary connection with $\fint_{B_4} |F_A|^2 \leq \Lambda$.  If $B_{4}(p)$ is $\delta'$-weakly flat wrt $\cL$ for $\delta'<\delta'(n,k,\Lambda,\delta,\epsilon)$, then we have
	\begin{align}
		&B_1(p)\subseteq \bigcup_a \big(\cA_a\cap B_{r_a}\big) \cup \bigcup_b \cB_b\cup \bigcup_c B_{r_c}(x_c)\, ,
	\end{align}
	such that
	\begin{enumerate}
		\item[(a)] $\cA_a\subseteq B_{2r_a}(x_a)$ are $\delta$-annular regions with respect to $\cL$.
		\item[(b)] $\cB_b\subseteq B_{\delta^{-1}r_b}(x_b)$ are $\delta$-bubble regions with respect to $\cL$.
		\item[(c)] $\sum_a r_a^{n-4} + \sum_b r_b^{n-4}\leq C(n,k,\Lambda)$ and $\sum_c r_c^{n-4}<\epsilon$.
	\end{enumerate}
\end{theorem}
\vspace{.5cm}

To state the quantitative energy identity will pick points $q\in \cL$ on the plane of symmetry and consider the slice $\cL^\perp_q$.  One can view the first part of the next result as a sliced version of the quantitative bubble tree theorem.  The final part of the next theorem contains the real content of the quantitative energy identity, and tells one how to compute the energy at a point by summing energies of bubbles.\\

\begin{theorem}[Quantitative Energy Identity]\label{t:quantitative_eq}
Under the conditions and decomposition of Theorem \ref{t:quantitative_bubbletree} if $\delta<\delta(n,k,\Lambda,\epsilon)$ and $\delta'<\delta'(n,k,\Lambda,\delta)$ then $\exists$ $\cG_\epsilon\subseteq \cL\cap B_1$ with $\Vol\big(B_1\setminus\cG_\epsilon\big)<\epsilon$ such that for each $q\in\cG_\epsilon$ the covering
\begin{align}
		\cL^\perp_q\cap B_1(q)&\subseteq \cA_q\cup\cB_q = \bigcup_a \cA_{q,a} \cup \bigcup_b \cB_{q,b} \notag\\
		&=\bigcup_a \Big(\cL^\perp_q \cap \cA_a\cap B_{r_a}\Big)\cup \bigcup_b \Big(\cL^\perp_q \cap \cB_b\cap B_{r_b}\Big)\, ,
\end{align}
satisfies
\begin{enumerate}
\item[(a)] $\int_{\cA_{q}}|F_A|^2<\epsilon$.
\item[(b)] $\big|\int_{B_1(q)} |F_A|^2 - \omega_{n-4}\int_{\cB_q} |F_A|^2\big|<\epsilon$.
\item[(s)] $\#\,\big\{\cA_{q,a}\big\}$, $\;\#\,\big\{\cB_{q,b}\big\}<N(n,k,\Lambda)$.
\end{enumerate}
\end{theorem}

\vspace{.5cm}

\subsection{Proof of the Quantitative Bubble Tree Decomposition and the Quantitative Energy Identity}

In this subsection we will prove both Theorem \ref{t:quantitative_bubbletree} and Theorem \ref{t:quantitative_eq} simultaneously.  Let us first pick $\epsilon'>0$ which will be fixed later.  Now since the ball $B_4(p)$ is $\delta'$-weakly flat we have the estimate
\begin{align}
	\fint_{B_4} |F_A[\cL]|^2 \leq \delta'\, ,
\end{align}
where we will choose $\delta'$ sufficiently small later on in the proof.  Now for every point $x\in B_2$ let us define the radius
\begin{align}
m_x\equiv \inf\{0<r<1:\; r^{4-n}\int_{B_r(x)}|F[\cL]|^2 \leq \delta''\}\, ,	
\end{align}
where $\delta''=\delta''(n,k,\Lambda,\delta)$ is chosen so that Theorem \ref{t:bubble_existence}, Theorem \ref{t:annular_existence} and Theorem \ref{t:annular_region} hold with $\epsilon'>0$.  Additionally, let us consider the set 
\begin{align}
\cM \equiv \{x\in B_2 : m_x>0\}\, .
\end{align}
Note for each ball $B_{m_x}(x)$ with $m_x>0$ we have the estimate
\begin{align}
\int_{B_{m_x}(x)} |F[\cL]|^2 = \delta'' m_x^{n-4}\, .	
\end{align}

Now let us consider the covering $\{B_{10m_x}(x)\}_{x\in\cM}$ of $\cM$, and then we can take a Vitali subcovering
\begin{align}
\cM\subseteq \bigcup_c B_{r_{c,0}}(x_{c,0})\equiv \bigcup_s B_{r_{s}}(x_{s})\, ,	
\end{align}
where $r_{s} = 10 m_{x_{s}}$ and $\{B_{r_{s}/10}(x_{s})\}$ are disjoint.  We give this collection two names because the remaining $c$-balls which will built later in this construction will be done by different means with a different purpose, so that although in the end we will group them altogether we keep them distinct for now for intuition simplicity.  Note now that we have the estimate
\begin{align}
\sum_s r_{s}^{n-4} &=10^{n-4} \sum_s \Big(\frac{r_{s}}{10}\Big)^{n-4}	 = 10^{n-4}\delta''^{-1}\sum_s \int_{B_{r_{s}/10}(x_{s})} |F[\cL]|^2\notag\\
&\leq 10^{n-4}\delta''^{-1}\int_{B_3} |F[\cL]|^2 = 10^{n-4}(\delta'')^{-1} \delta' < \delta\, ,
\end{align}
where in the last line we have chosen $\delta'<\delta'(n,k,\Lambda,\delta)$.\\

In particular, we now know that for every $x\in B_2\setminus \bigcup_s B_{r_{s}}(x_{s})$ and all $0<r<2$ that we have the estimate
\begin{align}
r^{4-n}\int_{B_r(x)}|F[\cL]|^2 < \delta''\, .	
\end{align}
Intuitively, using Theorem \ref{t:bubble_existence} and Theorem \ref{t:annular_existence} this is telling us that every ball not strictly contained in $\bigcup_s B_{r_{s}}(x_{s})$ is either a bubble region or an annular region.  Finishing the proof of Theorem \ref{t:quantitative_bubbletree} is about making this intuition precise and keeping track of the estimates along the way.\\

To make this precise we will build a sequence of coverings.  If $E\equiv \sup_{B_1(p)}\theta_1 = \overline \theta_1(p)$, then the $i^{th}$ covering of the sequence looks like
\begin{align}\label{e:quant_bubbletree:inductive_covering}
B_1(p)\setminus \bigcup_c B_{r_{c,i}}(x_{c,i})\subseteq \bigcup_{a}\big(\cA_{a,i}\cap B_{r_{a,i}}(x_{a,i})\big)\cup\bigcup_{b} \cB_{b,i}\cup\bigcup_{w} B_{r_{w,i}}(x_{w,i})\, ,
\end{align}
and will satisfy
\begin{enumerate}
\item[(a)] $\cA_{a,i}\subseteq B_{r_{a,i}}(x_{a,i})$ are $\delta$-annular with $\int_{\cL^\perp_q\cap \cA_i} |F_A|^2<C(n,k,\Lambda,i)\epsilon'$ for $q\in \cL\cap B_1(p)\setminus \bigcup_c B_{r_{c,i}}(x_{c,i})$.
\item[(b)] $\cB_{b,i}\subseteq B_{\delta^{-1} r_{b,i}}(x_{b,i})$ are $\delta$-bubble regions.
\item[(c)] For $q\in\cL$ with $\cL^\perp_q\cap \bigcup_c B_{r_{c,i}}(x_{c,i})=\emptyset$ if $\cB_q\equiv \cL^\perp_q\cap \bigcup_{b} \cB_{b,i}$ then there exists $\{x'_{w,i}\}\subseteq \{x_{w,i}\}$ with $B_{r'_{w,i}}(x'_{w,i})\cap \cL^\perp_q\neq \emptyset$ such that $\big|\theta_1(q) - \omega_{n-4}\int_{\cB_q}|F_A|^2 - \sum \theta_{r_{w,i}}(x'_{w,i})\big|<C(n,k,\Lambda,i)\epsilon'$.
\item[(w)] $B_{r}(x_{w,i})$ are $\delta$-weakly flat wrt $\cL$ for $4^{-1}\delta r_{w,i}\leq r\leq 4r_{w,i}$ and satisfy $\bar\theta_{r_{w,i}}(x_{w,i})\leq E-\epsilon_{n,k}\cdot i$.
\item[(s)] $\sum_a r_{a,i}^{n-4}+\sum_b r_{b,i}^{n-4}+\sum_w r_{w,i}^{n-4}<C(n,k,\Lambda,i)$ and $\sum_c r_{c,i}^{n-4}\leq C(n,k,\Lambda,i)\epsilon'$.
\end{enumerate}

Let us begin by observing that if we can build this sequence of coverings then we have proved the Theorems.  Indeed, for $i>\Lambda\, \epsilon_{n,k}^{-1}=i(n,k,\Lambda)$ we have that there are no $w$-balls in the covering, as their energy would necessarily be negative.  However if $\epsilon'<c(n,k,\Lambda)\epsilon$ it is then clear from $(a)\to(s)$ that the covering in \eqref{e:quant_bubbletree:inductive_covering} would then satisfy the conditions of Theorems \ref{t:quantitative_bubbletree} and \ref{t:quantitative_eq}. \\

Our inductive covering will itself follow from a series of covering constructions.  Let us begin with the following claim, which is the first step in the process:\\

{\bf Claim:} Let $B_{s}(x)$ be $\delta$-weakly flat wrt $\cL$ for $4^{-1}\delta r\leq s\leq 4r$, then we have the decomposition\newline $B_r(x)\setminus \bigcup_s B_{r_{s}}(x_{s})\subseteq \cA \cup \bigcup_b \cB_b \cup \bigcup_c B_{r_c}(x_c) \cup \bigcup_w B_{r_w}(x_w)$ such that 
\begin{enumerate}
\item[(a)] $\cA\subseteq B_{2r}(x)$ is a $\delta$-annular region with $\int_{\cL^\perp_q\cap \cA} |F_A|^2<\epsilon'$ for $q\in \cL\cap B_r(x)\setminus \bigcup_c B_{r_c}(x_c)$.
\item[(b)] $\cB_b\subseteq B_{\delta^{-1}r_{b}}(x_{b})$ are $\delta$-bubble regions.
\item[(c)] For $q\in\cL$ with $\cL^\perp_q\cap \Big(\bigcup_c B_{r_{c}}(x_{c})\cup \bigcup_s B_{r_s}(x_s)\Big)=\emptyset$ and $\cB_q\equiv \cL^\perp_q\cap \bigcup_{b} \cB_{b}$, there exists $\{x'_{w}\}\subseteq \{x_w\}$ with $B_{r'_w}(x'_w)\cap \cL^\perp_q\neq \emptyset$ such that $\big|\theta_r(q) - \omega_{n-4}\int_{\cB_q}|F_A|^2 - \theta_{r_{w}}(x'_{w})\big|<\epsilon'$.
\item[(w)] $B_{r}(x_{w})$ is $\delta$-weakly flat wrt $\cL$ for $4^{-1}\delta r_w\leq r\leq 4r_w$ and satisfies $\bar\theta_{r_{w}}(x_{w})\leq \bar\theta_r(x)-\epsilon_{n,k}$.
%\item[(e)] $B_{4r_{e}}(x_{e})$ are $10^{-n}\delta^2$-weakly flat wrt $\cL$ and $\sum r_e^{n-4}<\delta\,r^{n-4}$.
\item[(s)] $\sum_b r_{b}^{n-4}+\sum_w r_{w}^{n-4}<C(n,k,\Lambda) r^{n-4}$ and $\sum_c r_c^{n-4}<\epsilon' r^{n-4}$.
\end{enumerate}
\vspace{.25cm}

To prove the claim will require an application of Theorem \ref{t:annular_existence} and Theorem \ref{t:bubble_existence}.  Indeed, let us first apply Theorem \ref{t:annular_existence} in order to build an annular region $\cA= B_{r}(x)\setminus \overline B_{r_x}(\cC)$.  Using the estimates on $\cC^c$ we can cover
\begin{align}
B_{r}(x)\setminus \bigcup_s B_{r_{s}}(x_s)\subseteq \cA \cup \bigcup_{b'} B_{r_{b'}}(x_{b'}) \cup \bigcup_c B_{r_{c}}(x_{c})\, , 	
\end{align}
where $B_{r_{b'}}(x_{b'})$ are $\delta$ weakly flat but $B_s(x_{b'})$ is not $\delta$-weakly flat for some $\delta^4 r_{b'}\leq s\leq r_{b'}$.  We also have the estimates $\sum r_{b'}^{n-4}<C(n,k,\Lambda) r^{n-4}$ and $\sum r_{c}^{n-4}<\epsilon' r^{n-4}$.  To each $b'$-ball we apply Theorem \ref{t:bubble_existence} in order to cover
\begin{align}
B_{r_{b'}} \subseteq \cB_{b}\cup \bigcup_w B_{r_{b,w}}(x_{b,w})\, ,	
\end{align}
where with $r_{b}\equiv \delta r_{b'}$ we have that $\cB_{b}\subseteq B_{\delta^{-1}r_{b}}(x_{b})$ is a $\delta$-bubble region, and $B_{s}(x_{b,w})$ are $\delta$-weakly flat for $\delta r_{b,w}<s<r_{b,w}$ with $\bar \theta_{r_{b,w}}(x_{b,w}) < \bar \theta_{\delta^{-1} r_{b}}(x_{b})-\epsilon_{n,k}<\bar\theta_r(x)-\epsilon_{n,k}$ and $\sum_w r_{b,w}^{n-4}\leq C(n,k,\Lambda) r_{b}^{n-4}$.  By taking a union over all $b'$ we obtain the covering
\begin{align}
B_{r}(x)\setminus \bigcup_s B_{r_s}(x_s)\subseteq \cA \cup \bigcup_{b} \cB_b \cup \bigcup_w B_{r_{w}}(x_{w}) \cup \bigcup_e B_{r_{e}}(x_{e})\, ,
\end{align}
which satisfy the estimates $\sum r_{b}^{n-4}+\sum_w r_{w}^{n-4}\leq C(n,k,\Lambda) r^{n-4}$ and $\sum r_{c}^{n-4}<\epsilon' r^{n-4}$.  Condition $(a)$ of the claim follows from the definition of $\cC^c$ in Theorem \ref{t:annular_existence}, while condition $(c)$ follows by combining this with $(2)$ of Theorem \ref{t:bubble_existence}, and thus we have finished the proof of the Claim. $\square$ \\

With the Claim in hand, we are ready to finish the construction of the inductive covering of \eqref{e:quant_bubbletree:inductive_covering}, which will itself finish the proof.  Note first that for $i=0$ we may take the trivial covering $B_1(p)=B_{r_{w,0}}(x_{w,0})$, where $x_{w,0}\equiv p$ and $r_{w,0}\equiv 1$.  Thus we have the base step of the inductive construction.  Now having built the covering at stage $i$, let us see how to build the covering at stage $i+1$.  More precisely, observe that for each ball $B_{r_{w,i}}(x_{w,i})$ we may apply the Claim.  If we do this to each $w$-ball then we arrive at a new covering

\begin{align}
B_1(p)\setminus \bigcup_s B_{r_s}(x_s)\subseteq \bigcup_{a}\big(\cA_{a,i+1}\cap B_{r_{a,i}}(x_{a,i+1})\big)\cup\bigcup_{b} \cB_{b,i+1}\cup \bigcup_{c} B_{r_{c,i+1}}(x_{c,i+1})\cup\bigcup_{w} B_{r_{w,i+1}}(x_{w,i+1})\, ,
\end{align}
where $\bar \theta_{r_{w,i+1}}(x_{w,i+1})<\bar \theta_{r_{w,i}}(x_{w,i})-\epsilon_{n,k}<E-\epsilon_{n,k}(i+1)$ and we have the estimates
\begin{align}
&\sum_a r_{a,i+1}^{n-4}+\sum_b r_{b,i+1}^{n-4}+\sum_w r_{w,i+1}^{n-4}\leq C(n,k,\Lambda,i)+C(n,k,\Lambda)\sum_w r_{w,i}^{n-4}\leq C(n,k,\Lambda,i+1)\, ,\notag\\
&\sum_c r_{c,i+1}^{n-4}\leq C(n,k,\Lambda,i)\epsilon'+\epsilon'\sum_w r_{w,i}^{n-4}\leq C(n,k,\Lambda,i+1)\epsilon'\, .
\end{align}
Conditions $(a)$, $(b)$, $(w)$, and now $(s)$ are all clearly satisfied by the new covering, while $(c)$ follows from the $(i)$-inductive hypothesis of $(c)$ combined with $(c)$ of the Claim.  This finishes the proof of the inductive covering, and hence of Theorem \ref{t:quantitative_bubbletree} and Theorem \ref{t:quantitative_eq}. $\square$
\vspace{.5cm}

\section{Proof of Energy Identity}\label{s:energy_identity}

In this section we use the quantitative energy identity from Theorem \ref{t:quantitative_eq} in order to finish the proof of the energy identity itself.  Thus recall our setup that $A_i\to A$ with $|F_{A_i}|^2dv_{g_i}\to |F_A|^2dv_g + \nu$, where $\nu=e(x)\lambda^{n-4}_S$ is the $n-4$ rectifiable defect measure.  For each $\epsilon,\delta>0$ let us consider the subset $E_{\epsilon,\delta}\subseteq \text{supp}[\nu]$ defined by $x\in E_{\epsilon,\delta}$ if there exists a $n-4$ plane $\cL_x$ through $x$ and points $q_i\to x$ such that the following hold:
\begin{enumerate}
\item $\exists$ $\delta$-bubbles $\cB_{q_i,\mathbf{b}}\subseteq B_{\delta^{-1}r_{q_i,b}}(x_{q_i,b})$ with $x_{q_i,b}\in \cL^\perp_{q_i}$
\item  $\#\{\cB_{q_i,\mathbf b}\}\leq N(n,\Lambda)$.
\item If $\cB_{q_i}\equiv \bigcup \cL^\perp_{q_i}\cap \cB_{q_i,\mathbf b}$ then $\big|\int_{\cB_{q_i}}|F_{A_i}|^2 - e(x)\big|<\epsilon$.
\end{enumerate}

Our goal is to show that for each $\epsilon>0$ there exists $\delta_\epsilon>0$ such that $E_{\epsilon,\delta}$ is a set of full measure in $\text{supp}[\nu]$.  Let us first observe that if this is the case then we have finished the proof of the energy identity itself.  Indeed, consider the collection of points $E_0\equiv \bigcap_{\epsilon>0}E_{\epsilon,\delta_\epsilon}$, which is itself a set of full measure.  Then for such a point $x\in E_0$ we can find a sequence $q_i\to x$ with $\big|\int_{\cB_{q_i}}|F_A|^2 - e(x)\big|\to 0$.  Since $\#\{\cB_{q_i,b}\}\leq N$ we can apply Theorem \ref{t:bubble_region}.3 in order to find points $c_{i,b}\in \cB_{q_i}$ and radii $s_{i,b}>0$ such that for all $\eta>0$ and $R\geq R(n,k,\Lambda,\eta)$ we have
\begin{align}
	\big|\int_{\cB_{q_i}}|F_{A_i}|^2 - \int_{\cB_{q_i}\cap \bigcup B_{Rs_{i,b}}(x_{i,b})}|F_{A_i}|^2\big|<\eta\, ,
\end{align}
so that
\begin{align}
	\big|e(x) - \int_{\cB_{q_i}\cap \bigcup B_{Rs_{i,b}}(x_{i,b})}|F_{A_i}|^2\big|<\eta\, .
\end{align}
Indeed, the pointed limits $s_{i,b}^{-1}A_i\to B_b\in \cB[x]$ are then bubbles at $x\in E_0$ such that the energy identity
\begin{align}
e(x) = \sum_b \int_{\cL^\perp} |F_{B_b}|^2\, ,
\end{align}
holds.  Since $E_0$ is a set of full measure this finishes the proof of the Theorem.\\

Thus we need to show that each $E_{\epsilon,\delta}$ is a set of full measure.  To accomplish this let us remark that by the definition of $\nu$ being rectifiable we have that a.e. $x\in \text{supp}\{\nu\}$ is such that the unique tangent measure at $x$ is $e(x) \lambda^{n-4}_{\cL}$, for some $n-4$ plane $\cL$.  In particular, for any $\delta'>0$ and all $0<r<r(x,\delta')$ fixed sufficiently small that with $i$ sufficiently large we have
\begin{align}\label{e:energy_identity:1}
&\big|r^{4-n}\int_{B_r(x)}|F_{A_i}|^2 - \omega_{n-4}\,e(x)\big|<\delta'\, ,	\notag\\
&\text{$B_{4r}(x)$ is $\delta'$-weakly flat}.
\end{align}

Let us choose $\delta'>0$ sufficiently small that we can apply Theorem \ref{t:quantitative_eq} with $\frac{1}{2}\epsilon$ for some $\delta\leq \delta(n,k,\Lambda,\epsilon)$.  Let us consider the sets $\cG_{\epsilon,i}\subseteq \cL\cap B_r(x)$ coming from the Theorem and consider their Hausdorff limit $\cG_{\epsilon,i}\to \cG_{\epsilon}\subseteq \cL\cap B_r(x)$.  Using \eqref{e:energy_identity:1} above and the content of Theorem \ref{t:quantitative_eq} it is immediate that $\cG_\epsilon \subseteq E_{\epsilon,\delta}$.  Since $0<r<r(x,\delta')$ was arbitrary, we have in particular that $x$ is a $\epsilon$-density point of $E_{\epsilon,\delta}$.  Since $x\in \text{supp}[\nu]$ was a set of full measure, we have that $E_{\epsilon,\delta}\subseteq \text{supp}[\nu]$ is a set of full measure as well.  This completes the proof. $\square$\\

\section{Annulus/Bubble Decomposition}\label{s:bubble_decomp}

In this section we introduce one last quantitative covering result, which we call the annulus/bubble decomposition.  The decomposition will split a ball into two pieces, one of which are bubble regions with uniformly bounded curvature, and the other are annulus regions.  The most important aspect of this decomposition for us will be the effective $n-4$ content estimates that come with the covering, which will be used quite crucially in the proof of our global estimates on the hessian.  Our main result in this subsection is the following:\\

\begin{theorem}[Annulus/Bubble Decomposition]\label{t:bubble_decomposition}
Let $A$ be a stationary Yang-Mills connection satisfying \eqref{e:YM_assumptions2} and $\fint_{B_2} |F_A|^2 \leq \Lambda$.   Then for each $\delta>0$ we can write
	\begin{align}
		&B_1(p)\subseteq \bigcup_a \big(\cA_a\cap B_{r_a}\big) \cup \bigcup_b  \cB_b \, ,
	\end{align}
	such that
	\begin{enumerate}
		\item[(a)] $\cA_a\subseteq B_{2r_a}(x_a)$ are $\delta$-annular regions.
		\item[(b)] $\cB_b\subseteq B_{\delta^{-1}r_b}(x_b)$ are $\delta$-bubble regions.
		\item[(s)] $\sum_a r_a^{n-4} + \sum_b r_b^{n-4}\leq C(n,k,K,\Lambda,\delta)$.
	\end{enumerate}
\end{theorem}
\begin{remark}
If the connection $A$ is singular in the sense of \cite{TaoTian_YM} or \cite{RiPe2} then one must allow an extra piece to this decomposition, so that $B_1(p)\subseteq \bigcup_a \big(\cA_a\cap B_{r_a}\big) \cup \bigcup_b  \cB_b\cup \cS$, where $\cS$ is a closed set with $n-4$ measure zero.  We shall point out in the proof where this happens.
\end{remark}

\vspace{.25cm}

Note several major differences between the annulus/bubble and the quantitative bubble tree decomposition.  On the positive side, the original ball $B_1$ does not need to be weakly flat, and	 there are no $c$-balls in the covering.  Therefore every point lies in either an annular region or a bubble region.  On the negative side, the annular and bubble regions are not with respect to some fixed $\cL$.  This decomposition will be the more appropriate one in order to prove the $L^1$ hessian estimate, while as we have seen the quantitative bubble tree decomposition is the appropriate one to prove the energy identity.

\subsection{Proof of Theorem \ref{t:bubble_decomposition}}

The proof of Theorem \ref{t:bubble_decomposition} is by recursively applying the quantitative bubble tree decomposition of Theorem \ref{t:quantitative_bubbletree} with the weakly flat decomposition of Theorem \ref{t:weakly_flat_decomp}.  Let us begin with the following claim:\\

{\bf Claim 1:} Let $A$ be a stationary Yang-Mills connection satisfying \eqref{e:YM_assumptions2} and $\fint_{B_4} |F_A|^2 \leq \Lambda$.  Then for each $\delta>0$ we can write
	\begin{align}
		&B_1(p)\subseteq \bigcup_a \big(\cA_a\cap B_{r_a}\big) \cup \bigcup_b  \cB_b \cup \bigcup_d B_{r_d}(x_d) \, ,
	\end{align}
	such that
	\begin{enumerate}
		\item[(a)] $\cA_a\subseteq B_{2r_a}(x_a)$ are $\delta$-annular regions.
		\item[(b)] $\cB_b\subseteq B_{\delta^{-1}r_b}(x_b)$ are $\delta$-bubble regions.
		%\item[(c)] $\sum r_c^{n-4}<\delta$.
		\item[(d)] $B_{r_{d}}(x_{d})$ satisfy $\bar\theta_{r_{d}}(x_{d})\leq \bar\theta_1(x)-\eta(n,k,K,\Lambda,\delta)$.
		\item[(s)] $\sum_a r_a^{n-4} + \sum_b r_b^{n-4}+\sum_d r_d^{n-4}\leq C(n,k,\Lambda,\delta)$.
	\end{enumerate}
\vspace{.25cm}

To prove the Claim let us pick $\delta'(n,k,\Lambda,\delta)$ so that the quantitative bubble tree decomposition of Theorem \ref{t:quantitative_bubbletree} holds with $\epsilon= 10^{-6n}\omega_n$.  We will produce a sequence of coverings of the form
	\begin{align}
		&B_1(p)\subseteq \bigcup_a \big(\cA_{a,i}\cap B_{r_{b,i}}\big) \cup \bigcup_b  \cB_{b,i} \cup \bigcup_c B_{r_{c,i}}(x_{c,i}) \cup \bigcup_d B_{r_{d,i}}(x_{d,i}) \, ,
	\end{align}
such that in addition to $(a)$, $(b)$, and $(d)$ holding we will also have the estimates
\begin{enumerate}
\item[(c-i)] $\sum_c r_{c,i}^{n-4}\leq \epsilon^i$ .
\item[(s-i)] $\sum_a r_{a,i}^{n-4} + \sum_b r_{b,i}^{n-4}+\sum_d r_{d,i}^{n-4}\leq C(n,k,\Lambda,\delta)\sum_{j=0}^{i-1}\epsilon^j$ .
\end{enumerate}

Let us first observe that if we can build this sequence of coverings then we have proved the Claim.  Indeed, this sequence differs from our desired covering in the claim only by the existence of $c$-balls.  However, by the content estimate we have that $\sup_c r_{c,i}\to 0$ as $i\to \infty$.  In particular, for $i$ sufficiently (but uncontrollably) large we have that any $c$-ball must define a smooth bubble region\footnote{This is not true if $A$ is not smooth.  However, the estimate $(c-i)$ allows one to hausdorff limit the $c$-balls $\{B_{r_{c,i}}(x_{c,i})\}$ to a closed $n-4$ measure zero set $\cS$, which is the additional piece of the covering in the singular case.}, and therefore we can take the $c$-balls to be empty in the covering for $i$ sufficiently large, which produces the covering of our claim.

Thus let us now focus on building this sequence.  Note that for $i=0$ we can let the covering be defined by the single ball $B_1(p)\equiv B_{r_c}(x_c)$.  Therefore our goal will be to produce the covering at the $i+1$ stage given that we have produced the covering at stage $i$.

To accomplish this let us focus on each $c$ ball $B_{r_{c,i}}(x_{c,i})$ in the covering.  Let us now consider two options, either the ball $B_{4r_{c,i}}$ is $\delta'$-weakly flat or it is not.  In the first case we may apply the quantitative bubble tree decomposition of Theorem \ref{t:quantitative_bubbletree} in order to produce the covering:
\begin{align}
B_{r_{c,i}}(x_{c,i})\subseteq \bigcup_a \big(\cA_{c,i,a'}\cap B_{r_{c,i,a'}}\big) \cup \bigcup_b  \cB_{c,i,b'} \cup \bigcup_c B_{r_{c,i,c'}}(x_{c,i,c'}) \, .
\end{align}
which satisfy the estimates $\sum_{c'} r_{c,i,c'}^{n-4}\leq \epsilon r_{c,i}^{n-4}$ and $\sum_{a'} r_{c,i,a'}^{n-4} + \sum_{b'} r_{c,i,b'}^{n-4}\leq C(n,k,\Lambda,\delta)r_{c,i}^{n-4}$.  On the other hand, if $B_{4r_{c,i}}$ is not $\delta'$-weakly flat then we may apply Theorem \ref{t:weakly_flat_decomp} in order to write
\begin{align}
B_{r_{c,i}}(x_{c,i})\subseteq \bigcup_c B_{r_{c,i,c'}}(x_{c,i,c'}) \cup \bigcup_d B_{r_{c,i,d'}}(x_{c,i,d'}) \, .	
\end{align}
with the estimates $\sum_{c'} r_{c,i,c'}^{n-4}\leq \epsilon r_{c,i}^{n-4}$ and $\sum_{d'} r_{c,i,d'}^{n-4}\leq C(n,k,\Lambda,\delta)r_{c,i}^{n-4}$.  If we take the union of all these coverings of every $c$-ball, and collect different pieces together, we arrive at the $i+1$ covering 
\begin{align}
		&B_1(p)\subseteq \bigcup_a \big(\cA_{a,i+1}\cap B_{r_{b,i+1}}\big) \cup \bigcup_b  \cB_{b,i+1} \cup \bigcup_c B_{r_{c,i}}(x_{c,i+1}) \cup \bigcup_d B_{r_{d,i+1}}(x_{d,i+1}) \, ,
\end{align}
with the desired estimates
\begin{align}
&\sum_c r_{c,i+1}^{n-4} \leq \epsilon \sum_c r_{c,i}^{n-4}\leq \epsilon^{i+1}\, ,\notag\\
&\sum_a r_{a,i+1}^{n-4} + \sum_b r_{b,i+1}^{n-4}+\sum_d r_{d,i+1}^{n-4}\leq \sum_a r_{a,i}^{n-4} + \sum_b r_{b,i}^{n-4}+\sum_d r_{d,i}^{n-4} + C(n,k,K,\Lambda)\sum_c r_{c,i}^{n-4}\notag\\
&\leq  C(n,k,\Lambda,\delta)\sum_{j=0}^{i-1}\epsilon^j	+C \epsilon^i = C\sum_{j=0}^{i}\epsilon^j\, ,
\end{align}
which therefore finishes the inductive step of the construction and hence the Claim. $\square$ \\

With the Claim in hand let us now return to finish the proof of the Theorem.  Indeed, we will see that the proof of the Theorem is just a repeated application of the Claim.  More precisely, let us produce a sequence of coverings 
	\begin{align}
		&B_1(p)\subseteq \bigcup_a \big(\cA_{a,i}\cap B_{r_{b,i}}\big) \cup \bigcup_b  \cB_{b,i} \cup \bigcup_d B_{r_{d,i}}(x_{d,i}) \, ,
	\end{align}
which in addition to satisfying $(a)$ and $(b)$ will satisfy the conditions:
\begin{enumerate}
	\item[(d-i)] $B_{r_{d}}(x_{d})$ satisfy $\bar\theta_{r_{d}}(x_{d})\leq \bar\theta_1(x)-\eta(n,k,K,\Lambda,\delta)\, i$.
	\item[(s-i)] $\sum_a r_{a,i}^{n-4} + \sum_b r_{b,i}^{n-4}+\sum_d r_{d,i}^{n-4}\leq C(n,k,\Lambda,\delta,i)$.
\end{enumerate}

Let us first observe that once we have proved the existence of this sequence of coverings then we are done.  In fact, for $i>\eta^{-1}\Lambda$ there cannot be any $d$-balls in the covering as any such ball would have negative energy.  Therefore for such an $i$ we have produced the desired covering of the Theorem.\\

Thus let us concentrate on proving the existence of this sequence of coverings.  To produce the covering for $i=1$ we simply apply the Claim. We now construct the covering at the $i+1$ stage inductively, using that we have already built the covering at the $i^{th}$ stage.  To accomplish this let us focus on each $d$-ball $B_{r_{d,i}}(x_{d,i})$ in the covering and apply the Claim to this ball.  Taking a union produces the $i+1$ covering.  Hence, we have finished the proof of Theorem \ref{t:bubble_decomposition}.

\vspace{.5cm}

\section{Proof of Theorem \ref{t:main_L1_hessian}}

In this section we put together the annulus/bubble decomposition of Theorem \ref{t:bubble_decomposition} with the annulus structure of Theorem \ref{t:annular_region} and the bubble structure of Theorem \ref{t:bubble_region} in order to complete the proof of the $L^1$ hessian estimate on the curvature.  Indeed, let us choose $\delta(n,k,\Lambda)$ sufficiently small so that Theorem \ref{t:annular_region} holds for $\epsilon_n=10^{-n}$.  Then we can estimate our $L^1$ norm of the hessian by
\begin{align}\label{e:L1_hess:1}
\int_{B_1(p)}|\nabla^2 F| \leq \sum_a \int_{\cA_a\cap B_{r_a}}|\nabla^2 F|+\sum_b \int_{\cB_b}|\nabla^2 F|\, .	
\end{align}

On the $\delta$-bubble regions we have by Theorem \ref{t:bubble_region} that
\begin{align}
\int_{\cB_b}|\nabla^2 F|\leq C(n,k,\Lambda,\delta)\,r_b^{n-4}\, .
\end{align}

On the other hand, by Theorem \ref{t:annular_region} we have on the annular regions $\cA_a$ the scale invariant estimate
\begin{align}
	\int_{\cA_a\cap B_{r_a}}|\nabla^2 F|\leq \epsilon_n\,r_a^{n-4}\, .
\end{align}
Putting these together with \eqref{e:L1_hess:1} we obtain
\begin{align}
	\int_{B_1(p)}|\nabla^2 F| \leq \epsilon_n\sum_a r_a^{n-4}+C\delta\sum_b r_b^{n-4}\leq C(n,k,\Lambda)\Big(\sum_a r_a^{n-4}+\sum_b r_b^{n-4}\Big)\leq C(n,k,\Lambda)\, ,
\end{align}
where the last inequality is due to the $n-4$ content estimate of Theorem \ref{t:bubble_decomposition}.  This completes the proof.\footnote{If the connection $A$ is singular in the sense of \cite{TaoTian_YM} or \cite{RiPe2} one must also show $\nabla^2 f$ is measurable as a distribution.  This may be done using that the singular set has $n-4$ measure zero together with the monotonicity formula.} $\square$\\

\bibliographystyle{aomalpha}
\bibliography{nv_ym}

\def\cprime{$'$}
\providecommand{\bysame}{\leavevmode\hbox to3em{\hrulefill}\thinspace}
\providecommand{\noopsort}[1]{}
\providecommand{\mr}[1]{\href{http://www.ams.org/mathscinet-getitem?mr=#1}{MR~#1}}
\providecommand{\zbl}[1]{\href{http://www.zentralblatt-math.org/zmath/en/search/?q=an:#1}{Zbl~#1}}
\providecommand{\jfm}[1]{\href{http://www.emis.de/cgi-bin/JFM-item?#1}{JFM~#1}}
\providecommand{\arxiv}[1]{\href{http://www.arxiv.org/abs/#1}{arXiv~#1}}
\providecommand{\doi}[1]{\url{http://dx.doi.org/#1}}
\providecommand{\MR}{\relax\ifhmode\unskip\space\fi MR }
\providecommand{\MRhref}[2]{%
  \href{http://www.ams.org/mathscinet-getitem?mr=#1}{#2}
}
\providecommand{\href}[2]{#2}
\begin{thebibliography}{Don89}

\bibitem[CN13]{ChNa2}
\bgroup\scshape{}J.~Cheeger\egroup{} and \bgroup\scshape{}A.~Naber\egroup{},
  Quantitative stratification and the regularity of harmonic maps and minimal
  currents,  \emph{Comm. Pure Appl. Math.} \textbf{66} (2013), 965--990.
  \mr{3043387}.  \zbl{1269.53063}.  \doi{10.1002/cpa.21446}.  Available at
  \url{http://arxiv.org/abs/1107.3097}.

\bibitem[CN15]{ChNa_Codim4}
\bgroup\scshape{}J.~Cheeger\egroup{} and \bgroup\scshape{}A.~Naber\egroup{},
  Regularity of {E}instein manifolds and the codimension 4 conjecture,
  \emph{Ann. of Math. (2)} \textbf{182} (2015), 1093--1165. \mr{3418535}.
  \doi{10.4007/annals.2015.182.3.5}.  Available at
  \url{http://dx.doi.org/10.4007/annals.2015.182.3.5}.

\bibitem[Don89]{Donaldson_Compactification}
\bgroup\scshape{}S.~K. Donaldson\egroup{}, Compactification and completion of
  {Y}ang-{M}ills moduli spaces,  in \emph{Differential geometry ({P}e\~n\'\i
  scola, 1988)}, \emph{Lecture Notes in Math.} \textbf{1410}, Springer, Berlin,
  1989, pp.~145--160. \mr{1034277}.  \doi{10.1007/BFb0086420}.  Available at
  \url{http://dx.doi.org/10.1007/BFb0086420}.

\bibitem[DT98]{DoTh_gauge}
\bgroup\scshape{}S.~K. Donaldson\egroup{} and \bgroup\scshape{}R.~P.
  Thomas\egroup{}, Gauge theory in higher dimensions,  in \emph{The geometric
  universe ({O}xford, 1996)}, Oxford Univ. Press, Oxford, 1998, pp.~31--47.
  \mr{1634503}.

\bibitem[JN]{JiNa}
\bgroup\scshape{}W.~Jiang\egroup{} and \bgroup\scshape{}A.~Naber\egroup{}, L2
  curvature bounds on manifolds with bounded ricci curvature,  \emph{preprint}.
  Available at \url{http://arxiv.org/abs/1605.05583}.

\bibitem[LR02]{LinRiv_HarEQ}
\bgroup\scshape{}F.-H. Lin\egroup{} and
  \bgroup\scshape{}T.~Rivi{\`e}re\egroup{}, Energy quantization for harmonic
  maps,  \emph{Duke Math. J.} \textbf{111} (2002), 177--193. \mr{1876445}.
  \doi{10.1215/S0012-7094-02-11116-8}.  Available at
  \url{http://dx.doi.org/10.1215/S0012-7094-02-11116-8}.

\bibitem[NV]{NaVa+}
\bgroup\scshape{}A.~Naber\egroup{} and \bgroup\scshape{}D.~Valtorta\egroup{},
  Rectifiable-reifenberg and the regularity of stationary and minimizing
  harmonic maps,  \emph{accepted on Annals of Mathematics}. Available at
  \url{http://annals.math.princeton.edu/articles/10010}.

\bibitem[Pri83]{Price_Mon}
\bgroup\scshape{}P.~Price\egroup{}, A monotonicity formula for {Y}ang-{M}ills
  fields,  \emph{Manuscripta Math.} \textbf{43} (1983), 131--166. \mr{707042}.
  \doi{10.1007/BF01165828}.  Available at
  \url{http://dx.doi.org/10.1007/BF01165828}.

\bibitem[Riv02]{Riv_YM}
\bgroup\scshape{}T.~Rivi{\`e}re\egroup{}, Interpolation spaces and energy
  quantization for {Y}ang-{M}ills fields,  \emph{Comm. Anal. Geom.} \textbf{10}
  (2002), 683--708. \mr{1925499}.  \doi{10.4310/CAG.2002.v10.n4.a2}.  Available
  at \url{http://dx.doi.org/10.4310/CAG.2002.v10.n4.a2}.

\bibitem[RPa]{RiPe3}
\bgroup\scshape{}T.~Rivi{\`e}re\egroup{} and
  \bgroup\scshape{}M.~Petrache\egroup{}, Partial regularity of yang-mills
  stationary weak connections in high dimensions,  \emph{preprint}.

\bibitem[RPb]{RiPe}
\bgroup\scshape{}T.~Rivi{\`e}re\egroup{} and
  \bgroup\scshape{}M.~Petrache\egroup{}, The resolution of the yang-mills
  plateau problem in super-critical dimensions,  \emph{preprint}. Available at
  \url{http://cvgmt.sns.it/paper/2177/}.

\bibitem[RPc]{RiPe2}
\bgroup\scshape{}T.~Rivi{\`e}re\egroup{} and
  \bgroup\scshape{}M.~Petrache\egroup{}, The space of weak connections in high
  dimensions,  \emph{preprint}.

\bibitem[TT04]{TaoTian_YM}
\bgroup\scshape{}T.~Tao\egroup{} and \bgroup\scshape{}G.~Tian\egroup{}, A
  singularity removal theorem for {Y}ang-{M}ills fields in higher dimensions,
  \emph{J. Amer. Math. Soc.} \textbf{17} (2004), 557--593. \mr{2053951}.
  \doi{10.1090/S0894-0347-04-00457-6}.  Available at
  \url{http://dx.doi.org/10.1090/S0894-0347-04-00457-6}.

\bibitem[Tia00]{Tian_CalYM}
\bgroup\scshape{}G.~Tian\egroup{}, Gauge theory and calibrated geometry. {I},
  \emph{Ann. of Math. (2)} \textbf{151} (2000), 193--268. \mr{1745014}.
  \doi{10.2307/121116}.  Available at \url{http://dx.doi.org/10.2307/121116}.

\bibitem[Uhl82]{Uhl_Rem}
\bgroup\scshape{}K.~K. Uhlenbeck\egroup{}, Removable singularities in
  {Y}ang-{M}ills fields,  \emph{Comm. Math. Phys.} \textbf{83} (1982), 11--29.
  \mr{648355}.  Available at
  \url{http://projecteuclid.org/euclid.cmp/1103920742}.

\bibitem[Yu]{Wa}
\bgroup\scshape{}W.~Yu\egroup{}, Quantitative stratification of stationary
  yang-mills fields,  \emph{preprint}.

\end{thebibliography}

\end{document}